\documentclass[a4paper,11pt]{article}

\usepackage{authblk}
\usepackage{amsmath}
\usepackage{amsthm}

\usepackage[top=3.5cm, bottom=3.5cm, left=2.5cm, right=2.5cm]{geometry}

\usepackage[utf8]{inputenc}
\usepackage[T1]{fontenc}
\usepackage[french,english]{babel} 
\usepackage{enumerate}
\usepackage{enumitem}
\usepackage{color}
\usepackage[]{pdfpages}

\usepackage{titling}
\usepackage[colorlinks=false,linkcolor=blue]{hyperref}
\usepackage{csquotes}
\usepackage[nottoc, notlof, notlot, numbib]{tocbibind}
\usepackage[style=alphabetic,giveninits,doi=false,isbn=false,url=false,eprint=false,clearlang=true,sorting=nyt,maxbibnames=99]{biblatex}
\usepackage{mathtools}
\usepackage{thmtools}
\usepackage{bbm}
\usepackage{amssymb}
\usepackage{mathrsfs}
\usepackage{float}
\usepackage{stmaryrd}

\usepackage{setspace}
\usepackage{array}
\usepackage{tabularx}
\usepackage{ltablex} 

\DeclareUnicodeCharacter{221E}{$\infty$}
\DeclareUnicodeCharacter{00B4}{\'}
\DeclareUnicodeCharacter{03BB}{$\lambda$}
\DeclareUnicodeCharacter{0393}{$\Gamma$}
\DeclareUnicodeCharacter{211D}{$\mathbb{R}$}
\DeclareUnicodeCharacter{2212}{-}

\newcommand{\R}{\mathbb{R}}

\newcommand{\1}{\mathbbm{1}}
\newcommand{\E}{\mathbb{E}}
\newcommand{\ps}{\mathcal{P}}

\renewcommand{\d}{\mathrm{d}}
\renewcommand{\P}{\mathbb{P}} 
\newcommand{\F}{\mathcal{F}}

\newcommand{\C}{\mathcal{C}}
\renewcommand{\L}{\mathcal{L}}

\newtheorem{theorem}{Theorem}[section]
\newtheorem{proposition}[theorem]{Proposition}
\newtheorem{lemma}[theorem]{Lemma}
\newtheorem{corollary}[theorem]{Corollary}

\theoremstyle{definition}
\newtheorem{definition}[theorem]{Definition}

\theoremstyle{remark}
\newtheorem{rem}[theorem]{Remark}
\newtheorem{example}[theorem]{Example}

\theoremstyle{plain}
\newtheorem{assumption}{Assumption}

\usepackage{xcolor}

\title{\bf  Quasi-continuity method for mean-field diffusions: large deviations and central limit theorem }
\date{}

\author{Louis-Pierre \textsc{Chaintron}}
\affil{
\small{DMA, École normale supérieure, Université PSL, CNRS, 75005 Paris, France.
} \\
\small{CERMICS, École des ponts, 77420 Champs-sur-Marne, France.} \\
\small{Inria, Team M$\sf{\Xi}$DISIM, Inria Saclay, 91128 Palaiseau, France. } \\
\url{louis-pierre.chaintron@ens.fr}
}

\begin{document}

\maketitle

\abstract{A pathwise large deviation principle in the Wasserstein topology and a pathwise central limit theorem are proved for the empirical measure of a mean-field system of interacting diffusions.
The coefficients are path-dependent.
The framework allows for degenerate diffusion matrices, which may depend on the empirical measure, including mean-field kinetic processes.
The main tool is an extension of Tanaka's pathwise construction \cite{tanaka1984limit} to non-constant diffusion matrices. 
This can be seen as a mean-field analogous of Azencott's quasi-continuity method \cite{azencott1980grandes} for the Freidlin-Wentzell theory.
As a by-product, uniform-in-time-step fluctuation and large deviation estimates are proved for a discrete-time version of the mean-field system.
Uniform-in-time-step convergence is also proved for the value function of some mean-field control problems with quadratic cost.
}

\tableofcontents

\medskip
\medskip
\medskip

\noindent\textbf{Acknowledgements.} The author thanks Julien Reygner for fruitful discussions and careful proofreading.

\section{Introduction}

We consider a mean-field system of interacting diffusions in $\R^d$,
\begin{equation} \label{eq:InterPart}
\d X^{i,N}_t = b_t ( X^{i,N}, \pi ( \vec{X}^{N} ) ) \d t + \sigma_t ( X^{i,N}, \pi ( \vec{X}^{N} ) ) \d B^{i}_t, \qquad 1 \leq i \leq N, 
\end{equation} 
starting from i.i.d. initial conditions $X^{i}_0$,
the $B^i := (B^i_t)_{0 \leq t \leq T}$ being i.i.d. Brownian motions in $\R^{d'}$. 
These particles interact through their empirical measure
\[ \pi ( \vec{X}^{N} ) := \frac{1}{N} \sum_{i = 1}^N \delta_{X^{i,N}} \in \ps ( \C ( [0,T], \R^d ) ), \]
using the notation $\vec{X}^N := ( X^{i,N} )_{1 \leq i \leq N}$.
The coefficients $b$ and $\sigma$ are globally Lipschitz functions which are path-dependent in a non-anticipative way, and $\sigma$ is globally bounded (precise conditions are listed in Assumption \ref{ass:coef1} below). 
As $N \rightarrow +\infty$, $\pi ( \vec{X}^{N} )$ converges to the path-law $\L( \overline{X} )$ of the McKean-Vlasov equation
\[ \d \overline{X}_t = b_t ( \overline{X} , \L (\overline{X}) ) \d t + \sigma_t ( \overline{X} , \L (\overline{X}) ) \d B^1_t, \qquad X_0 \sim X^1_0. \]
Such a mean-field limit is often referred to as a \emph{propagation of chaos} result \cite{sznitman1991topics}.
Under suitable assumptions on $(b,\sigma)$, the mean-field limit holds with rate $N^{-1/2}$ for the mean-square convergence, see e.g. \cite[Theorem 1.4]{sznitman1991topics}.
When $\sigma \equiv \mathrm{Id}$, a pathwise construction due to \cite{tanaka1984limit} builds $\overline{X}$ and the $X^{i,N}$ in a unified framework.
In particular, the empirical measure $\pi( \vec{X}^N )$ of the particles is a continuous function of the one of the driving noises $\pi( \vec{B}^N )$.
This powerful approach rephrases the mean-field limit as a mere continuity result.
Furthermore, Tanaka's construction allows for computing the fluctuations of $\pi( \vec{X}^N )$ directly from the ones of $\pi( \vec{B}^N )$. 
From there, the objectives of this article are three-fold:
\begin{itemize}
    \item We extend Tanaka's construction to non-constant $\sigma$ using an Euler scheme for \eqref{eq:InterPart} for which we prove strong consistency estimates.     
    \item We compute the first-order fluctuations of $\pi ( \vec{X}^N )$ by proving a \emph{pathwise central limit theorem} (CLT), i.e. the convergence of $N^{-1/2} [ \pi ( \vec{X}^{N} ) - \L(\overline{X}) ]$ towards a Gaussian field.
    \item We prove a \emph{large deviation principle} (LDP), estimating $\P ( \pi ( \vec{X}^{N} ) \in A )$ at the exponential scale for any measurable $A \subset \ps ( \C ( [0,T], \R^d ) )$, which may not contain $\L(\overline{X})$. 
\end{itemize}
These three points correspond to Theorems \ref{thm:DisMFlimit}-\ref{thm:LDPGood}-\ref{thm:ClT} below.
The CLT quantifies the normal fluctuations of $\pi ( \vec{X}^N )$ in the $N \rightarrow +\infty$ limit, whereas the LDP quantifies large deviations.
The CLT and the LDP are first proved for the Euler approximation of \eqref{eq:InterPart}, before being transferred to the continuous-time system.
The CLT is proved for binary interactions of type $b(x,P) = \int \tilde{b}(x,y) \d P(y)$.
The LDP is proved in the Wasserstein-$p$ topology, for every $p \in [1,2)$.
These CLT and LDP results are not the first ones of their kind, but they are new at the considered level of generality, see Section \ref{subsec:lit} below for a detailed review of existing literature.
The main challenge of our work is to include non-constant $\sigma$ that depend on the measure argument and may be degenerate.
In particular, our results cover kinetic systems like mean-field interacting Langevin dynamics. 
We refer to our extended construction as a \emph{quasi-continuity} result, by analogy with the seminal Freidlin-Wentzell theory \cite{freidlin1998random} and the Azencott quasi-continuity method \cite{azencott1980grandes}.

To motivate this terminology, let us recall the analogy made in the introduction of \cite{dawson1989large}.
To simplify the presentation, we assume that $b_t (x,P) = b (x_t,P_t)$ and $\sigma \equiv \mathrm{Id}$.
From Ito's formula, the empirical measure $\pi ( \vec{X}^N_t )$ of the particles at time $t$ a.s. satisfies
\[ \d \int_{\R^d} \varphi \, \d \pi ( \vec{X}^N_t ) = \bigg[ \int_{\R^d} L_{\pi ( \vec{X}^N_t )} \varphi \, \d \pi ( \vec{X}^N_t ) \bigg] \d t + \frac{1}{N} \sum_{i=1}^N \nabla \varphi ( X^{i,N}_t ) \cdot \d B^{i}_t, \]
for any smooth test function $\varphi : \R^d \rightarrow \R$,
the measure-dependent generator being defined by $L_\pi \varphi (x) := b( x ,\pi) \cdot \nabla \varphi (x) + \tfrac{1}{2} \Delta \varphi (x)$.
The last term on the r.h.s is a martingale whose quadratic variation is of order $N^{-1}$.
This motivates the informal re-writing
\begin{equation} \label{eq:informalDG}
\d \pi ( \vec{X}^N_t ) = L^\star_{\pi ( \vec{X}^N_t )} \pi ( \vec{X}^N_t ) \d t + N^{-1/2} \d M_t,  
\end{equation}
where $( M_t )_{0 \leq t \leq T}$ is intended to be a measure-valued martingale term of order $1$, and $L^\star_\pi$ is the formal $L^2$-adjoint of $L_\pi$.
Equation \eqref{eq:informalDG} is known as the Dean-Kawasaki stochastic partial differential equation, see e.g. \cite{kawasaki1994stochastic,dean1996langevin,RenesseKasawaki,fehrman2022wellposednessdeankawasakinonlineardawsonwatanabe}.
This writing draws an (infinite-dimensional) analogy with the famous Freidlin-Wentzell diffusion
\begin{equation} \label{eq:fw}
\d X^\varepsilon_t = b ( X^\varepsilon_t ) \d t + \sqrt{\varepsilon} \d B_t,
\end{equation} 
where $b$ is now a vector field in $\R^d$ and $B := (B_t)_{0 \leq t \leq T}$ is a Brownian motion.
Let us describe how our approach enables us to push this analogy beyond qualitative statements, at the very level of proofs. 
We first recall some basic notions about the Freidlin-Wentzell theory.

\medskip

\textbf{\emph{Pathwise Freidlin-Wentzell theory.}} 
If $b$ is globally Lipschitz in \eqref{eq:fw}, setting $\C^d := \C([0,T],\R^d)$,
the Cauchy-Lipschitz theory provides a map $X : \C^d \rightarrow \C^d$, such that for every $\omega = (\omega_t)_{0 \leq t \leq T}$, $X (\omega)$ is the solution to the ordinary differential equation (ODE),
\[ \forall t \in [0,T], \quad X_t ( \omega ) = \int_0^t b ( X_s ( \omega ) ) \d s + \omega_t. \]
This allows us to build the diffusion \eqref{eq:fw} by simply setting $X^\varepsilon := X ( \sqrt{\varepsilon} B)$.
As $\varepsilon \rightarrow 0$,
the a.s. pathwise convergence of $X^\varepsilon$ towards $X^0$ is then a simple consequence of the continuity of $X$.
Looking back at \eqref{eq:informalDG}, this convergence corresponds to the $N \rightarrow +\infty$ limit.
Classically, the Freidlin-Wentzell LDP is obtained from the contraction principle, using the continuous map $X$ to push the large deviations of $( \sqrt{\varepsilon} B )_{\varepsilon > 0}$ given by  Schilder's theorem.
The first-order expansion follows by formally writing that 
\[ X^\varepsilon - X^0 = \sqrt{\varepsilon} \, Y + o (\sqrt{\varepsilon}),  \]
where the Gaussian process $Y = ( Y_t )_{0 \leq t \leq T}$ solves the SDE $\d Y_t = D b ( X^0_t ) \cdot Y_t \d t + \d B_t$, differentiating $X$ using the standard theory of flow maps, see e.g. \cite[Chapter 5.5]{friedman1975stochastic} or \cite[Chapter 2]{freidlin1998random}.

\medskip

\textbf{\emph{Azencott's quasi-continuity method.}}
The above construction does not allow for non-constant diffusion matrices $\sigma$ in \eqref{eq:fw} because the stochastic integral is not a continuous function of the driving path.
More precisely, a measurable map $X : \C^d \rightarrow \C^d$ could be defined as above on a set of full Wiener-measure but $X$ would not be continuous, preventing us from computing large deviations by contraction. 
In the seminal work \cite{azencott1980grandes}, a solution is developed which is by now referred to as the quasi-continuity method.
The main idea is to approximate the Brownian path by a curve in the Sobolev space $H^1 ( [0,T], \R^d)$, before defining a continuous contraction map $H^1 \rightarrow H^1$.
The key-result is then an approximation estimate stating that the made error is negligible at the exponential scale in terms of $\varepsilon$.
Thus, the large deviations can still be computed by contraction.
A similar argument can be found in \cite[Lemma 1.4.21]{deuschel2001large}, using an Euler-approximation of the stochastic integral, and showing that the made error is negligible.

An alternative approach relies on the rough-path theory which was first introduced by \cite{lyons1998differential}.
Extending the notion of stochastic process, this theory makes the stochastic integral be a continuous function of the driving noise.
The flow map is then built on a space of rough paths, yielding the large deviations by continuous contraction \cite{ledoux2002large}.
We emphasise that such a result is very strong, because it provides large deviations at the level of rough paths in addition to the process's level.
As a consequence, there is a technical price to pay to use this sophisticated machinery.

\medskip

\textbf{\emph{Tanaka's pathwise construction.}} 
Going back to the setting of \eqref{eq:informalDG}, the key-element of \cite{tanaka1984limit} is the construction of $X^P : \C^d \rightarrow \C^d$, for every $P \in \ps( \C^d )$, such that
\[ \forall \omega \in \C^d, \; \forall t \in [0,T], \qquad X^P_t ( \omega ) = \int_0^t b ( X_s ( \omega ), ( X^P_s )_\# P ) \d s + \omega_t. \]
We emphasise that the above equation is non-local, because the drift term depends on the whole map $X^P$ through the push-forward. 
\cite{tanaka1984limit} then proves that the map $\Psi : P \mapsto X^P_\# P$ is continuous in a suitable topology, before building the particle system in a pathwise way as
\[ X^{i,N} := X^{ \pi ( \vec{B}^N )} ( B^{i} ), \qquad \qquad  \pi ( \vec{X}^N ) = \Psi ( \pi ( \vec{B}^N ) ), \]
the dependence on the empirical measure $\pi ( \vec{B}^N )$ of the driving noises being continuous. 
Looking at \eqref{eq:informalDG}-\eqref{eq:fw}, there is a striking analogy with the pathwise Freidlin-Wentzell theory.
As a by-product, \cite{tanaka1984limit} directly obtains the large deviations by contraction from the Sanov theorem, and the CLT is obtained by performing a first-order expansion of $\Psi$ -- this is reminiscent of the delta method in statistics.
This framework is very convenient for mean-field limits of any kind: extensions are developed in \cite{Coghi2020PathwiseMT} for non-exchangeable systems, general driving noises, systems with common noise, càdlàg processes, reflected processes... 
Similarly to the Freidlin-Wentzell theory, a crucial limitation is the difficulty to handle non-constant diffusion matrices $\sigma$, because of the need for stochastic integration.

Up to our knowledge, there is no existing rough-path approach for extending Tanaka's construction.
However, we mention \cite{bailleul2021propagation} which proves the mean-field limit for interacting rough paths, and \cite{deuschel2018enhanced} for a related LDP. 
This last LDP was proved for $\sigma \equiv \mathrm{Id}$ using the Girsanov transform and a standard large deviation result for Gibbs measures:
this prevents it from going beyond the case of non-degenerate $\sigma$ that do not depend on the measure argument. 

\medskip

\emph{\textbf{Quasi-continuity method for mean-field systems.}}
Motivated by Tanaka's and Azencott's approaches, it is natural to extend Tanaka's construction to handle non-constant $\sigma$.
We first design a pathwise construction for a Euler-Maruyama approximation of \eqref{eq:InterPart} with general driving noises that may not be Brownian (Theorem \ref{thm:DisMFlimit}).
Exchangeability is not required either.
This allows us to prove mean-field limits, central limit theorems and large deviations for discretised systems at a great level of generality.
These novel results are already useful as such, the practitioner being interested in the behaviour of the numerically implemented model, rather than the theoretical one.
When the driving noises are i.i.d. Brownian motions, we extend these results to the continuous setting, by showing:
\begin{itemize}
    \item A uniform-in-time-step fluctuation estimate to recover the CLT (Theorem \ref{thm:ClT}).
    \item A uniform approximation at the exponential scale to recover the LDP (Theorem \ref{thm:LDPGood}).
\end{itemize}
These uniform estimates are quantitative and new up to our knowledge. 
They are of independent interest from a numerical analysis perspective. 
Following the weak convergence approach \cite{dupuis2011weak}, the exponential approximation relies on a stochastic control interpretation of exponential moments.
Thus, our estimates also have an interest from a mean-field control perspective. 
They provide a uniform-in-$N$ discrete-time approximation for mean-field control with quadratic cost.
Detailed statements of our results are given in Section \ref{sec:Results}.

\subsection{Review of related litterature} \label{subsec:lit}

For general results about mean-field limits and propagation of chaos, we refer to the survey articles \cite{LPChaos1,LPChaos2} and references therein.
For the specific case of the path-dependent setting, we refer to e.g. \cite{lacker2018strong,bernou2022path,lacker2023hierarchies}, which cover slightly less general settings. 
There is a lot of references for discretised versions of systems like \eqref{eq:InterPart}, among which \cite{bossy1997stochastic,antonelli2002rate,bao2019approximations,bernou2022path} and references therein.
To the best of our knowledge, our LDP and CLT results in the discrete-time setting are new.

\medskip

\emph{\textbf{Tanaka's pathwise construction}.}
Up to our knowledge, there are surprisingly few instances of Tanaka's pathwise construction in the mean-field literature.
We mention \cite[Section 3.1] {backhoff2020mean} where the LDP is proved in the Wasserstein-$p$ topology ($p \in [1,2)$) in a path-dependent setting with $\sigma \equiv \mathrm{Id}$ and binary interactions of type $b(x,P) = \int_{\C^d} \tilde{b}(x,y) \d P(y)$.
An exhaustive treatment of the $\sigma \equiv \mathrm{Id}$ setting is done in \cite{Coghi2020PathwiseMT} with extensions to many related situations (general noises, jump processes...). 
The path-construction is also used in \cite{LackerMFGLDP} with $\sigma \equiv \mathrm{Id}$ to prove the LDP in the Wasserstein-$1$ topology, before extending it to prove the LDP in a common noise setting.
This last framework is particularly intricate because the rate function no more has compact level sets.

\medskip

\emph{\textbf{Central limit theorem.}}
Untill recently, there has been much interest in proving the CLT for mean-field systems close to \eqref{eq:InterPart}.
The first related work seems to be \cite{McKean1975FluctuationsIT}, computing fluctuations for interacting particles in the two-state space.
A seminal work is then \cite{braun1977vlasov} for a deterministic Vlasov system, which inspired Tanaka's CLT \cite{tanaka1984limit} on path space (with $\sigma \equiv \mathrm{Id}$).
The fundamental works \cite{sznitman1984nonlinear,sznitman1985fluctuation} prove the CLT on path space for binary interactions, starting from a Girsanov formula to compute symmetric statistics and multiple Wiener integrals.
Many following works re-used this approach.
An alternative method for large deviations and CLT for Gibbs measures is \cite{bolthausen1986laplace}, which relies on embeddings in suitable Banach spaces.
This method was further developed and applied to diffusion systems in \cite{arous1990methode,pra1996mckean}.
A by now classical approach \cite{fernandez1997hilbertian,meleard1998convergence,jourdain1998propagation} computes the fluctuations of the random curve $t \mapsto \pi ( \vec{X}^N_t )$ in weighted Sobolev spaces. 
Some tightness estimates are first proved, before identifying the limit as an infinite-dimensional Ornstein-Uhlenbeck process.
Such results are weaker than the pathwise CLT, but this approach can take benefit of PDE structures to get sharper estimates.
A fairly general CLT for $t \mapsto \pi ( \vec{X}^N_t )$ is thus proved in \cite{jourdain2021central}, by careful analysis of a measure-valued flow.
The very recent work \cite{bernou2024uniform} uses analogous methods to prove (among many other results) a uniform-in-time and quantitative version of this CLT.
In the setting of mean-field games with common noise, we also mention the CLT result \cite{delarue2019master}.
To the best of our knowledge, our pathwise CLT for a system as general as \eqref{eq:InterPart} is new.

\medskip

\emph{\textbf{Large deviations.}}
There is an abundant literature for large deviations of the empirical measure of systems like \eqref{eq:InterPart}.
The seminal work \cite{tanaka1984limit} covers the case $\sigma \equiv \mathrm{Id}$ with binary interactions $b(x,\pi) = \int_{\C^d} \tilde{b}(x,y) \d \pi(y)$.
When $\tilde{b}(x,y) = -\nabla V (x-y)$, the Girsanov transform computes the density of the law of $\vec{X}^N$ w.r.t. to the law of the corresponding system of i.i.d. drift-less particles.
The gradient structure further removes the stochastic integral term from the density given by Girsanov's theorem.
The resulting density being a Gibbs measure, large deviations are then a standard result.
This method is applied in e.g \cite{arous1990methode,del1998large}, and in \cite{pra1996mckean} in a more general setting including environment noises.
For an extension of this method to non-gradient binary interactions, we refer to \cite{del2003note}.

For more general drifts $b$, the LDP is a harder task: it was initially proved in \cite{dawsont1987large} under Lipschitz conditions and non-degeneracy of $\sigma$, using an extended version of the weak topology on measures. 
The result is strengthened in \cite{budhiraja2012large} under much weaker assumptions (martingale problem framework), notably allowing for path-dependent settings. 
In both works, the initial data are assumed to be deterministic, and there is no direct way to handle random initial data from there.
We also refer to \cite{Carlo,budhiraj2022asymptotic}, which prove the LDP for \eqref{eq:informalDG} in a setting combining both mean-field and small-noise limits.

Computing the large deviation rate function as a relative entropy can be intricate, depending on the chosen approach for proving the LDP.
The link with stochastic control traces back to \cite{follmer1988random}.
We refer to \cite{cattiaux1995large,pra1996mckean,cattiaux1996minimization} for seminal results in this direction.
An exhaustive account of these links is given by \cite{leonard2012girsanov,fischer2014form}.

The question of establishing the LDP in the Wasserstein-$p$ topology is more recent. 
A necessary and sufficient condition for Sanov's theorem to hold in the Wasserstein topology is proved in \cite{wang2010sanov}. 
In the Brownian setting, this forces $p \in [1,2)$.
We already referred to \cite{backhoff2020mean,LackerMFGLDP}.
For large deviations in the strong topology, we refer to \cite{dawson2005large}.
For Gibbs measures in the Wasserstein topology, we also mention \cite{leonard1987large,reygner2018equilibrium,dupuis2020large,liu2020large}.
To the best of our knowledge, our LDP in the Wasserstein topology for a system as general as \eqref{eq:InterPart} is a new result.

\subsection{Notations} \label{subsec:not}

\begin{itemize}
\item $\ps (E)$ denotes the space of probability measures over a Polish space $E$, the default topology on it being the one of weak convergence.
\item $\mathcal{L}(X)$ denotes the law in $\ps(E)$ of a $E$-valued random variable $X$.
\item $T_\# P$ denotes the push-forward (or image) measure of $P \in \ps (E)$ by a measurable $T : E \rightarrow F$. 
\item When $E$ and $F$ are Banach spaces, $\C^1 (E,F)$ denotes the space of Fréchet-differentiable functions $E \rightarrow F$. 
$\C^{1,1} (E,F)$ is the sub-space of functions in $\C^1 (E,F)$ with globally Lipschitz derivative $D F$.
$\C_b^{1,1} (E,F)$ is the space of bounded functions in $\C^{1,1} (E,F)$ with bounded derivative.
\item $\delta_x$ denotes the Dirac measure at some point $x \in E$. 
\item Given $N$ variables $x^1, \ldots, x^N$ in some product space $E^N$, we will use the notations $\vec{x}^N := (x^1,\ldots,x^N)$ and
\[ \pi( \vec{x}^N ) := \frac{1}{N} \sum_{i=1}^N \delta_{x^i} \in \ps(E). \]
\item $\pi(\vec{x}^N,\vec{y}^N)$ will similarly denote the empirical measure
\[ \pi(\vec{x}^N,\vec{y}^N) := \frac{1}{N} \sum_{i=1}^N \delta_{(x^i,y^i)} \in \ps(E \times E), \]
with a slight abuse of notation.
\item $H( P \vert Q)$ denotes the relative entropy of $P$ w.r.t. $Q$, for $P, Q \in \ps (E)$, defined as $H( P \vert Q) = \E_P [ \log \tfrac{\d P}{\d Q} ]$ if $P \ll Q$, and $H( P \vert Q) = + \infty$ otherwise.
\item $\delta_P F (P) : x \mapsto \delta_P F (P,x)$ denotes the linear functional derivative at $P$ of a function $F : \ps(E) \rightarrow \R$, see Definition \ref{def:Diff} below.  
The convention is adopted that $\int_{E} \delta_P F (P) \d P =0$.
\item $\ps_p(E)$ denotes the set of measures $P \in \ps ( E)$ such that $\int_{E} d^p_E (x,x_0) \d P (x) < +\infty$ for the distance $d_E$ on $E$, $x_0 \in E$ and $p \geq 1$. 
\item $L^p_\P (\Omega, E)$ denotes the space of $E$-valued random variables $X : \Omega \rightarrow E$, on a probability space $( \Omega, \F, \P)$ such that $\L ( X )$ belongs to $\ps_p(E)$.
\item $W_p$ denotes the Wasserstein distance on $\ps_p (E)$, defined by
\[ W^p_p (P,Q) := \inf_{X \sim P, \, Y \sim Q} \E [ d_E^p ( X , Y ) ].  \]
\item $d, d' \geq 1$ are fixed integers, and $T >0$ is a given finite time horizon.
\item $\lfloor t \rfloor$ is the image of $t \in \R$ by the floor function. 
We similarly write $\lceil t \rceil$ for the ceiling function.
We recall that $\lfloor t \rfloor$ and $\lceil t \rceil$ are integers.
\item $\C^d$ denotes the space $\C([0,T],\R^d)$ of continuous paths, endowed with the uniform topology. 
A typical path in $\C^d$ will be denoted by $x$, and a typical one in $\C^{d'}$ by $\gamma$.
\item $\C^d_M$ denotes the space of paths $x \in \C^d$ with $\sup_{0 \leq t \leq T} \vert x_t \vert \leq M$, for some $M > 0$. 
\item $x_{\wedge t}$ denotes the path $(x_{s \wedge t})_{0 \leq s \leq T}$, for any $t \in [0,T]$ and $x \in \C^d$.
We recall that $s \wedge t$ is the minimum of $s$ and $t$.
\item $P_{\wedge t}$ denotes the push-forward measure of $P \in \ps (\C^d)$ by the map $x \mapsto x_{\wedge t}$. 
\item $P_t \in \ps ( \R^d)$ denotes the marginal-law at time $t$ of a path-measure $P \in \ps ( \C^d )$.
\item $P_\cdot \in \C([0,T],\ps ( \R^d))$ denotes the related curve $(P_t)_{0 \leq t \leq T}$ of time-marginals.
\item $\Sigma = \big( \Omega,(\mathcal{F}_t)_{0\leq t\leq T},\P,(B_t)_{0\leq t\leq T} \big)$ denotes a reference probability system in the terminology of \cite[Chapter 4]{fleming2006controlled}: $(\Omega,\mathcal{F}_T,\P)$ is a probability space, $(\mathcal{F}_t)_{0\leq t\leq T}$ is a filtration satisfying the usual conditions, and $(B_t)_{0\leq t\leq T}$ is a $(\mathcal{F}_t)_{0\leq t\leq T}$-Brownian motion.
\item $\vert a \vert$ denotes the Frobenius norm $\sqrt{\mathrm{Tr}[a a^\top]}$ of a matrix $a$.
\item $\nabla \cdot a$ is the vector field whose $i$-entry is the divergence of the vector field $( a^{i,j} )_{1 \leq j \leq d}$, for a matrix field $a  = (a^{i,j})_{1 \leq i,j \leq d}$. 
We similarly define differential operators on matrices. 
\end{itemize}

\begin{definition}[Linear functional derivative] \label{def:Diff}
A map $F : \ps_p (E) \rightarrow \R$ is differentiable at $P$ if there exists a measurable map 
\[ \delta_P F (P):
\begin{cases}
E \rightarrow \R, \\
x \mapsto\delta_P F (P,x),
\end{cases}
\]
such that for every $Q$ in $\ps_p (E)$, $\delta_P F (P)$ is $Q$-integrable and satisfies
\begin{equation*} \label{eq:FDiff}
\varepsilon^{-1} \bigl[ F( (1-\varepsilon) P + \varepsilon Q ) - F ( P ) \bigr] \xrightarrow[\varepsilon \rightarrow 0^+]{} \int_{E} \delta_P F (P) \, \d [ Q - P ]. 
\end{equation*} 
The map $\delta_P F (P)$ being defined up to an additive constant, we adopt the usual convention that $\int_{E} \delta_P F (P) \d P = 0$.
This map is called the \emph{linear functional derivative of $F$ at $P$}.
We notice that this definition does not depend on the behaviour of $F$ outside of an arbitrary small neighbourhood of $P$.
\end{definition}

By direct integration, the definition implies
\begin{equation*} \label{eq:Integration}
\forall P, Q \in \ps_p (E), \quad F(P)-F(Q) = \int_0^1 \int_{E} \delta_P F ((1-r)Q + r P) \d [ P - Q ] \d r, 
\end{equation*}
provided that the integral on the r.h.s. is well-defined.

\begin{example}[Linear case]
In the particular case $F(P) = \int_E f \d P$ for some measurable $f: E \rightarrow \R$ with polynomial growth of order $p$, we have $\delta_P F (P,x) = f(x) - \int_E f \d P$.
\end{example}

\section{Statement of the main results} \label{sec:Results}

Let us fix $p \geq 1$, together with positive integers $d$ and $d'$. 
We will restrict ourselves to $p \in [1,2)$ for large deviation results in Sections \ref{subsec:resDisB}-\ref{subsec:resCont}-\ref{subsec:contract}.
The coefficients are continuous functions
\[ b : [0,T] \times \C^d \times \ps_p ( \C^d ) \rightarrow \R^d, \qquad \sigma : [0,T] \times \C^d \times \ps_p ( \C^d ) \rightarrow \R^{d \times d'}. \]
For $P \in \ps_p ( \C^d )$, we recall that $P_{\wedge t} \in \ps_p ( \C^d )$ denotes the push-forward measure of $P$ under the map $x \in \C^d \mapsto x_{\wedge t} := (x_{s \wedge t})_{0 \leq s \leq T} \in \C^d$.
We make the following assumption on the coefficients.

\begin{assumption}[Coefficients] \label{ass:coef1}
There exist $L_b, L_\sigma > 0$ and 
\[ \overline{b} : [0,T] \times \C^d \times \ps_p ( \C^d ) \rightarrow \R^d, \qquad \overline\sigma : [0,T] \times \C^d \times \ps_p ( \C^d ) \rightarrow \R^{d \times d'}, \]
which are respectively globally $L_b$- and $L_\sigma$-Lipschitz and such that, for every $(t,x,P)$,
\[ b_t(x,P) = \overline{b}_t ( x_{\wedge t}, P_{\wedge t} ), \qquad \sigma_t(x,P) = \overline{\sigma}_t ( x_{\wedge t}, P_{\wedge t} ). \]
Moreover, $\sigma$ is globally bounded by some $M_\sigma > 0$.
\end{assumption}

In particular, this includes standard Lipschitz coefficients of the kind $b_t (x,P) = b_t ( x_t,P_t)$.
Our framework further includes pathwise settings which are non-Markov but still non-anticipative like models with delay. 
The Lipschitz-continuity w.r.t. $t$ will only be needed in the time-continuous settings of Sections \ref{subsec:resCont}-\ref{subsec:contract}, but we keep \ref{ass:coef1} as such for the sake of simplicity. 

\begin{lemma} \label{lem:JulEnvMcK}
Under \ref{ass:coef1}, for any Brownian motion $( B_t )_{0 \leq t \leq T}$ and initial data $\overline{X}_0$ with finite $\E [ \vert \overline{X}_0 \vert^p ]$,
strong existence and pathwise-uniqueness hold for solutions of the the path-dependent McKean-Vlasov equation 
\[ \d \overline{X}_t = b_t ( \overline{X}, \L ( \overline{X} ) ) \d t + \sigma_t ( \overline{X}, \L ( \overline{X} )) \d B_t \]
that satisfy $\E [ \sup_{0 \leq t \leq T} \vert \overline{X}_t \vert^p ] < +\infty$. 
\end{lemma}

This result is contained in \cite[Appendix A]{djete2022mckean}.
See \cite{lacker2022superposition,bernou2022path} for similar results.

\subsection{Discretised system with general driving noise} \label{subsec:resDis} 

Let $h > 0$ be a fixed discretisation parameter. 
Given $t \geq 0$, we define $t_h := h \lfloor t/h \rfloor$.
On a probability space $(\Omega,\F,\P)$, we study the discretised McKean-Vlasov SDE
\begin{equation} \label{eq:GenMcK}
\begin{cases}
\d X^h_{t_h} = b_{t_h} ( X^h, \mathcal{L}(X^h) ) \d t + \sigma_{t_h} (  X^h, \mathcal{L}(X^h) ) \d W_{t}, \\
X^h_{0} = \zeta,
\end{cases}    
\end{equation}
the input being the random variable 
\[ (\zeta,W) : \Omega \rightarrow \R^d \times \C^{d'}, \]
which belongs to $L^p_\P ( \Omega, \R^d \times \C^{d'})$.
The driving noise $W$ can be {any continuous random path, which does not have to be a Brownian motion}.
In particular, \eqref{eq:GenMcK} does not involve any stochastic integral, the differential being a mere notation.
To be non-ambiguous, we specify the notion of solution for \eqref{eq:GenMcK}.
For every $P \in \ps_p(\C^d)$, $(\zeta,W) \in \R^d \times \C^{d'}$, $X \in \C^d$, we define the path $f^{h,P} (\zeta,W,X)$ in $\C^d$ by
\begin{multline*} 
f_t^{h,P} (\zeta,W,X) = \zeta +
b_{t_h} ( X, P ) [ t - t_h ] + \sigma_{t_h} ( X, P ) [ W_{t} - W_{t_h} ] \\
+ \sum_{i=0}^{\lfloor t/h \rfloor-1} h b_{ih}( X, P ) + \sigma_{ih} ( X, P ) [ W_{(i+1)h} - W_{i h} ],
\end{multline*}
for every $t$ in $[0,T]$.
We use the usual convention that an empty sum equals $0$. 

\begin{definition} \label{def:GenMcK}
A pathwise solution of \eqref{eq:GenMcK} with input $(\zeta,W) \in L^p_\P ( \Omega, \R^d \times \C^{d'})$, 
is a random variable $X^h : \Omega \rightarrow \C^d$ that belongs to $L^p_\P ( \Omega, \C^d)$ such that for $\P$-a.e. $\omega$ in $\Omega$,
\[ \forall t \in [0,T], \quad X^h_t( \omega ) = f_t^{h,\mathcal{L}(X^h)} ( \zeta( \omega ),W( \omega ),X^h( \omega ) ).  \]
\end{definition} 

Closely related to \eqref{eq:GenMcK} is the discretised system of interacting particles $\vec{X}^{h,N} = (X^{h,i,N})_{1 \leq i \leq N}$ defined by
\begin{equation} \label{eq:hNpart}
\begin{cases}
\d X^{h,i,N}_t = b_{t_h} ( X^{h,i,N},\pi(\vec{X}^{h,N}) ) \d t + \sigma_{t_h} ( X^{h,i,N}, \pi ( \vec{X}^{h,N} ) ) \d W^{i,N}_t, \\
X^{h,i,N}_0 = \zeta^{i,N}, 
\end{cases}      
\end{equation}
for $1 \leq i \leq N$, with given inputs
\begin{equation*}
( \vec{\zeta}^N, \vec{W}^N ): \omega \mapsto \big( \zeta^{i,N}(\omega), W^{i,N}(\omega) )_{1 \leq i \leq N},
\end{equation*}
in $L^p_\P ( \Omega, (\R^d \times \C^{d'} )^N )$.
We emphasise that \emph{no independence nor exchangeability are required}.

\begin{rem}[Heterogeneous mean-field]
This formalism includes heterogeneous particle systems of the type
\begin{equation*} 
\d X^{h,i,N}_t = b_{t_h} ( X^{h,i,N},R^{i,N},\pi(\vec{X}^{h,N},R^{i,N}) ) \d t + \sigma_{t_h} (X^{h,i,N},R^{i,N},\pi(X^{h,i,N},R^{i,N}) ) \d W^{i,N}_t,    
\end{equation*}
for any family $( R^{i,N} )_{1 \leq i \leq N, N \geq 1}$ of random paths in $L^p_{\P} ( \Omega, \C^{k} )$ under natural Lispchitz assumptions on $(b,\sigma)$. 
Indeed, this amounts to replacing the random path $X^{h,i,N} \in \C^d$ by $Z^{h,i,N} := (X^{h,i,N},R^{i,N}) \in \C^{d+k}$. 
The following results can then be written in this setting, see \cite[Section 3.4]{Coghi2020PathwiseMT}.
This extends some results of \cite{pra1996mckean} for McKean-Vlasov processes that interact with random media.
See also \cite[Section 3.3]{Coghi2020PathwiseMT} for adding a common noise to \eqref{eq:hNpart}.
\end{rem}

We consider time-discretised systems because the $h \rightarrow 0$ limit would require some notion of stochastic integral that is not available at this level of generality.
However, when $\sigma \equiv \mathrm{Id}$, 
all the following results hold for $h = 0$, because the flow maps are still continuous functions of the driving noise. 
Adapting our proofs to this setting can be done readily, and we refer to \cite{Coghi2020PathwiseMT} for an exhaustive discussion of this setting.

\begin{theorem}[Pathwise construction] \label{thm:DisMFlimit}
Under \ref{ass:coef1},
\begin{enumerate}[label=(\roman*),ref=(\roman*)]
    \item\label{itm:thmMFexistN} There exists a pathwise unique strong solution $\vec{X}^{h,N}$ of \eqref{eq:hNpart}.
    \item\label{item:thmMFexist}  There exists a unique pathwise solution $X^h$ of \eqref{eq:GenMcK} in the sense of Definition \ref{def:GenMcK}, and 
    \[ \Psi_h :
    \begin{cases}
    \ps_p (\R^d \times \C^{d'}) \rightarrow \ps_p (\C^d), \\
    \mathcal{L}(\zeta,W) \mapsto \mathcal{L}(X^h),
    \end{cases} \]
    is continuous.
    \item\label{item:thmMFLim} $\P$-almost surely, $ \pi ( \vec{X}^{h,N} ) = \Psi_h ( \pi ( \vec{\zeta}^N, \vec{W}^N )  )$.
\end{enumerate}
\end{theorem}

The above setting reduces the mean-field limit to a simple continuity result.
No exchangeability nor other specific structures are required.
Since the map $\Psi_h$ is continuous, the following result is a direct consequence of the contraction principle \cite[Theorem 1.3.2]{dupuis2011weak}.
For general definitions about large deviations, see \cite{deuschel2001large,dupuis2011weak}.

\begin{corollary}[Mean-field limit and large deviations] \label{cor:DisLD}
$\phantom{a}$
\begin{enumerate}[label=(\roman*),ref=(\roman*)]
\item If $W_p ( \pi ( \vec{\zeta}^N , \vec{W}^N ) , \L ( \zeta, {W} ) ) \rightarrow 0$, $\P$-a.s. as $N \rightarrow +\infty$, then
\[ W_p ( \pi ( \vec{X}^{h,N} ) , \L ( {X}^h ) ) \rightarrow 0, \quad \P\text{-a.s. as } N \rightarrow + \infty. \]
\item If the sequence of the $\L( \pi ( \vec\zeta^N , \vec{W}^N ) )$ satisfies the LDP with good rate function $I : \ps_p ( \R^d \times \C^{d'}) \rightarrow [0,+\infty]$, then the sequence of the $\L( \pi ( \vec{X}^{h,N} ) )$ satisfies the LDP with good rate function 
\[ P \in \ps_p ( \C^d) \mapsto \inf_{\substack{R \in \ps( \R^d \times \C^{d'}) \\ \Psi_h (R) = P}} I(R). \]
\end{enumerate}
\end{corollary}

Theorem \ref{thm:DisMFlimit} is proved in Section \ref{subsec:MhFLDP}.
For proving the CLT, we need further assumptions on the coefficients.

\begin{assumption}[Strenghtening of \ref{ass:coef1}] \label{ass:coef2} \ref{ass:coef1} holds and for every $(t,x) \in [0,T] \times \C^d$,
\begin{enumerate}[label=(\roman*),ref=(\roman*)]
\item\label{item:DLip} $x \mapsto \overline{b}_t ( x, P), \overline\sigma_t (x,P)$ are Fréchet-differentiable and $(x,P) \mapsto D_x \overline{b}_t (x,P) , D_x \overline\sigma_t (x,P)$ are globally Lipschitz uniformly in $t$.
\item\label{item:DmuLip} $P \mapsto \overline{b}_t (x,P), \overline\sigma_t (x,P)$ have linear derivatives $y \mapsto \delta_P \overline{b}_t (x,P,y), \delta_P \overline\sigma_t (x,P,y)$, in the sense of Definition \ref{def:Diff}, which are Fréchet-differentiable w.r.t. $y$. \item\label{item:DDmuLip} $(x,P,y) \mapsto D_y \delta_P \overline{b}_t (x,P,y), D_y \delta_P \overline\sigma_t (x,P,y)$ are globally Lipschitz uniformly in $t$.
\end{enumerate}
\end{assumption}

Under \ref{ass:coef1}-\ref{ass:coef2}, we notice that $b_t(x,P)$ and $\sigma_t(x,P)$ are differentiable w.r.t. $(x,P)$ too, with derivatives only depending on $x_{\wedge t}$ and $P_{\wedge t}$. 
For the sake of simplicity, we further assume that all the particles start at $0$: we remove the dependence of $\Psi_h$ on the initial condition.
To state a result at this level of generality, we restrict ourselves to bounded noises taking values in the space $\C^{d'}_{M'}$ of paths with norm bounded by $M' >0$.  
This restriction is not needed if $\sigma \equiv \mathrm{Id}$, see \cite[Section 5]{Coghi2020PathwiseMT}.

\begin{theorem}[Central Limit Theorem] \label{thm:DisCLT}
Under \ref{ass:coef2}, let $(W^i)_{i \geq 1}$ be an i.i.d. sequence of $\C^{d'}_{M'}$-valued variables with common law $R$. 
Let $\vec{X}^{h,N}$ be the related system of particles starting at $0$ given by \eqref{eq:hNpart}, and let $X^h$ be the solution of \eqref{eq:GenMcK} driven by $W^1$.
Then, for every $\varphi \in \C^{1,1}_b ( \C^{d}, \R )$, the random variable
\[ \sqrt{N} \bigg[ \frac{1}{N} \sum_{i =1}^N \varphi ( X^{h,i,N} ) - \E [ \varphi( X^h ) ] \bigg] \]
converges in law towards a centred Gaussian variable whose variance $\sigma^2_\varphi$ is explicit in terms of $\Psi_h (R)$, see Corollary \ref{cor:hMCLT} below.
\end{theorem}

Theorem \ref{thm:DisCLT} is proved in Section \ref{subsec:disCLT}.

\subsection{Discretised setting with Brownian noise} \label{subsec:resDisB}

First, we apply Corollary \ref{cor:DisLD} to Brownian noises. 
To clearly separate this Brownian setting from the generic one of the previous section, we will write $(X^i_0)_{1 \leq i \leq N}$ for the initial conditions instead of $( \zeta^{i,N} )_{1 \leq i \leq N}$.
Since we aim for a LDP in the Wasserstein topology, \cite{wang2010sanov} tells us that we must restrict ourselves to $p \in [1,2)$ and that we need the following assumption.

\begin{assumption}[Initial exponential moments] \label{ass:IniExp}
$p \in [1,2)$ and the particles are initialised from a i.i.d. sequence $(X^i_0)_{i \geq 1}$ such that
\[ \forall \alpha >0, \qquad \E \big[ e^{\alpha \vert X^1_0 \vert^p} \big] < +\infty.  \]
\end{assumption}

For any $P$ in $\ps_p ( \C^d )$, we define $\Gamma_h ( P)$ as being the path-law of the solution to 
\begin{equation} \label{eq:defGamma}
\d Y^h_t = b_{t_h} ( Y^h, P) \d t + \sigma_{t_h} ( Y^h, P) \d B_t, \quad Y^h_0 = X^1_0, 
\end{equation} 
where $( B_t )_{0 \leq t \leq T}$ is a $\R^{d'}$-valued Brownian motion.
Since $\sigma$ is globally bounded and $b$ has linear growth, \ref{ass:coef1}-\ref{ass:IniExp} imply that
\[ \forall \alpha >0, \qquad \E \big[ \exp \big[ \alpha \sup_{0 \leq t \leq T} \vert Y^h_t \vert^p \big] \big] < + \infty, \]
as required by \cite[Theorem 1.1]{wang2010sanov}.

\begin{proposition}[Discretised LDP] \label{pro:LDPhB}
Under \ref{ass:coef1}-\ref{ass:IniExp}, let $(B^i)_{i \geq 1}$ be an i.i.d. sequence in $\C^{d'}$ of Brownian motions. 
Let $\vec{X}^{h,N}$ be the related system of particles starting from $(X^i_0)_{i \geq 1}$.
Then, the sequence of the $\mathcal{L}(\pi( \vec{X}^{h,N} ))$ satisfies the LDP with good rate function
\[ I_h : P \in \ps_p ( \C^d ) \mapsto H( P \vert \Gamma_h ( P) ). \] 
\end{proposition}

We recall that the relative entropy $H$ is defined in Section \ref{subsec:not}.
Proposition \ref{pro:LDPhB} is proved in Section \ref{subsec:MhFLDP}.
We then extend Theorem \ref{thm:DisCLT} to the Brownian setting.
We rely on an approximation procedure by truncating the noises.
To do so, we need uniform estimates for fluctuation processes given by Proposition \ref{pro:approxhM}.
In particular, we need uniform  propagation of chaos estimates with rate $N^{-1/2}$.
For such estimates to hold, it is convenient to restrict ourselves to binary interactions.

\begin{assumption}[Binary interactions] \label{ass:pair}
There exist functions $\tilde{b} : [0,T] \times \C^d \times \C^d \rightarrow \R^d$ and $\tilde{\sigma} : [0,T] \times \C^d \times \C^d \rightarrow \R^{d \times d'}$ such that
\[ b_t (x,P) = \int_{\C^d} \tilde{b}_t (x_{\wedge t},y) \, \d P_{\wedge t} (y), \qquad \sigma_t (x,P) = \int_{\C^d} \tilde\sigma_t (x_{\wedge t},y) \, \d P_{\wedge t} (y). \]
Moreover, $\tilde{b}$, $\tilde\sigma$ are $\C^{1,1}$ and $\tilde\sigma$ is bounded, so that $b$, $\sigma$ satisfy \ref{ass:coef1}-\ref{ass:coef2}.
\end{assumption}

Let $(B_t)_{0 \leq t \leq T}$, $(\tilde{B}_t)_{0 \leq t \leq T}$ be independent Brownian motions in $\R^d$.
Let $X^h$ denote the strong solution of the discretised McKean-Vlasov equation
\[ \d X^h_{t_h} = b_{t_h} ( X^h, \L(X^h) ) \d t + \sigma_{t_h} ( X^h, \L(X^h) ) \d B_t, \qquad X^h_0 = 0, \]
which corresponds to \eqref{eq:GenMcK} driven by $B$. 
We recall that strong existence and pathwise-uniqueness are provided by Theorem \ref{thm:DisMFlimit}.
Let $\tilde{X}^h$ denote the solution of the same equation driven by $\tilde{B}$.
We introduce the solution $\delta X^h$ of the McKean-Vlasov SDE 
\begin{multline} \label{eq:McKNoiseh}
\d \delta X^h_{t_h} = \big[ D_x \sigma_{t_h} ( X^h, \L ( X^h ) ) \cdot \delta X^h + \delta_P \sigma_{t_h} ( X^h, \L ( X^h ), \tilde{X}^h ) \big] \d B_t + \big\{ D_x b_{t_h} ( X^h, \L ( X^h ) ) \cdot \delta X^h \\
\phantom{abcdefgh}+ \delta_P b_{t_h} ( X^h, \L ( X^h ), \tilde{X}^h ) + \E \big[ D_y \delta_P b_{t_h} ( X^h , \L ( X^h ), {X}^h ) \cdot  \delta X^h \big\vert ( \tilde{B}_s )_{0 \leq s \leq t_h} \big] \big\} \d t \\
+ \E \big[ D_y \delta_P \sigma_{t_h} ( X^h, \L ( X^h ), {X}^h ) \cdot  \delta X^h \big\vert ( \tilde{B}_s )_{0 \leq s \leq t_h} \big] \d B_t, \qquad \delta X^h_0 = 0.\phantom{abcd}
\end{multline}
Well-posedness for \eqref{eq:McKNoiseh} is a direct induction (similar to Lemma \ref{lem:SolMap} below). 
This can be seen as a discretised version of a McKean-Vlasov SDE with common noise like the one studied in \cite[Appendix A]{djete2022mckean}.

\begin{proposition} \label{pro:ClThB}
Under \ref{ass:pair}, let $(B^i)_{i \geq 1}$ be an i.i.d. sequence in $\C^{d'}$ of Brownian motions. 
Let $\vec{X}^{h,N}$ be the related system of particles starting from $0$ given by \eqref{eq:hNpart}.
Then, for every $\varphi \in \C^{1,1}_b ( \C^{d}, \R )$, the random variable
\[ \sqrt{N} \bigg[ \frac{1}{N} \sum_{i =1}^N \varphi ( X^{h,i,N} ) - \E [ \varphi( X^h ) ] \bigg] \]
converges in law towards a centred Gaussian variable with variance $\sigma^2_{h,\varphi}$ given by
\[ \sigma^2_{h,\varphi} := \E \big\{ [ \varphi ( \tilde{X}^h ) + \E [ D \varphi ( X^h ) \cdot \delta X^h - \varphi ( X^h ) \vert \tilde{B} ] ]^2 \big\}. \]
\end{proposition}
    
The underlying Gaussian field can be described in terms of an infinite-dimensional SDE, see \cite[Section 4]{tanaka1984limit}.
Proposition \ref{pro:ClThB} is proved in Section \ref{subsec:disCLTB}, using results from Appendix \ref{app:tech}.

\subsection{Continuous setting with Brownian noise} \label{subsec:resCont}

We now extend Proposition \ref{pro:LDPhB} to the continuous setting. 
For $P \in \ps_p ( \C^d )$, we define
$\Gamma_0 ( P)$ as being the path-law of the solution to 
\begin{equation*} 
\d Y_t = b_{t} ( Y, P) \d t + \sigma_{t} ( Y, P) \d B_t, \quad Y_0 = X^1_0, 
\end{equation*} 
where $( B_t )_{0 \leq t \leq T}$ is a $\R^{d'}$-valued Brownian motion.
We notice that $\Gamma_0 (P)$ can be seen as the $h \rightarrow 0$ limit of $\Gamma_h (P)$ given by \eqref{eq:defGamma}.
In the following, we will use the same convention for extending definitions to $h = 0$.

\begin{theorem}[Exponential approximation and LDP] \label{thm:LDPGood}
Let us assume that $\sigma_t ( x,P) = \sigma_t (P)$ does not depend on $x$.
Under \ref{ass:coef1}-\ref{ass:IniExp}, let $(B^i)_{i \geq 1}$ be an i.i.d. sequence in $\C^{d'}$ of Brownian motions. 
For $h \in [0,1]$, let $\vec{X}^{h,N}$ be the related system of particles starting from $(X^i_0)_{i \geq 1}$.
Then, for every bounded Lipschitz $F : \ps_p ( \C^d ) \rightarrow \R$,
\[ \sup_{N \geq 1} \big\vert N^{-1} \log \E \big[ \exp \big[ N F ( \pi ( \vec{X}^{h,N} ) \big] \big] - N^{-1} \log \E \big[ \exp \big[ N F ( \pi ( \vec{X}^{0,N} ) \big] \big] \big\vert \xrightarrow[h \rightarrow 0]{} 0. \]
Moreover, the sequence of the $\mathcal{L}(\pi( \vec{X}^{0,N} ))$ satisfies the LDP with good rate function
\[ I_0 : P \in \ps_p ( \C^d ) \mapsto H( P \vert \Gamma_0 ( P) ). \] 
Moreover, $I_h$ $\Gamma$-converges towards $I_0$ as $h \rightarrow 0$.
\end{theorem}

Theorem \ref{thm:Contraction} below provides the LDP when $\sigma$ does depend on $x$ but not on $P$. 
The above result can be rephrased as the quantitative convergence of the value function of a mean-field control problem with quadratic cost, see Proposition \ref{pro:valueF} below.
Theorem \ref{thm:LDPGood} is proved in Section \ref{subsec:contLDP}.
Some technical results about rate functions are deferred to Appendix \ref{app:LDP} (in particular, links with stochastic control).

\begin{rem}[Mean-field limit and scheme convergence]
From Lemma~\ref{lem:Approxh}-\ref{item:lemApphhinf} and Lemma~\ref{lem:PoCh}, we have the consistency results, for every $h \in [0,1]$, $N \geq 1$ and $1 \leq i \leq N$,
\[ \E [ \sup_{0 \leq t \leq T} \vert X^{h,i,N}_t - X^{0,i,N}_t \vert^2 ] \leq C \vert h \log h \vert, \quad \E [ \sup_{0 \leq t \leq T} \vert X^{h,i,N}_t - \overline{X}^{h,i}_t \vert^2 ] \leq C N^{-1}, \]
for $C >0 $ independent of $(h,N)$.
In particular, the $h \rightarrow 0$ and $N \rightarrow +\infty$ limits commute with quantitative rates of convergence, see Lemma~\ref{lem:PoCh}.
Theorem \ref{thm:LDPGood} proves the same commutation at the level of large deviations, whereas Theorem \ref{thm:ClT} proves it for normal fluctuations.
\end{rem}

We now extend Proposition \ref{pro:ClThB} to the continuous setting.
As previously, let $(B_t)_{0 \leq t \leq T}$, $(\tilde{B}_t)_{0 \leq t \leq T}$ be independent Brownian motions in $\R^d$.
Let $X$ denote the strong solution of the McKean-Vlasov equation given by Lemma \ref{lem:JulEnvMcK},
\[ \d X_{t} = b_{t} ( X, \L(X) ) \d t + \sigma_{t} ( X, \L(X) ) \d B_t, \qquad X_0 = 0, \]
which corresponds to the $h= 0$ extension of \eqref{eq:GenMcK} driven by $B$. 
Let $\tilde{X}$ denote the solution of the same equation driven by $\tilde{B}$.
We then introduce the solution $\delta X$ of 
\begin{multline} \label{eq:McKNoise}
\d \delta X_{t} = \big[ D_x \sigma_{t} ( X, \L ( X ) ) \cdot \delta X + \delta_P \sigma_{t} ( X, \L ( X), \tilde{X} ) \big] \d B_t + \big\{ D_x b_{t} ( X, \L ( X) ) \cdot \delta X \\
\phantom{abcdefgh}+ \delta_P b_{t} ( X, \L ( X ), \tilde{X} ) + \E \big[ D_y \delta_P b_{t} ( X , \L ( X ), X ) \cdot  \delta X \big\vert ( \tilde{B}_s )_{0 \leq s \leq t} \big] \big\} \d t \\
+ \E \big[ D_y \delta_P \sigma_{t} ( X, \L ( X ), X ) \cdot  \delta X \big\vert ( \tilde{B}_s )_{0 \leq s \leq t} \big] \d B_t, \qquad \delta X_0 = 0.\phantom{abcd}
\end{multline}
Well-posedness for such a path-dependent McKean-Vlasov SDE with common noise can be found in \cite[Appendix A]{djete2022mckean}.

\begin{theorem} \label{thm:ClT}
Under \ref{ass:pair}, let $(B^i)_{i \geq 1}$ be an i.i.d. sequence in $\C^{d'}$ of Brownian motions. 
Let $\vec{X}^{h,N}$ be the related system of particles starting from $0$ given by \eqref{eq:hNpart}.
Let further $\overline{X}^{h,i}$ denote the solution of \eqref{eq:GenMcK} with driving noise $B^i$.
There exists $C >0$ such that  
\[ \sup_{N \geq 1} \frac{1}{\sqrt{N}} \sum_{i=1}^N \E \bigg[ \sup_{0 \leq t \leq T} \big\vert X^{h,i,N}_t - \overline{X}^{h,i}_t - [ X^{0,i,N}_t - \overline{X}^{0,i}_t ] \big\vert \bigg] \leq C \vert h \log h \vert^{1/2}, \]
for every $h \in [0,1]$.
Moreover, for every $\varphi \in \C^{1,1}_b ( \C^{d}, \R )$,
\[ \sqrt{N} \bigg[ \frac{1}{N} \sum_{i =1}^N \varphi ( X^{0,i,N} ) - \E [ \varphi( \overline{X}^{0,1} ) ] \bigg] \]
converges in law towards a centred Gaussian variable with variance $\sigma^2_{\varphi}$ given by
\[ \sigma^2_{\varphi} := \E \big\{ [ \varphi ( \tilde{X} ) + \E [ D \varphi ( X ) \cdot \delta X - \varphi ( X ) \vert \tilde{B} ] ]^2 \big\}. \]
\end{theorem}
    
The underlying Gaussian field can be described in terms of an infinite-dimensional SDE, see \cite[Section 4]{tanaka1984limit}.
Theorem \ref{thm:ClT} is proved in Section \ref{subsec:CLT}.

\subsection{Contraction} \label{subsec:contract}

Since the map
\begin{equation*} \label{eq:projmarg}
\begin{cases}
\ps_p ( \C^d ) &\rightarrow \C ( [0,T], \ps_p ( \R^d ) ), \\
\phantom{ab}P &\mapsto P_\cdot := (P_t)_{0 \leq t \leq T},
\end{cases}
\end{equation*}
is continuous, Proposition \ref{pro:LDPhB} and Theorem \ref{thm:LDPGood} induce by contraction the LDP for the sequence of the $\L( \pi_\cdot ( \vec{X}^{h,N} ) )$ with good rate function
\[ \overline{I}_h : P_\cdot \in \C ( [0,T], \ps_p ( \R^d ) ) \mapsto \inf_{\substack{Q \in \ps_p(\C^d) \\ Q_\cdot = P_\cdot}} H ( Q \vert \Gamma_h ( Q ) ), \]
where $h \in [0,1]$. 
The random curve $\pi_\cdot ( \vec{X}^{0,N} )$ precisely corresponds to \eqref{eq:informalDG} in the Introduction, this framework being natural by comparison with \eqref{eq:fw}.
In this setting, the following result allows for a dependence of $\sigma$ on $x$.

\begin{assumption} \label{ass:coef3}
The coefficients are globally Lipschitz functions independent of the history $b_t (x,P) = b_t(x_t,P_t)$ and $\sigma_t (x,P) = \sigma_t(x_t)$.
Moreover, $\sigma$ is globally bounded.
\end{assumption}

\begin{theorem} \label{thm:Contraction}
Under \ref{ass:IniExp}-\ref{ass:coef3}, let $(B^i)_{i \geq 1}$ be an i.i.d. sequence in $\C^{d'}$ of Brownian motions. 
For $h \in [0,1]$, let $\vec{X}^{h,N}$ be the related system of particles starting from $(X^i_0)_{i \geq 1}$.
Then, for every bounded Lipschitz $F : \C ( [0,T], \ps_p ( \R^d ) ) \rightarrow \R$,
\[ \sup_{N \geq 1} \big\vert N^{-1} \log \E \big[ \exp \big[ N F ( \pi_\cdot ( \vec{X}^{h,N} ) \big] \big] - N^{-1} \log \E \big[ \exp \big[ N F ( \pi_\cdot ( \vec{X}^{0,N} ) \big] \big] \big\vert \xrightarrow[h \rightarrow 0]{} 0. \]
Moreover, the sequence of the $\mathcal{L}(\pi_\cdot ( \vec{X}^{h,N} ))$ satisfies the LDP with good rate function $\overline{I}_h$.   
\end{theorem}

From Lemma \ref{lem:RepEntr}, we can intuite the representation formula
\[ \overline{I}_0 ( P_\cdot ) = H(P_0 \vert \L(X^1_0) ) + \int_0^T \int_{\R^d} \frac{1}{2} \vert v_t \vert^2 \d P_t \d t, \]
if there exists $v \in L^2 ([0,T] \times \R^d, \R^{d'}, \d t \otimes \d P_t )$ such that
\[ \partial_t P_t = \nabla \cdot \big[ - P_t b_t (\cdot,P_t) - P_t \sigma_t (\cdot,P_t) v_t + \tfrac{1}{2} \nabla \cdot [ P_t \sigma_t \sigma^\top_t (\cdot,P_t) ] \big], \]
in the sense of distributions, and $\overline{I}_0 ( P_\cdot ) = +\infty$ otherwise. 
When $d = d'$,
this formula is proved in \cite[Lemma 6.17]{LackerMFGLDP} when $\sigma \equiv \mathrm{Id}$, and in \cite[Lemma 4.8]{dawsont1987large} under a non-degeneracy assumption on $\sigma$.
Theorem \ref{thm:Contraction} is proved in Section \ref{subsec:contractproof}.

\section{Constructions on the path space} \label{sec:Cons}

This section is devoted to the discretised setting.
The construction of the map $\Psi_h$ is performed in Section \ref{subsec:ConsCons}.
The results about mean-field limit and large deviations are proved in Section \ref{subsec:MhFLDP}. 
The central limit theorem is eventually proved in Section
\ref{subsec:disCLT}.

\subsection{Construction of the discretised system} \label{subsec:ConsCons}

We follow the presentation of \cite[Section 2]{Coghi2020PathwiseMT}, adapting their notations to our setting.
Before studying the non-linear SDE \eqref{eq:GenMcK}, we show well-posedness for the related ODE obtained by freezing the random inputs $( \zeta( \omega ),W( \omega ) )$ and the measure argument $\mathcal{L}(X)$. 

\begin{lemma}[Measure-frozen equation] \label{lem:fixed} 
$\phantom{A}$
\begin{enumerate}[label=(\roman*),ref=(\roman*)]
    \item\label{item:lemfrozfixe} For every $(x_0,\gamma,P) \in \R^d \times \C^{d'} \times \ps_p ( \C^d )$, the map $x \mapsto f^{h,P}(x_0,\gamma,x)$ has a unique fixed-point $S^{h,P}(x_0,\gamma)$ in $\C^d$.
    \item\label{item:lemfrozLip} For any $\gamma \in \C^{d'}$, there exists $L(\gamma)$ such that for every $(x_0,P) \in \R^d \times \ps_p ( \C^d )$ and $(x_0',\gamma',P') \in \R^d \times \C^{d'} \times \ps_p ( \C^d )$, $\vert S^{h,P}(x_0,\gamma)-S^{h,P'}(x_0',\gamma') \vert \leq L(\gamma) [ \vert x_0-x_0' \vert + \vert \gamma-\gamma'\vert + W_p(P,P')]$, and $\gamma \mapsto L(\gamma)$ is bounded on bounded sets.
    \item\label{item:lemfrozSDE}  For $(\zeta,W) \in L^p_\P ( \Omega, \R^d \times \C^{d'})$ and $P \in \ps_p ( \C^d )$, the path-dependent SDE
    \begin{equation} \label{eq:frozSDE}
    \d Y_t = b_{t_h}(Y,P) \d t + \sigma_{t_h} (Y,P) \d W_t, \quad Y_0 = \zeta, 
    \end{equation} 
    has a pathwise-unique strong solution in $L^p_\P ( \Omega, \C^d )$, given by $S^{h,P}(\zeta,W)$.
\end{enumerate}
\end{lemma}

The proof of Lemma \ref{lem:fixed} is simple, because \eqref{eq:frozSDE} is an explicit Euler scheme for a SDE.

\begin{proof}
\ref{item:lemfrozfixe}
If $( S^{h,P}_s ( x_0 , \gamma ) )_{0 \leq s \leq i h}$ is uniquely defined for some $i \geq 0$, then  
\[ S^{h,P}_t ( x_0 , \gamma ) = S^{h,P}_{ih} ( x_0 , \gamma ) + (t-ih) b_{ih} ( S^{h,P}_{\wedge ih} ( x_0 , \gamma ), P ) + \sigma_{ih} ( S^{h,P}_{\wedge ih} ( x_0 , \gamma ) , P) [ \gamma_{t} - \gamma_{ih} ], \]
for $ih \leq t \leq (i+1)h$, uniquely extending the definition till time $(i+1)h$.
All this makes sense because $b_{ih} (x,P)$ and $\sigma_{i h} (x,P)$ only depend on $(x_s)_{0 \leq s \leq ih}$ using \ref{ass:coef1}.
By induction, this proves existence and uniqueness for the fixed-point $S^{h,P}(x_0,\gamma)$.

\ref{item:lemfrozLip} We consider $(x_0,\gamma,P), (x'_0,\gamma',P') \in \R^d \times \C^{d'} \times \ps_p (\C^d)$.
By induction, if
\begin{equation} \label{eq:Lipi}
\sup_{0 \leq s \leq ih} \lvert S^{h,P}_{s} ( x_0 , \gamma ) - S^{h,P'}_{s} ( x'_0 , \gamma' ) \rvert \leq L_i \big[ \lvert x_0 - x'_ 0 \rvert + \lvert \gamma - \gamma' \rvert + W_p ( P_{\wedge ih}, P'_{\wedge ih} ) \big], 
\end{equation} 
for some $L_i > 0$, then \ref{ass:coef1} yields, for $ih \leq t \leq (i+1)h$, 
\begin{multline*}
\lvert S^{h,P}_{t} ( x_0 , \gamma ) - S^{h,P'}_{t} ( x'_0 , \gamma' ) \rvert \leq \lvert S^{h,P}_{i h} ( x_0 , \gamma ) - S^{h,P'}_{i h} ( x'_0 , \gamma' ) \rvert +  2 M_\sigma \sup_{ih \leq s \leq t} \lvert \gamma_s - \gamma'_s \rvert \\
+ \big[ h L_b + 2 L_\sigma \sup_{ih \leq s \leq (i+1) h} \lvert \gamma_s \rvert \big] \big[ \sup_{0 \leq s \leq ih} \lvert S^{h,P}_{s} ( x_0 , \gamma ) - S^{h,P'}_{s} ( x'_0 , \gamma' ) \rvert + W_p ( P_{\wedge i h}, P'_{\wedge i h} ) \big].
\end{multline*} 
As a consequence, using that $W_p(P_{\wedge ih}, P'_{\wedge ih}) \leq W_p(P_{\wedge (i+1)h}, P'_{\wedge (i+1)h})$, \eqref{eq:Lipi} holds till time $(i+1)h$ for $L_{i+1} := L_i [ 1 + h L_b + 2 L_\sigma \sup_{ih \leq s \leq (i+1) h} \lvert \gamma_s \rvert ] + \max\{2M_\sigma, h L_b + 2 L_\sigma \sup_{ih \leq s \leq (i+1) h} \lvert \gamma_s \rvert\}$.

\ref{item:lemfrozSDE} As a consequence of \ref{item:lemfrozLip}, $Y : \omega \mapsto S^{h,P}(\zeta(\omega),W(\omega))$ is measurable. 
Using \ref{ass:coef1} and taking expectations, there exists $C >0$ such that
\[
\E \sup_{ih \leq t \leq (i+1)h} \lvert Y_t \rvert^p \leq C \bigg[ 1 + \E \sup_{0 \leq s \leq ih} \lvert Y_{s} \rvert^p + M^p_\sigma \E \sup_{ih \leq t \leq (i+1)h} \lvert W_t - W_{ih} \rvert^p + \int_{\C^d} \sup_{0\leq s \leq ih} \lvert x_s \rvert^p \d P_{\wedge ih} ( x )\bigg].   
\]
By induction, $\E [ \sup_{0 \leq t \leq T} \lvert Y_t \rvert^p ]$ is thus finite if $\E [ \lvert \zeta \rvert^p ]$ and $\E [ \sup_{0 \leq t \leq T} \lvert W_t \rvert^p ]$ are finite.
The fact that $Y$ is the pathwise-unique strong solution of \eqref{eq:frozSDE} then stems from \ref{item:lemfrozfixe}.
\end{proof}

We now turn to the discretised McKean-Vlasov equation \eqref{eq:GenMcK}. 

\begin{lemma}[Pathwise solution map] \label{lem:SolMap} 
For any $R \in \ps_p(\R^d \times \C^{d'})$, there exists a unique map
\[ \overline{S}^{h,R} :
\begin{cases}
\R^d \times \C^{d'} &\rightarrow \C^d,  \\
$\phantom{,}$( x_0, \gamma ) &\mapsto \overline{S}^{h,R} ( x_0, \gamma ),
\end{cases} \]
that is everywhere well-defined, that belongs to $L^p_R (\R^d \times \C^{d'}, \C^d)$, and that satisfies
\begin{equation} \label{eq:EverywhereODE}
\forall (x_0,\gamma,t) \in \R^d \times \C^{d'} \times [0,T], \quad \overline{S}^{h,R}_t (x_0, \gamma ) = f_{t}^{h,\overline{S}^{h,R}_\# R} (x_0,\gamma, \overline{S}^{h,R} (x_0, \gamma )). 
\end{equation} 
Moreover, for $(\zeta,W) \in L^p_\P ( \Omega, \R^d \times \C^{d'})$, \eqref{eq:GenMcK} has a unique solution in the sense of Definition \ref{def:GenMcK}, given by $\overline{S}^{h,\L(\zeta,W)}(\zeta,W)$.
\end{lemma}

In particular, $\overline{S}^{h,R}_\# R$ belongs to $\ps_p ( \C^d )$. 
In the case $R = \L ( \zeta, W)$, $\overline{S}^{h,R}_\# R$ is the law of the solution of \eqref{eq:GenMcK}. 
Following \cite[Section 2]{Coghi2020PathwiseMT}, $\overline{S}^{h,R}_\# R$ can be characterised as the unique fixed-point of the map
\[ P \mapsto S^{h,P}_\# R. \]
In \cite[Section 2.1]{tanaka1984limit}, \cite[Lemma 13]{Coghi2020PathwiseMT} and \cite[Lemma 3.1]{backhoff2020mean}, the fixed-point is built by showing convergence for an iterative scheme.
The result is much easier for our explicit Euler scheme.

\begin{proof}
We first show that \eqref{eq:GenMcK} has a pathwise unique strong solution, before using this solution to build $\overline{S}^{h,R}(x_0,\gamma)$.

By induction, if $( X^{h}_s )_{0 \leq s \leq i h}$ is $\P$-a.s. uniquely defined for some $i \geq 0$, then  
\[  X^{h}_t = X^{h}_{ih} + (t-ih) b_{ih} ( X^h_{\wedge ih}, \mathcal{L}( X^{h}_{\wedge ih} ) ) + \sigma_{ih} ( X^h_{\wedge ih}, \mathcal{L}( X^{h}_{\wedge ih} ) ) [ W_{t} - W_{ih} ], \]
for $ih \leq t \leq (i+1)h$, uniquely extending the definition till time $(i+1)h$.
All this makes sense because $b_{ih} (x,P)$ and $\sigma_{ih} (x,P)$ only depend on $x_{\wedge ih}$ and $P_{\wedge ih}$ using \ref{ass:coef1}.
By induction, the above Euler scheme builds $(X^h_{\wedge i h}, \L(X^h_{\wedge ih}))$ for every $i \geq 0$, proving existence and pathwise uniqueness for \eqref{eq:GenMcK}.  
In particular, this proves existence and uniqueness for $\L (X^h)$.

To obtain an everywhere-defined solution, we set $\overline{S}^{h,R}(x_0,\gamma) := S^{h,\mathcal{L}(X^h)}(x_0,\gamma)$, where $\L(X^h)$ is the path-law of the solution of \eqref{eq:GenMcK} for any $(\zeta,W)$ with law $R$.
$\overline{S}^{h,R}(\zeta,W)$ a.s. satisfies \eqref{eq:frozSDE} with $P = \L ( X^h )$, so that 
$\overline{S}^{h,R}(\zeta,W)$ is the pathwise solution of \eqref{eq:GenMcK}. 
Consequently, $\overline{S}^{h,R}_\# R = \mathcal{L}(X^h)$, and $\overline{S}^{h,R}(x_0,\gamma)$ satisfies \eqref{eq:EverywhereODE} for every $(x_0,\gamma)$.
Uniqueness stems from using Lemma \ref{lem:fixed}-\ref{item:lemfrozfixe} with $P = \overline{S}^{h,R}_\# R$.
\end{proof}

We now state the main result of this section.

\begin{proposition}[Continuous fixed-point map] \label{pro:contFixed}
The map 
\[
\Psi_h :
\begin{cases}
\ps_p (\R^d \times \C^{d'}) &\rightarrow \ps_p (\C^{d}), \\
\phantom{abcde}R &\mapsto \overline{S}^{h,R}_\# R,
\end{cases}
\]
is continuous.
Moreover, for $(\zeta,W) \in L^p_\P ( \Omega, \R^d \times \C^{d'})$, $\Psi_h (\L(\zeta,W))$ is the law of the unique solution of \eqref{eq:GenMcK}.
\end{proposition}

When the matrix $\sigma$ is constant, we can show that $\Psi_h$ is globally Lipschitz-continuous (see Lemma \ref{lem:restrict} below), as in \cite[Theorem 7]{Coghi2020PathwiseMT} and \cite[Lemma 3.4]{backhoff2020mean}.
In a setting that includes common noise, \cite[Lemma 6.16]{LackerMFGLDP} could also be adapted to show that $\Psi_h$ is uniformly continuous, but still requiring that $\sigma$ is constant. 
To handle general $\sigma$, we rely on a compactness approach.

\begin{proof}
In view of Lemma \ref{lem:SolMap}, we only have to prove the continuity of $\Psi_h$. 
We use the sequential characterisation in the metric space $(\ps_p (\R^d \times \C^{d'}),W_p)$. 
Let $( R_k )_{k \geq 1}$ be a sequence in $\ps_p (\R^d \times \C^{d'})$ that converges towards some $R$. 
Since the $W_p$-convergence implies weak convergence \cite[Definition 6.8]{villani2009optimal}, the Skorokhod representation theorem provides a probability space $(\tilde{\Omega},\tilde\F,\tilde\P)$ that supports a sequence $( \zeta^k, W^k )_{k \geq 1}$, where $( \zeta^k, W^k )$ is $R_k$-distributed and a.s. converges towards some $R$-distributed $( \zeta, W )$.
Let $X^{h,k}$ denotes the pathwise solution of \eqref{eq:GenMcK} driven by $( \zeta^k, W^k )$, and let $X^{h}$ denote the one driven by $( \zeta, W )$.

\medskip
\emph{\textbf{Step 1.} Uniform in $k$ bounds.}
As in the proof of Lemma \ref{lem:fixed}-\ref{item:lemfrozSDE}, \ref{ass:coef1} yields
\begin{equation*} 
\sup_{i h \leq s \leq (i+1) h} \lvert X^{h,k}_s \rvert^p \leq C \big[ 1 + \sup_{0 \leq s \leq i h} \lvert X^{h,k}_s \rvert^p + \sup_{ih \leq s \leq (i+1) h} \lvert W^k_s \rvert^p + \E_{\tilde{\P}} \sup_{0 \leq s \leq i h} \lvert X^{h,k}_s \rvert^p \big], 
\end{equation*} 
for every $i \geq 0$ and a constant $C > 0$ that does not depend on $k$. 
From \cite[Definition 6.8-(i)]{villani2009optimal}, $( \E_{\tilde{\P}} [ \vert \zeta^k \vert^p ] )_{k \geq 1}$ and $( \E_{\tilde{\P}} [ \sup_{0 \leq t \leq T} \vert W^k_t \vert^p ] )_{k \geq 1}$ are converging sequences, hence bounded ones. 
We now take expectations to get by induction that
\begin{equation} \label{eq:MomentUnifk}
\sup_{k \geq 1} \E_{\tilde\P} \, \sup_{0 \leq t \leq T} \vert X^{h,k}_t \vert^p  < +\infty, 
\end{equation}
and then $\tilde{\P}$-a.s.,
\begin{equation} \label{eq:DomUnifk}
\sup_{0 \leq t \leq T} \lvert X^{h,k}_t \rvert^p \leq C \big[ 1 + \vert \zeta^k \vert^p + \sup_{0 \leq t \leq T} \lvert W^k_t \rvert^p \big],     
\end{equation}
for some $C >0$ independent of $k$.
For $s,t \in [0,T]$ with $t_h  \leq s \leq t \leq t_h + h$, using \ref{ass:coef1},
\begin{align*}
\vert X^{h,k}_t - X^{h,k}_s \vert \leq (t-s) \vert b_{t_h} ( X^{h,k} , \mathcal{L}(X^{h,k}) ) \vert +  M_\sigma \vert W^k_{t} - W^k_s \vert.   
\end{align*}
Using \ref{ass:coef1} and the bound \eqref{eq:MomentUnifk}, $\E_{\tilde\P} [ \sup_{0 \leq t \leq T} \vert b_t ( X^{h,k} , \L ( X^{h,k} ) ) \vert ]$ is bounded uniformly in $k$.
Taking expectations, we deduce that for every $\delta \in (0,h]$,
\begin{equation} \label{eq:TransferContk}
\E_{\tilde\P} \sup_{\vert t-s \vert \leq \delta} \vert X^{h,k}_t - X^{h,k}_s \vert  \leq C \delta + C \E_{\tilde\P} \sup_{\vert t-s \vert \leq \delta} \vert W^k_t - W^k_s \vert,  
\end{equation}
for a constant $C > 0$ that does not depend on $k$ and $\delta$. 

\medskip
\emph{\textbf{Step 2.} Weak relative compactness.}
For $\varepsilon, M > 0$, we decompose
\begin{multline*}
\E_{\tilde\P} \sup_{\vert t-s \vert \leq \delta} \vert W^k_t - W^k_s \vert^p
\leq \varepsilon^p + (2 M)^p \tilde\P \bigg[ \sup_{\vert t-s \vert \leq \delta} \vert W^k_t - W^k_s \vert > \varepsilon \bigg] \\
+ 2^p \E_{\tilde\P} \bigg[ \sup_{0 \leq t \leq T} \vert W^k_t \vert^p \1_{\sup_{0 \leq t \leq T} \vert W^k_t \vert \geq M} \bigg].
\end{multline*}
Since $( R_k )_{k \geq 1}$ converges in $( \ps_p (\R^d \times \C^{d'}), W_p)$, we deduce from \cite[Definition 6.8-(iii)]{villani2009optimal} that $M_\varepsilon > 0$ exists such that 
\[ \sup_{k \geq 1} \E_{\tilde\P} \sup_{0 \leq t \leq T} \vert W^k_t \vert^p \1_{\sup_{0 \leq t \leq T} \vert W^k_t \vert \geq M_\varepsilon} \leq \varepsilon. \]
Since $(R_k)_{k \geq 1}$ converges in $W_p$, it is tight, and \cite[Theorem 7.3]{billingsley2013convergence} provides that 
\[ \sup_{k \geq 1} \tilde\P \bigg[ \sup_{\vert t-s \vert \leq \delta} \vert W^k_t - W^k_s \vert > \varepsilon \bigg] \leq \varepsilon / M_\varepsilon^p, \]
for every small enough $\delta >0$. Gathering everything, we get that
\[ \sup_{k \geq 1} \, \E_{\tilde\P} \sup_{\vert t-s \vert \leq \delta} \vert W^k_t - W^k_s \vert^p \xrightarrow[\delta \rightarrow 0]{} 0. \]
From \eqref{eq:TransferContk}, using $p \geq 1$ and Jensen's inequality, we obtain that
\begin{equation*} 
\sup_{k \geq 1} \, \E_{\tilde\P} \sup_{\vert t-s \vert \leq \delta} \big\vert X^{h,k}_t - X^h_t \big\vert \xrightarrow[\delta \rightarrow 0]{} 0.
\end{equation*}
Using this together with \eqref{eq:MomentUnifk}, \cite[Theorem 7.3]{billingsley2013convergence} now gives that the $\overline{S}^{h,R_k}_\# R_k = \L ( X^{h,k} )$ form a relatively compact sequence in $\ps ( \C^{d} )$ for the weak topology.

\medskip
\emph{\textbf{Step 3.} $W_p$-convergence.}
Since $R_k$ converges in the $W_p$-topology, $( \vert \zeta^k \vert^p )_{k \geq 1}$ and $( \vert W^k \vert^p)_{k \geq 1}$
are uniformly integrable \cite[Definition 6.8-(iii)]{villani2009optimal}.
From \eqref{eq:DomUnifk}, we then deduce that $(\sup_{0 \leq t \leq T} \vert X^{h,k}_t \vert^p )_{k \geq 1}$ is uniformly integrable.  
Using \emph{\textbf{Step 2.}} and the characterisation \cite[Definition 6.8-(iii)]{villani2009optimal}, this implies that $( \L ( X^{h,k} ) )_{k \geq 1}$ is relatively compact in $(\ps_p ( \C^d ),W_p)$.
Let $P$ be any limit point in $\ps_p ( \C^d)$.

Using the $\tilde\P$-a.s. convergence of $(\zeta^k,W^k)$, and the continuity of $(x_0,\gamma,P) \mapsto S^{h,P} ( x_0, \gamma )$, which is implied by Lemma \ref{lem:fixed}-\ref{item:lemfrozLip}, $\overline{S}^{h,R_k}(\zeta^k,W^k) = S^{h,\L ( X^{h,k} )} (\zeta^k,W^k)$ has a sub-sequence that $\tilde\P$-a.s. converges towards $S^{h,P} (\zeta,W)$. 
Along such a sub-sequence, we deduce that $\L ( X^{h,k} ) = \L ( \overline{S}^{h,R_k}(\zeta^k,W^k) )$ weakly converges towards $\L(S^{h,P} (\zeta,W))$, so that $P = S^{h,P}_{\#} R$. 
This implies that $S^{h,P} (\zeta,W)$ is a solution of \eqref{eq:GenMcK}.
The pathwise uniqueness in Lemma \ref{lem:SolMap} then gives that $S^{h,P} (\zeta,W) = \overline{S}^{h,R} (\zeta,W)$, $\tilde\P$-a.s., so that $P = \overline{S}^{h,R}_\# R = \Psi_h (R)$.
The unique limit point in $W_p$ of the relatively compact sequence $(\overline{S}^{h,R_k}_{\#} R_k)_{k \geq 1}$ is thus $\Psi_h(R)$; hence $\Psi_h (R_k) = \overline{S}^{h,R_k}_{\#} R_k$ converges towards $\Psi_h (R)$.
\end{proof}

To prove the central limit theorem, we will need a stronger continuity result.

\begin{lemma}[Lipschitz-continuous restriction]
\label{lem:LipResticM}
For any $M' > 0$, the restriction of $\Psi_h$ to $\ps_p ( \R^d \times \C^{d'}_{M'})$ is Lipschitz-continuous.
\end{lemma}

\begin{proof}
Let us consider $R$ and $R'$ in $\ps_p(\R^d \times \C^{d'}_{M'})$. 
On any probability space $(\tilde\Omega,\tilde\F,\tilde\P)$, let $(\zeta,W)$ and $(\zeta',W')$ be respectively $R$ and $R'$-distributed variables. 
$W$ and $W'$ are $\tilde\P$-a.s. bounded by $M'$, hence using \ref{ass:coef1},
\begin{multline*}
\sup_{ih \leq t \leq (i+1) h} \lvert \overline{S}^{h,R}_t (\zeta,W) - \overline{S}^{h,R'}_t (\zeta',W') \rvert^p \leq C_{i,M'} \big[ \sup_{0 \leq t \leq i h} \lvert \overline{S}^{h,R}_t (\zeta,W) - \overline{S}^{h,R'}_t (\zeta',W') \rvert^p \\
+ \sup_{ih \leq t \leq (i+1) h} \lvert W_t - W'_t \rvert^p + \E_{\tilde\P} \sup_{0 \leq t \leq ih} \lvert \overline{S}^{h,R}_t (\zeta,W) - \overline{S}^{h,R'}_t (\zeta',W') \rvert^p \big], 
\end{multline*}
for $C_{i,M'} > 0$ that only depends on $(i,M')$.
Taking expectations, we get by induction that
\[ \E_{\tilde\P} \big[ \sup_{0 \leq t \leq T} \lvert \overline{S}^{h,R}_t (\zeta,W) - \overline{S}^{h,R'}_t (\zeta',W') \rvert^p \big] \leq L_{M'} \E_{\tilde\P} \big[ \lvert \zeta - 
\zeta' \rvert^p \big] + L_{M'} \E \big[ \sup_{0 \leq t \leq T} \lvert W_t - W'_t \rvert^p \big], \]
for some $L_{M'} > 0$.
We then use $( \overline{S}^{h,R} (\zeta,W), \overline{S}^{h,R'} (\zeta',W') )$ as a coupling to get the bound 
\[  W^p_p( \Psi_h (R), \Psi_h ( R')) = W^p_p(\overline{S}^{h,R}_\# R, \overline{S}^{h,R'}_\# R') \leq L_{M'} \E_{\tilde\P} \big[ \lvert \zeta - 
\zeta' \rvert^p \big] + L_{M'} \E_{\tilde\P} \big[ \sup_{0 \leq t \leq T} \lvert W_t - W'_t \rvert^p \big]. \]
Minimising over all possible couplings eventually gives that
\[ W^p_p( \Psi_h (R), \Psi_h ( R')) \leq L_{M'} W^p_p( R, R'), \]
concluding the proof.
\end{proof}

\begin{rem}[Inverse map]
If we assume that $d = d'$ and $\sigma_t (x,P)$ is always invertible, we can define the solution map $U^{h,P}(\gamma)$ for the ODE 
\begin{multline*} 
\forall t \in [0,T], \quad x_t = x_{\lceil t / h \rceil h - h} + \sigma^{-1}_{t_h} (\gamma,P ) \bigg[ \gamma_t - \gamma_0 - b_{t_h} ( \gamma, P ) [ t - t_h ] \\
- \sum_{i=0}^{\lceil t/h \rceil-2} h b_{ih}( \gamma, P ) + \sigma_{ih} ( \gamma_{ih}, P ) [ x_{(i+1)h} - x_{ih} ] \bigg],
\end{multline*}
with the convention $x_{-h} = 0$.
If $\sigma^{-1}$ is uniformly bounded, we can show as in Lemma \ref{lem:fixed} that $U^{h,P}_\# P$ belongs to $\ps_p ( \C^d )$ if $P$ belongs to $\ps_p ( \C^d )$. 
Moreover, by construction,
\[ \forall (x_0,\gamma) \in \R^d \times \C^d, \quad ( x_0, S^{h,P} ( x_0 , U^{h,P} ( \gamma ) ) ) = (x_0, U^{h,P} (S^{h,P} ( x_0 ,\gamma)) = (x_0,\gamma). \]
From this, similarly to \cite[Lemma 31]{Coghi2020PathwiseMT} or \cite[Proof of Theorem 3.1]{backhoff2020mean}, we can show that $\Psi_h$ is one-to-one with
\[ \forall P \in \ps_p ( \C^d), \quad \Psi_h^{-1} (P) = ( X_0, U^{h,P})_\# P, \]
for $X_0 : \C^d \rightarrow \R^d, (x_t)_{0 \leq t \leq T} \mapsto x_0$.
\end{rem}

\subsection{Mean-field limit and large deviations} \label{subsec:MhFLDP}

We now analyse the discretised particle system \eqref{eq:hNpart} using the the results of Section \ref{subsec:ConsCons}, and we prove the mean-field limit and the large deviations from \eqref{eq:GenMcK}.

\begin{proof}[Proof of Theorem \ref{thm:DisMFlimit}]
Following \cite[Section 3.1]{Coghi2020PathwiseMT}, we can see \eqref{eq:hNpart} as a special instance of the generalised McKean-Vlasov equation \eqref{eq:GenMcK}. 
Indeed, let us consider the space $( \Omega_N, \F_N, \P_N )$, where $\Omega_N := \{ 1,\ldots,N \}$, $\F_N := 2^{\Omega_N}$ is the power set of $\Omega_N$, and $\P_N := \frac{1}{N} \sum_{i=1}^N \delta_i$. 
On this space, a $N$-tuple $\vec{x}^N$ in some product space $E^N$ can be identified to the $E$-valued random variable $i \in \Omega_N \mapsto x^i \in E$.
In this setting, the empirical measure $\pi( \vec{x}^N )$ is precisely the law of $\vec{x}^N$ under $\P_N$.
Going back to \eqref{eq:hNpart}, for each $\omega \in \Omega$, we apply this to the $N$-tuples
\[ ( \vec{\zeta}^N,\vec{W}^N )(\omega) = ( \zeta^{i,N}(\omega), W^{i,N}(\omega) )_{1 \leq i \leq N} \quad \text{and} \quad \vec{X}^N(\omega) = ( X^{i,N}(\omega) )_{1 \leq i \leq N}. \]
For fixed $\omega \in \Omega$, the empirical measure $\pi(( \vec{\zeta}^N,\vec{W}^N )(\omega))$ is the law of $( \vec{\zeta}^N,\vec{W}^N )(\omega)$ under $\P_N$. 
Similarly, $\pi( \vec{X}^N (\omega))$ is the law of $\vec{X}^N (\omega)$ under $\P_N$. 
The empirical measure $\pi( \vec{X}^N(\omega))$ is precisely the one that appears in \eqref{eq:hNpart}, so that the particle system \eqref{eq:hNpart} appears as an instance of the generalised McKean-Vlasov equation \eqref{eq:GenMcK} on the probability space $(\Omega_N,\F_N,\P_N)$.
Using the notations of Proposition \ref{lem:SolMap}, we eventually get that
\[ \L( {X}^h) = \Psi_h ( \L( \zeta, {W}) ) \quad \text{and} \quad \pi( \vec{X}^{h,N} ) = \Psi_h ( \pi( \vec{\zeta}^N,\vec{W}^N ) ) \quad \P\text{-a.s.}, \]
together with, $\P$-a.s.,
\[ {X}^h = S^{h,\L({X}^h)} (\zeta,{W}) \quad \text{and} \quad X^{h,i,N} =  S^{h,\pi( \vec{X}^{h,N} )} (\zeta^i,W^i), \]
for $1 \leq i \leq N$. 
In particular, the continuity of $\Psi_h$ and Lemma \ref{lem:fixed}-\ref{item:lemfrozLip} guarantee that $\omega \mapsto \vec{X}^{h,N} (\omega)$ is measurable.
This unified formalism gives the desired results \ref{itm:thmMFexistN}-\ref{item:thmMFexist}-\ref{item:thmMFLim}, using Proposition \ref{lem:SolMap} and the continuity proved in Proposition \ref{pro:contFixed}.
\end{proof}

\begin{proof}[Proof of Proposition \ref{pro:LDPhB}]
From the Sanov theorem in the $W_p$-topology \cite[Theorem 1.1]{wang2010sanov}, the sequence of the $\L(\pi(\vec{X}^{N}_0,\vec{B}^N))$ satisfies the LDP with good rate function $R' \in \ps_p (\R^d \times \C^{d'}) \mapsto H(R' \vert R)$, where $R := \L(X^1_0,B^1)$.
By contraction \cite[Theorem 1.3.2]{dupuis2011weak}, Corollary \ref{cor:DisLD} proves that the $\L(\pi(\vec{X}^{h,N}))$ satisfies the LDP with good rate function
\[ I_h : P \mapsto \inf_{\substack{R' \in \ps_p (\R^d \times \C^{d'}) \\ \Psi_h (R') = P}} H (R' \vert R). \]
From Lemma \ref{lem:SolMap}, $\Psi_h (R') = P$ implies $\Psi_h ( R' ) = S^{h,P}_\# R'$. 
\cite[Lemma A.1]{fischer2014form} now gives that
\[ I_h (P) = H ( P \vert  S^{h,P}_\# R). \]
By definition of $\Gamma_h$, we notice that $S^{h,P}_\# R = \Gamma_h (P)$, concluding the proof.
\end{proof}

\subsection{CLT for the discretised system} \label{subsec:disCLT}

This section adapts the results of \cite[Section 2]{tanaka1984limit} and \cite[Section 5]{Coghi2020PathwiseMT} to our discretised setting.
The CLT will result from the following key-result which is a mere rewriting of \cite[Theorem 1.1]{tanaka1984limit} or \cite[Theorem 35]{Coghi2020PathwiseMT}, using the linear derivative from Definition \ref{def:Diff}.

\begin{theorem}[Tanaka] \label{thm:Tanaka}
Given a Polish space $E$, let $f : E \times \ps_p ( E) \rightarrow \R$ be a bounded function such that $P \mapsto f(x,P)$ has a linear functional derivative $y \mapsto \delta_P f (x,P,y)$ that is bounded and Lipschitz in $P$ uniformly in $(x,y)$.
Then, given any i.i.d. sequence $( X^i )_{i \geq 1}$ with common law $P \in \ps (E)$, the random variable
\[ \sqrt{N} \bigg[ \frac{1}{N} \sum_{i = 1}^N f ( X^i , \pi ( \vec{X}^N ) ) - \int_{E} f(x,P) \d P (x) \bigg]  \]
converges in law towards a centred Gaussian variable with variance $\sigma^2_f$ given by
\[ \sigma^2_f := \int_{E} \bigg[ f ( y, P ) + \int_E [ \delta_P f ( x, P, y ) - f ( x, P ) ] \d P ( x ) \bigg]^2 \d P ( y ). \]
\end{theorem}

Throughout this section, we assume that \ref{ass:coef2} holds.
All the particles now start at $0$, so that we can remove the dependence on the initial condition.
As required by Theorem \ref{thm:DisCLT}, we restrict ourselves to bounded driving paths $\gamma \in \C^{d'}_{M'}$. 
As a consequence of \ref{ass:coef1}, the paths $S^{h,P} (\gamma)$ stay in $\C^d_M$ for some $M >0$ that only depend on $(h,M')$ (see Lemma \ref{lem:fixed}-\ref{item:lemfrozLip}).
To alleviate notations, we remove all the $h$-exponents in the remainder of this section, writing
\[ \Psi : \ps_p ( \C^{d'}_{M'} ) \rightarrow \ps_p ( \C^d_M ), \]
for the fixed-point map from Proposition \ref{pro:contFixed}.
Let us fix $R \in \ps_p(\C^{d'}_{M'})$. 
We recall that for $\gamma \in \C^{d'}_{M'}$, 
\begin{equation} \label{eq:CLTdynS}
\forall t \in [0,T], \quad S^{\Psi(R)}_t ( \gamma ) = \int_0^t b_{s_h} \big( S^{\Psi (R)} ( \gamma ), S^{\Psi (R)}_\# R \big) \d s + \sigma_{s_h} \big( S^{\Psi (R)} ( \gamma ), S^{\Psi (R)}_\# R \big) \d \gamma_s.
\end{equation}
The linear derivative of $\overline{S}^R = S^{\Psi (R)}$ w.r.t. $R$ can be computed (if well-defined) writing that
\begin{equation} \label{eq:diifS}
\delta_P \overline{S}^R ( \gamma, \tilde\gamma) = \lim_{\varepsilon \rightarrow 0} \varepsilon^{-1} \big[ \overline{S}^{(1-\varepsilon)R + \varepsilon \delta_{\tilde\gamma}} ( \gamma) - \overline{S}^{R} ( \gamma) \big], 
\end{equation} 
where we recall the convention $\int_{\C^{d'}_{M'}} \delta_P \overline{S}^R ( \gamma, \gamma') \d R ( \gamma' ) = 0$ from Definition \ref{def:Diff}.
After careful differentiation of each term in \eqref{eq:CLTdynS}, this suggests that for every $t \in [0,T]$,
\begin{align*}
\delta_P \overline{S}^R_t ( \gamma, \tilde\gamma) =& \int_0^t \big[ D_x b_{s_h} \big( \overline{S}^R ( \gamma), \overline{S}^R_\# R \big) \cdot \delta_P \overline{S}^R ( \gamma, \tilde\gamma) + \delta_P b_{s_h} \big( \overline{S}^R ( \gamma), \overline{S}^R_\# R , \overline{S}^R ( \tilde\gamma) \big) \big] \d s \\
&+ \int_0^t \int_{\C^{d'}} D_y \delta_P b_{s_h} \big( \overline{S}^R ( \gamma), \overline{S}^R_\# R , \overline{S}^R ( \gamma') \big) \cdot \delta_P \overline{S}^R ( \gamma', \tilde\gamma) \, \d R ( \gamma' ) \d s \\
&+ \int_0^t \big[ D_x \sigma_{s_h} \big( \overline{S}^R ( \gamma), \overline{S}^R_\# R \big) \cdot \delta_P \overline{S}^R_t ( \gamma, \tilde\gamma) + \delta_P b_{s_h} \big( \overline{S}^R ( \gamma), \overline{S}^R_\# R , \overline{S}^R ( \tilde\gamma) \big) \big] \d \gamma_s \\
&+ \int_0^t \int_{\C^{d'}} D_y \delta_P \sigma_{s_h} \big( \overline{S}^R ( \gamma), \overline{S}^R_\# R , \overline{S}^R ( \gamma') \big) \cdot \delta_P \overline{S}^R_t ( \gamma', \tilde\gamma) \, \d R ( \gamma' ) \d \omega_s.
\end{align*}
To rigorously build $\delta_P \overline{S}^R$ we follow the approach of \cite[Section 2.1]{tanaka1984limit}.  
For $(P,x,\tilde{x},X) \in \ps_p ( \C^d_M) \times \C^d_M \times \C^d_M \times \C(\C^d_M,\C^d)$, let us define the path $g^P (x,\tilde{x},X)$ in $\C^d$ by
\begin{multline*}
g^P_t (x,\tilde{x},X) := \int_0^t \bigg[ D_x b_{s_h} ( x, P ) \cdot X (x) + \delta_P b_{s_h} ( x, P , \tilde{x} ) + \int_{\C^{d'}} D_y \delta_P b_{s_h} ( x, P , x' ) \cdot X( x' ) \, \d P ( x' ) \bigg] \d s \\
+ \int_0^t \bigg[ D_x \sigma_{s_h} ( x, P ) \cdot X (x) + \delta_P \sigma_{s_h} ( x, P , \tilde{x} ) + \int_{\C^{d'}} D_y \delta_P \sigma_{s_h} ( x, P , x' ) \cdot X( x' ) \, \d P ( x' ) \bigg] \d \gamma_s.
\end{multline*}
We then introduce an intermediary ODE.

\begin{lemma}[Differential map] \label{lem:difMap} 
We assume \ref{ass:coef2}.
\begin{enumerate}[label=(\roman*),ref=(\roman*)]
    \item\label{item:lemDifexist}  For every $(P,\tilde{x}) \in \ps_p ( \C^d_M ) \times \C^{d}_M$, the equation 
    \[ \forall x \in \C^d_M, \quad U^P_t ( x , \tilde{x} ) = g^P_t ( x , \tilde{x}, U^P (\cdot,\tilde{x}) ), \]
    has unique solution $x \mapsto U^P_t ( x , \tilde{x} )$ in $\C ( \C^d_M , \C^d )$.
    \item\label{item:lemDifLip} The map $(P,x,\tilde{x}) \mapsto U^P_t ( x , \tilde{x} )$ is globally bounded and Lipschitz-continuous. 
\end{enumerate}
\end{lemma}

The global bound on $U^P ( x, \tilde{x})$ is only possible because we restricted ourselves to $\C^d_M$.

\begin{proof}
These result are the analogous of \cite[Lemma 2.2]{tanaka1984limit} in a simpler setting (we consider  discretised dynamics instead of ODEs).
Alternatively, \ref{item:lemDifexist} can be obtained by readily adapting the induction in the proof of Lemma \ref{lem:SolMap}, while \ref{item:lemDifLip} is a direct adaptation of the induction in the proof of Lemma \ref{lem:LipResticM}.
\end{proof}

We can now define $\delta_P \overline{S}^R$ by setting 
\begin{equation} \label{eq:defdif}
\delta_P \overline{S}^R ( \gamma, \tilde\gamma) := U^{ \overline{S}^R_\# R} \big( \overline{S}^R ( \gamma ) , \overline{S}^R ( \tilde\gamma ) \big). 
\end{equation} 
We notice that $\delta_P \overline{S}^R$ defined by \eqref{eq:defdif} indeed satisfies the equation computed above.
From this and \eqref{eq:CLTdynS}, it is a standard ODE-like computation to verify that \eqref{eq:diifS} indeed holds (see e.g. \cite[Lemma 42]{Coghi2020PathwiseMT}).

\begin{corollary}[Lispchitz-continuity] \label{cor:LipDiff}
Under \ref{ass:coef2}, the map $(R, \gamma, \tilde\gamma) \mapsto \delta_P \overline{S}^R ( \gamma, \tilde\gamma )$ is globally bounded and Lipschitz-continuous on $\ps_p (\C^{d'}_{M'}) \times \C^{d'}_{M'} \times \C^{d'}_{M'}$.
\end{corollary}

\begin{proof}
From \eqref{eq:defdif}, this result is a mere concatenation of Lemma \ref{lem:fixed}-\ref{item:lemfrozLip}, Lemma \ref{lem:LipResticM} and Lemma \ref{lem:difMap}-\ref{item:lemDifLip}.
\end{proof}

The following result is a detailed version of Theorem \ref{thm:DisCLT}, obtained as a consequence of Theorem \ref{thm:Tanaka}.

\begin{corollary} \label{cor:hMCLT} 
Under \ref{ass:coef2}, let $(W^i)_{i \geq 1}$ be an i.i.d. sequence of $\C^{d'}_{M'}$-valued variables with common law $R$. 
Let $\vec{X}^{h,N}$ be the related system of particles starting from $0$ given by \eqref{eq:hNpart}, and let $X^h$ be a related solution of \eqref{eq:GenMcK}.
Then, for every $\varphi \in \C^{1,1}_b ( \C^{d}, \R )$, the random variable
\[ \sqrt{N} \bigg[ \frac{1}{N} \sum_{i =1}^N \varphi ( X^{h,i,N} ) - \E [ \varphi( X^h ) ] \bigg] \]
converges in law towards a centred Gaussian variable with variance $\sigma^2_\varphi$ given by 
\[ \sigma^2_\varphi := \int_{\C^{d'}_{M'}} \bigg[ \varphi(\overline{S}^R ( \tilde\gamma )) + \int_{\C^{d'}_{M'}} [ D \varphi (\overline{S}^R ( \gamma )) \cdot \delta_P \overline{S}^R (\gamma,\tilde\gamma) - \varphi(\overline{S}^R ( \gamma )) ] \d R ( \gamma ) \bigg]^2 \d R ( \tilde\gamma ). \]
\end{corollary}
\begin{proof}
The function $f : \C^{d'}_{M'} \times \ps_p ( \C^{d'}_{M'} ) \rightarrow \R$ defined by $f(\gamma,R) := \varphi(\overline{S}^R(\gamma))$ has a linear derivative given by $\delta_P f (\gamma,R,\tilde\gamma) = D \varphi ( \overline{S}^R(\gamma) ) \cdot \delta_P \overline{S}^R(\gamma,\tilde\gamma)$.
Using Lemma \ref{lem:LipResticM} and Corollary \ref{cor:LipDiff}, $f$ satisfies the assumptions of Theorem \ref{thm:Tanaka}, which now gives the result.
\end{proof}

\subsection{Extension to the Brownian setting} \label{subsec:disCLTB}

Throughout this section, we assume that \ref{ass:pair} holds.
Our purpose is to extend the results of the previous section for proving Proposition \ref{pro:ClThB}.
Let $(B^{i})_{i \geq 1}$ be a countable sequence in $\C^{d'}$ of i.i.d. Brownian motions.
We rely on an approximation procedure by stopping the $B^i$ when they reach a given threshold, and by showing that the fluctuations of the related particle system can be uniformly controlled.
For $M >0$, we introduce the stopping time
\[ \tau^i_M := \inf \{ t \in [0,T], \; \vert B^i_t \rvert \geq M \}, \]
with the convention $\tau^i_M = T$ if the set on the r.h.s. is empty.
The related particle system given by \eqref{eq:hNpart} is
\[ \d X^{h,M,i,N}_t = b_{t_h} ( X^{h,M,i,N} , \pi ( \vec{X}^{h,M,N} ) ) \d t + \sigma_{t_h} ( X^{h,M,i,N} , \pi ( \vec{X}^{h,M,N} ) ) \d B^i_{t \wedge \tau^i_M}, \]
with $X^{h,M,i,N}_0 = 0$, $1 \leq i \leq N$.
Since we discretised with the time step $h >0$, we recall that stochastic integration is not needed to make sense of the above SDE. 
The limit particle system \eqref{eq:hNpart} when $\zeta^i = 0$ and $W^i = B^i$ corresponds to $\vec{X}^{h,\infty,N}$.
For each $1 \leq i \leq N$,
let us introduce the McKean-Vlasov process $\overline{X}^{h,M,i}$ solution of
\[ \d \overline{X}^{h,M,i}_t = b_{t_h} ( \overline{X}^{h,M,i} , \L ( \overline{X}^{h,M,i} ) ) \d t + \sigma_{t_h} ( \overline{X}^{h,M,i} , \L ( \overline{X}^{h,M,i} ) ) \d B^i_{t \wedge \tau^i_M}, \quad \overline{X}^{h,M,i}_0 = 0, \]
which corresponds to \eqref{eq:GenMcK} for the stopped Brownian motions. 
Let us introduce the difference process $\delta^{h,M,i,N} := X^{h,M,i,N} - \overline{X}^{h,M,i}$, together with $\delta^{h,\infty,i,N} := X^{h,\infty,i,N} - \overline{X}^{h,\infty,i}$ when $M = +\infty$.
We first need that the mean-field limit holds uniformly in $M$, and that the obtained approximation is uniform in $N$.

\begin{lemma}[Uniform limit] \label{lem:PoCM}
There exists $C_h > 0$ such that for every $N \geq 1$, $M >0$, $1 \leq i \leq N$,
\begin{enumerate}[label=(\roman*),ref=(\roman*)]
\item\label{item:lemPoCMbound} $\E [ \sup_{0 \leq t \leq T} \vert X^{h,M,i,N}_t \vert^2 ] \leq C_h$.
\item\label{item:lemPoCMPoCM} $\E [ \sup_{0 \leq t \leq T} \vert \delta^{h,M,i,N}_t \vert^2 ] \leq C_h N^{-1}$.
\item\label{item:lemPoCMPoCinf} $\E [ \sup_{0 \leq t \leq T} \vert \delta^{h,\infty,i,N}_t \vert^2 ] \leq C_h N^{-1}$.
\item\label{item:lemPoCMMinf} $\E [ \sup_{0 \leq t \leq T} \vert X^{h,M,i,N}_t - X^{h,\infty,i,N}_t \vert^2 ] \leq C_h \E [ \1_{\tau^1_M < T} \sup_{0 \leq t \leq T} \vert B^1_t \vert^2 ]$.
\item\label{item:lemPoCMinfinf} $\E [ \sup_{0 \leq t \leq T} \vert \overline{X}^{h,M,i}_t - \overline{X}^{h,\infty,i}_t \vert^2 ] \leq C_h \E [ \1_{\tau^1_M < T} \sup_{0 \leq t \leq T} \vert B^1_t \vert^2 ]$.
\end{enumerate}
\end{lemma}

We then strengthen these results for controlling the fluctuations.

\begin{proposition} \label{pro:approxhM} There exists $C_h >0$ such that for every $M >0$,
\[ \sup_{N \geq 1} \frac{1}{\sqrt{N}} \sum_{i=1}^N \E \bigg[ \sup_{0 \leq t \leq T} \big\lvert \delta^{h,M,i,N}_t - \delta^{h,\infty,i,N}_t \big\rvert \bigg] \leq C_h \E^{1/2} \big[ \1_{\tau^1_M < T} \sup_{0 \leq t \leq T} \vert B^1_t \vert^2 \big]. \]
\end{proposition}

With these results at hand, Proposition \ref{pro:ClThB} will easily follow.
The proofs of the approximation results rely on the classical coupling method from \cite[Theorem 1.4]{sznitman1991topics}.  
To alleviate notations, we drop the exponents $(h,N)$ in the remainder of this section (although the dependence on $N$ is crucial). 
We will repeatedly use the technical estimates proved in Appendix \ref{app:tech}.
In the following, $C_h$ is a generic constant that may change from line to line, but staying independent of $(M,N)$.

\begin{proof}[Proof of Lemma \ref{lem:PoCM}]
\ref{item:lemPoCMbound} By symmetry, all the $X^{M,i}$ have the same law.
The bound is then a straight-forward induction very close to \textbf{\emph{Step 1.}} in the proof of Proposition \ref{pro:contFixed}, using that $\E [ \sup_{0 \leq t \leq T} \vert B^i_{t \wedge \tau^i_M} \vert^2 ]$ is bounded uniformly in $M$. 

\ref{item:lemPoCMPoCM} By definition of $X^{M,i}$ and $\overline{X}^{M,i}$, for $0 \leq t \leq T$,
\begin{multline*}
\delta^{M,i}_t = \int_0^t [ b_{s_h} ( X^{M,i} , \pi(\vec{X}^{M}) ) - b_{s_h} ( \overline{X}^{M,i} , \L(\overline{X}^{M,i}) ) ] \d s \\
+ \int_0^t [ \sigma_{s_h} ( X^{M,i} , \pi(\vec{X}^{M}) ) - \sigma_{s_h} ( \overline{X}^{M,i} , \L(\overline{X}^{M,i}) ) ] \d B^i_{s \wedge \tau^i_M}.
\end{multline*}   
We then take the square of this expression and we use Jensen's inequality, before taking the supremum in time and expectations to get
\begin{multline} \label{eq:interPocMM}
\E \big[ \sup_{0 \leq s \leq t} \vert \delta^{M,i}_t \vert^2 \big] \leq 2 t \int_0^t \E \vert b_{s_h} ( X^{M,i} , \pi(\vec{X}^{M}) ) - b_{s_h} ( \overline{X}^{M,i} , \L(\overline{X}^{M,i}) ) \vert^2 \d s \\
+ 2 \E \sup_{0 \leq s \leq t} \bigg\vert \int_0^s [ \sigma_{r_h} ( X^{M,i} , \pi(\vec{X}^{M}) ) - \sigma_{r_h} ( \overline{X}^{M,i} , \L(\overline{X}^{M,i}) ) ] \d B^i_{r \wedge \tau^i_M} \bigg\vert^2.
\end{multline} 
For the drift term, we split
\begin{multline*}
b_{s_h} ( X^{M,i} , \pi(\vec{X}^{M}) ) - b_{s_h} ( \overline{X}^{M,i} , \L(\overline{X}^{M,i}) ) = b_{s_h} ( X^{M,i} , \pi(\vec{X}^{M}) ) - b_{s_h} ( \overline{X}^{M,i} , \pi(\vec{\overline{X}}^{M}) ) \\
+ [ b_{s_h} ( \overline{X}^{M,i} , \pi(\vec{\overline{X}}^{M}) ) - b_{s_h} ( \overline{X}^{M,i} , \L(\overline{X}^{M,i}) ) ].
\end{multline*}
Using \ref{item:lemPoCMbound}, \ref{ass:pair}, and Lemma \ref{lem:techiid} with the i.i.d. variables $(\overline{X}^i,\overline{Y}^i) = (\overline{X}^{M,i},0)$, we get that
\[ \E \big\vert b_{s_h} ( \overline{X}^{M,i} , \pi(\vec{\overline{X}}^{M}) ) - b_{s_h} ( \overline{X}^{M,i} , \L(\overline{X}^{M,i}) ) \big\vert^2 \leq C_h N^{-1}. \]
Further using the Lipschitz assumption \ref{ass:pair} on $b$, Jensen's inequality, and the fact that $((X^{M,j},\overline{X}^{M,j}))_{1 \leq j \leq M}$ is identically distributed, we deduce that
\begin{equation*}
\E \big\vert b_{s_h} ( X^{M,i} , \pi(\vec{X}^{M}) ) - b_{s_h} ( \overline{X}^{M,i} , \L ( \overline{X}^{M,i}) ) \big\vert^2 \leq C_h \E \big[ \sup_{0 \leq r \leq s_h} \vert X^{M,i}_r - \overline{X}^{M,i}_r \vert^2 \big] + C_h N^{-1}.
\end{equation*} 
To handle the second line in \eqref{eq:interPocMM}, we rely on stochastic integration. 
From the optional stopping theorem, $( B^i_{t \wedge \tau^i_M} )_{0 \leq t \leq T}$ is a square-integrable martingale; its quadratic variation is the process $( t \wedge \tau^i_M )_{0 \leq t \leq T}$.
The Burkholder-Davis-Gundy (BDG) inequality then yields
\begin{align*}
\E \sup_{0 \leq s \leq t} \bigg\vert \int_0^s &[ \sigma_{r_h} ( X^{M,i} , \pi(\vec{X}^{M}) ) - \sigma_{r_h} ( \overline{X}^{M,i} , \L(\overline{X}^{M,i}) ) ] \d B^i_{r \wedge \tau^i_M} \bigg\vert^2 \\
&\leq C_h \E \int_0^{t \wedge \tau^i_M} \vert \sigma_{s_h} ( X^{M,i} , \pi(\vec{X}^{M}) ) - \sigma_{s_h} ( \overline{X}^{M,i} , \L(\overline{X}^{M,i}) ) \vert^2 \d s \\
&\leq C_h \int_0^t \E \vert \sigma_{s_h} ( X^{M,i} , \pi(\vec{X}^{M}) ) - \sigma_{s_h} ( \overline{X}^{M,i} , \L(\overline{X}^{M,i}) ) \vert^2 \d s. 
\end{align*}
We can now handle the r.h.s. as we did for the drift term.
Gathering the terms in \eqref{eq:interPocMM},
\[ \E \big[ \sup_{0 \leq s \leq t} \vert \delta^{M,i}_t \vert^2 \big] \leq C_h N^{-1} + C_h \int_0^t \E \big[ \sup_{0 \leq r \leq s} \vert \delta^{M,i}_r \vert^2 \big] \d s. \]
The conclusion then follows from the Gronwall Lemma.

\ref{item:lemPoCMPoCinf} If $\tau_M^i = T$ in the above proof of \ref{item:lemPoCMPoCM}, we notice that all the computations remain valid when replacing $(B^i_{t \wedge \tau_M^i})_{0 \leq t \leq T}$ by $(B^i_t)_{0 \leq t \leq T}$.
This yields the result.

\ref{item:lemPoCMMinf} As for proving \ref{item:lemPoCMPoCM}, for $0 \leq t \leq T$,
\begin{multline*}
X^{M,i}_t - X^{\infty,i}_t = \int_0^t [ \sigma_{s_h} ( X^{M,i} , \pi(\vec{X}^{M}) ) - \sigma_{s_h} ( X^{\infty,i} , \pi(\vec{X}^{\infty}) ) ] \d B^i_{s \wedge \tau^i_M} \\
- \int_0^t \sigma_{s_h} ( X^{\infty,i} , \pi(\vec{X}^{\infty}) ) ] \d [ B^i_s - B^i_{s \wedge \tau^i_M} ] + \int_0^t [ b_{s_h} ( X^{M,i} , \pi(\vec{X}^{M}) ) - b_{s_h} ( X^{\infty,i} , \pi(\vec{X}^{\infty}) ) ] \d s.
\end{multline*} 
Setting $B^{i,M}_s := B^i_s - B^i_{\tau^i_M}$, the second integral on the r.h.s. is bounded by
\begin{equation} \label{eq:QuadVarMh}
M_\sigma \big\vert \1_{\tau^i_M < t} B^{i,M}_{t} - \1_{\tau^i_M < t_h} B^{i,M}_{t_h} \big\vert + \sum_{j=0}^{\lfloor t/h \rfloor-1} M_\sigma \big\vert \1_{\tau^i_M < (j + 1) h} B^{i,M}_{(j+1)h}- \1_{\tau^i_M < j h} B^{i,M}_{j h} \big\vert, 
\end{equation} 
and further by $C_h \1_{\tau^i_M < T} \sup_{0 \leq s \leq t} \vert B^i_s \vert$. 
The result then follows by reproducing the estimates and the Gronwall argument in the proof of \ref{item:lemPoCMPoCM}.

\ref{item:lemPoCMinfinf} By symmetry, we recall that $\E [ \sup_{0 \leq t \leq T} \vert X^{h,M,i,N}_t - X^{h,\infty,i,N}_t \vert^2 ]$ does not depend on $i$. 
The result follows by taking the $N \rightarrow +\infty$ limit in the uniform estimate \ref{item:lemPoCMMinf}, using \ref{item:lemPoCMPoCM}-\ref{item:lemPoCMPoCinf}.
\end{proof}

\begin{proof}[Proof of Proposition \ref{pro:approxhM}]
For every $0 \leq t \leq t' \leq T$, we have the decomposition
\begin{equation*} 
\delta^{M,i}_{t'} - \delta^{\infty,i}_{t'} = \delta^{M,i}_t - \delta^{\infty,i}_t + \int_t^{t'} b_s^{M,i} \d s + \int_t^{t'} \sigma_s^{M,i} \d B^i_s + \int_t^{t'} \overline{\sigma}^{M,i}_s \d [ B^i_{s \wedge \tau^i_M} - B^i_s ], 
\end{equation*} 
where
\[ b_s^{M,i} := b_{s_h} ( X^{M,i}, \pi( \vec{X}^{M} ) ) - b_{s_h} ( \overline{X}^{M,i}, \L(\overline{X}^{M,i}) ) - [ b_{s_h} ( X^{\infty,i}, \pi( \vec{X}^{\infty} ) ) - b_{s_h} ( \overline{X}^{\infty,i}, \L(\overline{X}^{\infty,i}) ) ],  \]
\[ \sigma_s^{M,i} := \sigma_{s_h} ( X^{M,i}, \pi( \vec{X}^{M} ) ) - \sigma_{s_h} ( \overline{X}^{M,i}, \L(\overline{X}^{M,i}) ) - [ \sigma_{s_h} ( X^{\infty,i}, \pi( \vec{X}^{\infty} ) ) - \sigma_{s_h} ( \overline{X}^{\infty,i}, \L(\overline{X}^{\infty,i}) ) ], \]
\[ \overline{\sigma}^{M,i}_s := \sigma_{s_h} (X^{M,i}, \pi( \vec{X}^{M})) - \sigma_{s_h} (\overline{X}^{M,i}, \L(\overline{X}^{M,i}) ).  \]
Taking absolute values, supremum in time and expectations gives 
\begin{multline}  \label{eq:decompFluctM}
\E \big[ \sup_{t \leq s \leq t'} \vert \delta^{M,i}_s - \delta^{\infty,i}_s \vert \big] \leq \E \big[ \vert \delta^{M,i}_t - \delta^{\infty,i}_t \vert \big] + \int_t^{t'} \E[  \vert b^{M,i}_s \vert ] \d s + \E \sup_{t \leq s \leq t'} \bigg\vert \int_t^s \sigma^{M,i}_r \d B^i_r \bigg\vert \\
+ \E \sup_{t \leq s \leq t'} \bigg\vert \int_t^s \overline{\sigma}^{M,i}_r \d [B^i_{r \wedge \tau^i_M} - B^i_r] \bigg\vert. 
\end{multline} 
In the following, $C_h$ is a generic constant, which may change from line to line but staying \emph{independent of $(M,N)$ and $(t,t')$}.
To handle the $b^{M,i}_s$-term, we use the splitting
\begin{align*}
b_s^{M,i} = \; &b_{s_h} ( X^{M,i}, \pi( \vec{X}^{M} ) ) - b_{s_h} ( \overline{X}^{M,i}, \pi( \vec{\overline{X}}^{M} ) ) - [ b_{s_h} ( X^{\infty,i}, \pi( \vec{X}^{\infty} ) ) - b_{s_h} ( \overline{X}^{\infty,i}, \pi( \vec{\overline{X}}^{\infty} ) ) ] \\
+\, &b_{s_h} ( \overline{X}^{M,i}, \pi ( \vec{\overline{X}}^{M}) ) - b_{s_h} ( \overline{X}^{M,i}, \L(\overline{X}^{M,i}) ) - [ b_{s_h} ( \overline{X}^{\infty,i}, \pi (\vec{\overline{X}}^{\infty}) ) - b_{s_h} ( \overline{X}^{\infty,i}, \L(\overline{X}^{\infty,i}) ) ].
\end{align*}
Let $\mathbf{b}^1_s$ (resp. $\mathbf{b}^2_s$) denote the first (resp. the second) line. 
To control $\mathbf{b}^1_s$, we use \ref{ass:pair} and Lemma \ref{lem:techC11} applied to $(X^i,Y^i) = (X^{M,i},X^{\infty,i})$ and $(\overline{X}^i,\overline{Y}^i) = (\overline{X}^{M,i}_s,\overline{X}^{\infty,i}_s)$ to get
\begin{multline*}
 \vert \mathbf{b}^1_s \vert \leq \frac{1}{N}\sum_{j=1}^N \Vert \tilde{b}_{s_h} \Vert_\mathrm{Lip} \big[ \sup_{0 \leq r \leq s} \vert \delta^{M,i}_r -\delta^{\infty,i}_r \vert + \sup_{0 \leq r \leq s} \vert \delta^{M,j}_r -\delta^{\infty,j}_r \vert \big] \\
 + \Vert D\tilde{b}_{s_h} \Vert_\mathrm{Lip} \, R^{i,j}_s \big[ \sup_{0 \leq r \leq s} \vert \delta^{\infty,i}_r \vert + \sup_{0 \leq r \leq s} \vert \delta^{\infty,j}_r \vert \big],
\end{multline*}
where
\[ R^{i,j}_s := \sup_{0 \leq r \leq s} \vert X^{M,i}_r -X^{\infty,i}_r \vert + \sup_{0 \leq r \leq s} \vert \overline{X}^{M,i}_r -\overline{X}^{\infty,i}_r \vert + \sup_{0 \leq r \leq s} \vert X^{M,j}_r -X^{\infty,j}_r \vert + \sup_{0 \leq r \leq s} \vert \overline X^{M,j} _r-\overline X^{\infty,j}_r \vert
\]
We then take supremum in time and expectations, and we use the Cauchy-Schwarz inequality and the bounds from Lemma \ref{lem:PoCM} to get 
\begin{multline*}
\E \big[ \sup_{t \leq r \leq s} \vert \mathbf{b}^1_r \vert \big] \leq C_h \E \big[ \sup_{0 \leq r \leq s} \vert \delta^{M,i}_r - \delta^{\infty,i}_r \vert \big] + \frac{C_h}{N} \sum_{j=1}^N \E \big[ \sup_{0 \leq r \leq s} \vert \delta^{M,j}_r - \delta^{\infty,j}_r  \vert \big] \\
+ \frac{C_h}{\sqrt{N}} \E^{1/2} \big[ \1_{\tau^1_M < T} \sup_{0 \leq r \leq T} \vert B^1_r \vert^2 \big]. 
\end{multline*} 
We recall that $( \delta^{M,j} )_{1 \leq j \leq M}$ is identically distributed.
The supremum in time was not needed to estimate $\E[ \vert \mathbf{b}^1_s \vert ]$, but it will be useful for performing the same computation with the $\boldsymbol{\sigma}^1_s$-term below.
We control $\mathbf{b}^2_s$ using Lemma \ref{lem:techiid} applied to the i.i.d. processes $(\overline{X}^j,\overline{Y}^j) = (\overline{X}^{M,j},\overline{X}^{\infty,j})$, $1 \leq j \leq N$.
From the bounds in Lemma \ref{lem:PoCM},
\[ \sup_{0 \leq r \leq s} \E [ \vert \mathbf{b}^2_r \vert^2 ] \leq C_h N^{-1} \E \big[ \1_{\tau^1_M < T} \sup_{0 \leq r \leq T} \vert B^1_r \vert^2 \big]. \]
For the $\overline{\sigma}^{M,i}_s$-term in \eqref{eq:decompFluctM}, we bound the integral using \eqref{eq:QuadVarMh} with $\sup_{t \leq s \leq t'} \vert \overline{\sigma}^{M,i}_s \vert$ instead of $M_\sigma$. We then split
\begin{multline*}
\E \big[ \sup_{t \leq s \leq t'} \vert \overline{\sigma}^{M,i}_s \vert^2 \big] \leq 2 \E \big[ \sup_{0 \leq s \leq T} \vert \sigma_{s_h} (X^{M,i}, \pi( \vec{X}^{M})) - \sigma_{s_h} (\overline{X}^{M,i}, \pi( \vec{\overline{X}}^{M}) ) \vert^2 \big] + \\
2 \E \big[ \sup_{0 \leq s \leq T} \vert \sigma_{s_h} (\overline{X}^{M,i}, \pi( \vec{\overline{X}}^{M}) ) - \sigma_{s_h} (\overline{X}^{M,i}, \L(\overline{X}^{M,i}) ) \vert^2 \big].  
\end{multline*} 
Using \ref{ass:coef1} and Lemma \ref{lem:PoCM}-\ref{item:lemPoCMPoCM}, the first term is bounded by $C_h N^{-1}$. 
From \ref{ass:pair} and Lemma \ref{lem:techiid}, the same bound holds for the second term.
From the Cauchy-Schwarz inequality, 
\[ \E \sup_{t \leq s \leq t'} \bigg\vert \int_t^s \overline{\sigma}^{M,i}_r \d [ B^i_{r \wedge \tau^i_M} - B^i_r ] \bigg\vert \leq \frac{C_h}{\sqrt{N}} \E^{1/2} \big[ \1_{\tau^1_M < T} \sup_{0 \leq r \leq T} \vert B^1_r \vert^2 \big]. \]
For the $\sigma^{M,i}_{s}$-term in \eqref{eq:decompFluctM}, we split $\sigma^{M,i}_{s} = \boldsymbol{\sigma}^{1}_{s} + \boldsymbol{\sigma}^{2}_{s}$ as we did for $b^{M,i}_s$.
Using the BDG inequality, 
\[ \E \sup_{t \leq s \leq t'} \bigg\vert \int_t^s \boldsymbol{\sigma}^{1}_r \d B^i_r \bigg\vert \leq C_h \E \bigg[ \bigg( \int_t^{t'} \vert \boldsymbol{\sigma}^{1}_s \vert^2 \d s \bigg)^{1/2} \bigg] \leq C_h (t'-t)^{1/2} \E \big[ \sup_{t \leq s \leq t'} \vert \boldsymbol{\sigma}^{1}_{s} \vert \big]. \]
We then handle $\E \big[ \sup_{t \leq s \leq t'} \vert \boldsymbol{\sigma}^{1}_{s} \vert \big]$ as we did for $\E \big[ \sup_{t \leq s \leq t'} \vert \mathbf{b}^{1}_{s} \vert \big]$.
Using the BDG inequality and Jensen's inequality, 
\[ \E \sup_{t \leq s \leq t'} \bigg\vert \int_0^s \boldsymbol{\sigma}^{2}_r \d B^i_r \bigg\vert \leq C_h \bigg( \int_t^{t'} \E [ \vert \boldsymbol{\sigma}^{2}_{s} \vert^2 ] \d s \bigg)^{1/2}, \]
and we handle $\E [ \vert \boldsymbol{\sigma}^{2}_{s} \vert^2 ]$ as we did for $\E [ \vert \mathbf{b}^{2}_{s} \vert^2 ]$.
Let us gather all the terms in \eqref{eq:decompFluctM}, using that 
\begin{equation} \label{eq:subAdd}
\forall s \in [t,t'], \quad \sup_{0 \leq r \leq s} \vert \delta^{M,i}_r - \delta^{\infty,i}_r \vert \leq \sup_{0 \leq r \leq t} \vert \delta^{M,i}_r - \delta^{\infty,i}_r \vert + \sup_{t \leq r \leq s} \vert \delta^{M,i}_r - \delta^{\infty,i}_r \vert. 
\end{equation} 
Setting $f^{M}(t,t') := \sqrt{N} \, \E [ \sup_{t \leq s \leq t'} \vert \delta^{M,i}_s - \delta^{\infty,i}_s \vert ]$, for any $0 \leq t \leq t' \leq T$, 
\[ f^{M} (t,t') \leq C_h f^{M} (0,t) + C_h \E^{1/2} \big[ \1_{\tau^1_M < T} \sup_{0 \leq s \leq T} \vert B^1_s \vert^2 \big] + C_h (t'-t)^{1/2} f^M (t,t') + C_h \int_t^{t'} f^{M} (t,s) \d s, \]
for a constant $C_h$ that does not depend on $(t,t')$.
Using \eqref{eq:subAdd}, Lemma \ref{lem:Gronwall} now gives the desired bound on $f^M (0,T)$.
\end{proof}

\begin{proof}[Proof of Proposition \ref{pro:ClThB}]
Let us fix $\varphi$ in $\C^{1,1}(\C^d,\R)$, before defining
\[ \delta^{M}_\varphi := N^{-1} {\textstyle\sum}_{i=1}^N \big[ \varphi ( X^{M,i} ) - \E [ \varphi ( \overline{X}^{M,i} ) ] \big]. \]
We then decompose
\begin{align*}
\delta^{M}_\varphi - \delta^{\infty}_\varphi =& \, N^{-1} {\textstyle\sum}_{i=1}^N \big[ \varphi ( X^{M,i} ) - \varphi ( \overline{X}^{M,i} ) \big] - N^{-1} {\textstyle\sum}_{i=1}^N \big[ \varphi ( X^{\infty,i} ) - \varphi ( \overline{X}^{\infty,i} ) \big]  \\
&\,+ N^{-1} {\textstyle\sum}_{i=1}^N \big[ \varphi ( \overline{X}^{M,i} ) - \E [ \varphi ( \overline{X}^{M,i} ) ] \big] - N^{-1} {\textstyle\sum}_{i=1}^N \big[ \varphi ( \overline{X}^{\infty,i} ) - \E [ \varphi ( \overline{X}^{\infty,i} ) ] \big].
\end{align*} 
As previously, we control the first line using Lemma \ref{lem:techC11}, Lemma \ref{lem:PoCM} and Proposition \ref{pro:approxhM}.
Similarly, we handle the second line using Lemma \ref{lem:techiid} and Lemma \ref{lem:PoCM}.
At the end of the day, we obtain that
\[ \sup_{N \geq 1} {\sqrt{N}} \, \E \big[ \lvert \delta^{M}_\varphi - \delta^{\infty}_\varphi \rvert \big] \xrightarrow[M \rightarrow + \infty]{} 0. \]
As a consequence, the convergence given by Corollary \ref{cor:hMCLT} holds uniformly in $M$, with variance $\sigma^2_{M,\varphi}$. 
Let $(B_t)_{0 \leq t \leq T}$, $(\tilde{B}_t)_{0 \leq t \leq T}$ be independent Brownian motions in $\R^{d'}$.
Let $X^M$ denote the strong solution of the discretised McKean-Vlasov \eqref{eq:GenMcK} with driving noise $( {B}_{t \wedge \tau_M})_{0 \leq t \leq T}$. 
Similarly, let $\tilde{X}^M$ denote the strong solution of \eqref{eq:GenMcK} with driving noise $( \tilde{B}_{t \wedge \tilde\tau_M})_{0 \leq t \leq T}$, the definition of $\tilde{\tau}_M$ being straightforward. 
Let us then introduce
the solution $\delta X^M$ of the McKean-Vlasov SDE with common noise
\begin{align*} \label{eq:McKNoisehM}
\d \delta X^M_{t_h} &= \big[ D_x \sigma_{t_h} ( X^M, \L ( X^M ) ) \cdot \delta X^M + \delta_P \sigma_{t_h} ( X^M, \L ( X^M ), \tilde{X}^h ) \big] \d B_{t \wedge \tau_M} \\
&+ \big[ D_x b_{t_h} ( X^M, \L ( X^M ) ) \cdot \delta X^M + \delta_P b_{t_h} ( X^M, \L ( X^M ), \tilde{X}^M ) \big] \d t \\
&+ \E \big[ D_y \delta_P b_{t_h} ( X^M , \L ( X^M ), X^M ) \cdot  \delta X^M \big\vert ( \tilde{B}_s )_{0 \leq s \leq t_h} \big] \d t \\
&+ \E \big[ D_y \delta_P \sigma_{t_h} ( X^M, \L ( X^M ), X^M ) \cdot  \delta X^M \big\vert ( \tilde{B}_s )_{0 \leq s \leq t_h} \big] \d B_{t \wedge \tau_M}, \qquad \delta X^M_0 = 0.\phantom{abcd}
\end{align*}
Well-posedness for the above equation is given by Lemma \ref{lem:difMap}. 
This can be seen as a discretised version of a McKean-Vlasov SDE with common noise. 
From Corollary \ref{cor:hMCLT} and the equation \eqref{eq:defdif} satisfied by $\overline{S}^{h,M,R}$, we deduce that
\[ \sigma^2_{M,\varphi} := \E \big\{ [ \varphi ( \tilde{X}^M ) + \E [ D \varphi ( X^M ) \cdot \delta X^M - \varphi ( X^M ) \vert \tilde{B} ] ]^2 \big\}. \]
Using synchronous coupling and a Gronwall argument (as in the above proofs), we can show that
\[ \E[ \sup_{0 \leq t \leq T} \vert \delta X^{M}_t - \delta X^{h}_t \vert ] \xrightarrow[M \rightarrow +\infty]{} 0, \]
where $\delta X^h$ is the solution of \eqref{eq:McKNoiseh} with the same $(B,\tilde{B})$. 
From this, we can deduce that $\sigma^2_{M,\varphi}$ converges to $\sigma^2_{h,\varphi}$ as required by Proposition \ref{pro:ClThB}.
This completes the proof.
\end{proof}

\section{Extension to the continuous setting} \label{sec:Conti}

\subsection{Large deviations and stochastic control} \label{subsec:contLDP}

Throughout this section, we assume \ref{ass:coef1}-\ref{ass:IniExp} and $p \in [1,2)$.
We recall that the notion of reference system $\Sigma = (\Omega,(\F_t)_{t \leq 0 \leq T}, \P, ( \vec{B}^N_t)_{0 \leq t \leq T})$ is defined in Section \ref{subsec:not}.
For such a $\Sigma$, $h \in [0,1]$ and $N \geq 1$, we define the controlled system $\vec{X}^{h,N,\vec{u}^N}$ by 
\begin{multline} \label{eq:NControlled}
\d X^{h,i,N,\vec{u}^N}_t = b_{t_h} ( X^{h,i,N,\vec{u}^N}, \pi( \vec{X}^{h,N,\vec{u}^N} ) ) \d t + \sigma_{t_h} ( X^{h,i,N,\vec{u}^N}, \pi( \vec{X}^{h,N,\vec{u}^N} ) ) u^{i,N}_t \d t \\
+ \sigma_{t_h} ( X^{h,i,N,\vec{u}^N}, \pi( \vec{X}^{h,N,\vec{u}^N} ) ) \d B^i_t, \quad \vec{X}^{h,i,N,\vec{u}^N} = X^{h,i,N}_0, \quad 1 \leq i \leq N,
\end{multline}
where $\vec{X}^{h,N}_0 \in L^p_\P(\Omega,\R^d)$ and $\vec{u}^N = (\vec{u}^N_t)_{0 \leq t \leq T}$ is any $(\R^{d'})^N$-valued square-integrable progressively measurable process. 
Such a process will be called an \emph{admissible control}.
Following \cite[Section 3, page 9]{budhiraja2012large}, strong existence and pathwise uniqueness for \eqref{eq:NControlled} is given by the Girsanov transform if $\int_0^T \vert \vec{u}^N_t \rvert^2 \d t \leq M$ a.s. for some $M >0$, and still holds otherwise using a localisation argument.

Let $F : \ps_p ( \C^d ) \rightarrow \R$ be a bounded and Lipschitz-continuous function. 
Using the notations
\[ \mathcal{E} ( \vec{u}^N ) := \frac{1}{N} \sum_{i = 1}^N \mathcal{E} ( u^{i,N} ), \qquad \mathcal{E} ( u^{i,N} ) := \int_0^T \frac{1}{2} \vert u^{i,N}_t \vert^2 \d t, \]
we introduce the stochastic control problem $\F^{h,N} := \inf_{\Sigma, \vec{X}^{h,N}_0,\vec{u}^N} J^h_\Sigma (\vec{X}^{h,N}_0, \vec{u}^N )$, where
\[ J^h_\Sigma (\vec{X}^{h,N}_0, \vec{u}^N ) := N^{-1} H( \L(\vec{X}^{h,N}_0) \vert \L(X^1_0)^{\otimes N} ) + \E \big[ \mathcal{E} ( \vec{u}^N ) + F( \pi ( \vec{X}^{h,N,\vec{u}^N} ) ) \big], \]
and $\mathcal{L}(X^1_0)$ is given by \ref{ass:IniExp}. 
Let us fix a reference system $\Sigma$. 
We first restrict the class of controls $(\vec{X}^{h,N}_0,\vec{u}^N)$.
We recall that $(X^{h,i,N}_0,u^{i,N})_{1 \leq i \leq N}$ is \emph{exchangeable} if the law of $(X^{h,\tau(i),N}_0,u^{\tau(i),N})_{1 \leq i \leq N}$ is the same for any permutation $\tau$ of $\{ 1 , \ldots, N \}$.

\begin{lemma}[Exchangeability] \label{lem:restrict}
We can restrict the minimisation defining $\F^{h,N}$ to controls $(\vec{X}^{h,N}_0,\vec{u}^N)$ such that for every $1 \leq i \leq N$,
\begin{enumerate}[label=(\roman*),ref=(\roman*)]
\item\label{item:lemresiid} $(X^{h,i,N}_0,u^{i,N})_{1 \leq i \leq N}$ is exchangeable and $\E [ \mathcal{E}(u^{1,N}) ] \leq 2 \lVert F \Vert_\infty$.
\item\label{item:lemresL1} $\E[ \vert X^{h,1,N}_0 \vert^p ] \leq C$, where $C$ only depends on $(p,\lVert F \rVert_\infty,\L(X^1_0))$.
\end{enumerate}
\end{lemma}

\begin{proof}
By choosing $\L(\vec{X}^{h,N}_0) = \L(X^1_0)^{\otimes N}$ and $\vec{u}^N \equiv 0$, we get that $\F^{h,N} \leq \Vert F \Vert_\infty.$
This allows us to restrict ourselves to admissible controls $\vec{u}^N$ with $\E[ \mathcal{E} ( \vec{u}^N ) ] \leq 2 \Vert F \Vert_\infty$. 
Let us fix such a $( \vec{X}^{h,N}_0, \vec{u}^N )$. 
Let $\tau$ be a uniformly distributed random permutation of $\{ 1 , \ldots, N \}$ that is independent of $\vec{B}^N$, $\vec{X}_0^{h,N}$, and $\vec{u}^N$. 
Defining $X^{h,\tau,i,N}_0 := X^{h,\tau(i),N}_0$ and $u^{\tau,i,N} := u^{\tau(i),N}$, we notice that $( X^{h,\tau,i,N}_0, u^{\tau,i,N} )_{1 \leq i \leq N}$ is exchangeable.
The symmetry of \eqref{eq:NControlled} then implies that $\pi(X^{h,\tau,N,\vec{u}^{\tau,N}})$ has the same law as $\pi(X^{h,N,\vec{u}^N})$.
For any permutation $\tau'$, we define 
\[ f^N_{\tau'} : (x_1,...,x_N) \mapsto (x_{\tau'(1)},\ldots,x_{\tau'(N)}). \] The contraction property of entropy \cite[Lemma A.1]{fischer2014form} gives that
\[ H ( \L ( \vec{X}^{h,\tau',N}_0 ) \vert \L(X^1_0)^{\otimes N} ) = H ( (f^N_{\tau'})_\# \L ( \vec{X}^{h,N}_0 ) \vert (f^N_{\tau'})_\# \L(X^1_0)^{\otimes N} ) \leq H ( \L ( \vec{X}^{h,N}_0 ) \vert \L(X^1_0)^{\otimes N} ). \]
Since $\L( \vec{X}^{h,\tau,N}_0 ) = \E [ (f^N_{\tau})_\# \L( \vec{X}^{h,N}_0 ) ]$ and $P \mapsto H ( P \vert Q)$ is convex, Jensen's inequality then implies that
\[ H ( \L ( \vec{X}^{h,\tau,N}_0 ) \vert \L(X^1_0)^{\otimes N} ) \leq \E [ H ( (f^N_{\tau})_\# \L ( \vec{X}^{h,N}_0 ) \vert \L(X^1_0)^{\otimes N} ) ] \leq H ( \L ( \vec{X}^{h,N}_0 ) \vert \L(X^1_0)^{\otimes N} ). \]
From all this, we can restrict oursleves to assuming \ref{item:lemresiid}.
Similarly, we can assume that
\[ N^{-1} H( \L( \vec{X}^{h,N}_0 ) \vert \L ( X^{1}_0)^{\otimes N} ) \leq 2 \lVert F \rVert_\infty. \]
Using the symmetry assumption \ref{item:lemresiid}, the tensorisation property \cite[Equation (2.10)]{csiszar1984sanov} yields 
\[ H( \L (X^{h,1,N}_0) \vert \L( X^1_0 )) \leq N^{-1} H( \L( \vec{X}^{h,N}_0 ) \vert \L ( X^{1}_0)^{\otimes N} ). \] 
Let us use the dual representation \cite[Proposition 3.1-(iii)]{leonard2012girsanov} for the relative entropy
\begin{equation} \label{eq:VarEntrop} 
H( \L (X^{h,1,N}_0) \vert \L( X^1_0 )) = \sup_{\substack{\phi \text{ measurable} \\ \E [ \phi ( X^{1}_0 ) ] < +\infty}} \E [ \phi ( X^{h,1,N}_0 ) ] - \log \E [ e^{\phi ( X^{1}_0 )} ],   
\end{equation}
with $\phi ( x ) := \vert x \vert^p$, $\E [ e^{\vert X^{1}_0 \vert^p} ]$ being finite from \ref{ass:IniExp}.
Gathering inequalities now allows us to assume the bound \ref{item:lemresL1} on $\E[ \vert X^{h,1,N}_0 \vert^p ]$.
\end{proof}

\begin{lemma}[Uniform approximation] \label{lem:Approxh}
There exists $C > 0$ such that for every $N \geq 1$, $h \in [0,1]$, every $(\vec{X}^{h,N}_0,\vec{u}^N)$ satisfying \ref{item:lemresiid}-\ref{item:lemresL1} in Lemma \ref{lem:restrict}, and $1 \leq i \leq N$,
\begin{enumerate}[label=(\roman*),ref=(\roman*)]
\item\label{item:lemAppbound} $\E [ \sup_{0 \leq t \leq T} \vert X^{h,i,N,\vec{u}^N}_t \vert^p ] \leq C$.
\item\label{item:lemAppShifthh} $\E [ \sup_{0 \leq t \leq T} \vert X^{h,i,N,\vec{u}^N}_t -  X^{h,i,N,\vec{u}^N}_{t_h} \vert^p ] \leq C \vert h \log h \vert^{p/2}$.
\item\label{item:lemAppShifth} $\E [ \sup_{0 \leq t \leq T} \vert X^{0,i,N,\vec{u}^N}_t -  X^{0,i,N,\vec{u}^N}_{t_h} \vert^p ] \leq C \vert h \log h \vert^{p/2}$.
\item\label{item:lemApphhinf} $\E [ \sup_{0 \leq t \leq T} \vert X^{h,i,N,0}_t - X^{0,i,N,0}_t \vert^p ] \leq C \vert h \log h \vert^{p/2}$.
\end{enumerate}
In the above estimates, we assume that $\vec{X}^{h,N,\vec{u}^N}_0 = \vec{X}^{0,N,\vec{u}^N}_0 = \vec{X}^{h,N}_0$.
\end{lemma}

\begin{proof} In the following proof, $C >0$ is a constant that may change from line to line, but always satisfying the requirements of Lemma \ref{lem:Approxh}.
With a slight abuse, we use the convention $t_0 = t$.
We write $X^{h,i} = X^{h,i,N,\vec{u}^N}$ for conciseness.

\ref{item:lemAppbound}
Using \ref{ass:coef1}, for $1 \leq i \leq N$, $0 \leq t \leq T$,
\begin{equation} \label{eq:JensenW}
\vert b_{t_h} ( X^{h,i}, \pi ( \vec{X}^{h} ) ) \vert^p \leq C [ 1 + \sup_{0 \leq s \leq t_h} \vert X^{h,i}_s \vert^p + N^{-1} \textstyle{ \sum_{j=1}^N \sup_{0 \leq s \leq t_h} \vert X^{h,j}_s \vert^p }].    
\end{equation}
From Lemma \ref{lem:restrict}-\ref{item:lemresiid}, all the $X^{h,i}$ have the same law. 
Moreover, $\vec{u}^N$ is square-integrable and $\sigma$ is bounded.
When $h >0$ a direct induction shows that $\E \big[ \sup_{0 \leq s \leq t} \vert X^{h,i}_s \vert^p \big]$ is finite under \ref{ass:coef1}.
When $h =0$, this is still true as a classical consequence of the Gronwall lemma.
Integrating in \eqref{eq:NControlled}, taking the $p$-power and using Jensen's inequality, before taking supremum in time and expectations, we get that
\begin{multline*}
\E \big[ \sup_{0 \leq s \leq t} \vert X^{h,i}_s \vert^p \big] \leq C + C T^{p-1} \int_0^t \E \big[ \sup_{0 \leq r \leq s} \vert X^{h,i}_r \vert^p \big] \d s
+ C M^p_\sigma \int_0^t \E [ \vert u^{i,N}_s \vert^p ] \d s, \\
+ C \E \sup_{0 \leq s \leq t} \bigg\vert  \int_0^s \sigma_{r_h}(X^{i,h},\pi(\vec{X}^{h})) \d B^i_r \bigg\vert^p.
\end{multline*} 
Since $p \in [1,2]$, the penultimate term is bounded by $C [ 1 + \E [ \mathcal{E}(u^{i,N}) ] ]$. 
From the BDG inequality and \ref{ass:coef1}, 
\[ \E \sup_{0 \leq s \leq t} \bigg\vert  \int_0^s \sigma_{r_h}(X^{i,h},\pi(\vec{X}^{h})) \d B^i_r \bigg\vert^p \leq C \E \bigg[ \bigg( \int_0^t \vert \sigma_{s_h}(X^{i,h},\pi(\vec{X}^{h})) \vert^2 \d s \bigg)^{p/2} \bigg] \leq C. \]
Gathering the terms, the Gronwall lemma now gives the desired uniform-in-$h$ bound.

\ref{item:lemAppShifthh}-\ref{item:lemAppShifth} Let $h'$ be in $\{0,h\}$. For $1 \leq i \leq N$ and $0 \leq t \leq T$,
\begin{multline*}
X^{h',i}_t - X^{h',i}_{t_h} = \int_{t_h}^t b_{s_{h'}} ( X^{h',i}, \pi ( \vec{X}^{h',N} ) ) \d s + \int_{t_h}^t \sigma_{s_{h'}} ( X^{h',i}, \pi ( \vec{X}^{h',N} ) ) u^{i,N}_s \d s \\
+ \int_{t_h}^t \sigma_{s_{h'}} ( X^{h',i}, \pi ( \vec{X}^{h',N} ) ) \d B^i_s.
\end{multline*} 
We then take the $p$-power and supremum in time.
Taking supremum in time and expectations in \eqref{eq:JensenW}, \ref{ass:coef1} and \ref{item:lemAppbound} yield
\[ \E \sup_{0 \leq t \leq T} \bigg\vert \int_{t_h}^t b_{s_{h'}} ( X^{h',i}, \pi ( \vec{X}^{h',N} ) ) \d s \bigg\vert^p \leq C h^{p}. \]
Similarly, using $p \in [1,2]$, Jensen's inequality and the bound on $\sigma$,
\[  \E \sup_{0 \leq t \leq T} \bigg\vert \int_{t_h}^t \sigma_{s_{h'}} ( X^{h',i}, \pi ( \vec{X}^{h',N} ) ) u^{i,N}_s \d s \bigg\vert^p \leq C h^{p/2} \E \bigg[ \bigg( \int_0^T \vert u^{i,N}_t \vert^2 \d t \bigg)^{p/2} \bigg] \leq C h^{p/2} \E^{p/2} [ 2 \mathcal{E}(u^{i,N})], \]
and $\E [ \mathcal{E}(u^{i,N})]$ is bounded using Lemma \ref{lem:restrict}-\ref{item:lemresiid}.
Since $\sigma$ is bounded, the estimate \cite[Theorem 1]{fischer2009moments} for the continuity modulus of Ito processes yields
\[ \E \sup_{0 \leq t \leq T} \bigg\vert \int_{t_h}^t \sigma_{s_{h'}} ( X^{h',i}, \pi ( \vec{X}^{h',N} ) ) \d B^i_s \bigg\vert^p \leq C \vert h \log h \vert^{p/2}. \]
Gathering all the terms gives the desired bound.

\ref{item:lemApphhinf} 
We now write $X^{h,i} = X^{h,i,N,0}$, for $h \in [0,1]$. 
For every $0 \leq t \leq t' \leq T$ and $1 \leq i \leq N$,
\[ X^{h,i}_{t'} - X^{0,i}_{t'} = X^{h,i}_{t} - X^{0,i}_{t} + \int_t^{t'} b^{h,i}_s \d s + \int_t^{t'} \sigma^{h,i}_s \d B^i_s,    \]
where
\[ b^{h,i}_{s} := b_{s_h} ( X^{h,i}, \pi ( \vec{X}^{h,N} ) ) - b_{s} ( X^{0,i}, \pi ( \vec{X}^{0,N} ) ), \]
and $\sigma^{h,i}_{s}$ is similarly defined. 
We then take the $p$-power of each side and supremum in time. We decompose as
\[ b^{h,i}_{s} = b_{s_h} ( X^{h,i}_{\wedge s_h}, \pi ( \vec{X}^{h,N}_{\wedge s_h} ) ) - b_{s_h} ( X^{0,i}_{\wedge s_h}, \pi ( \vec{X}^{0,N}_{\wedge s_h} ) )
+ b_{s_h} ( X^{0,i}_{\wedge s_h}, \pi ( \vec{X}^{0,N}_{\wedge s_h} ) ) - b_{s} ( X^{0,i}_{\wedge s}, \pi ( \vec{X}^{0,N}_{\wedge s} ) ). \]
Hence, using \ref{ass:coef1}, 
\begin{multline} \label{eq:driftDec}
\vert b^{h,i}_s \vert^p \leq C \sup_{0 \leq r \leq s_h} \vert X^{h,i}_r - X^{0,i}_r \vert^p + C N^{-1} \textstyle{\sum_{j=1}^N \sup_{0 \leq r \leq s_h} \vert X^{h,j}_r - X^{0,j}_r \vert^p} \\ + C h^p + C \sup_{0 \leq r \leq s} \vert X^{0,i}_r - X^{0,i}_{r_h} \vert^p + C N^{-1} \textstyle{\sum_{j=1}^N \sup_{0 \leq r \leq s_h} \vert X^{0,j}_r - X^{0,j}_{r_h} \vert^p}.  
\end{multline}
We recall that $((X^{h,j}, X^{0,j} ))_{1 \leq j \leq N}$ is exchangeable, and we now take expectations.
Using the BDG inequality, 
\[ \E \bigg\vert \sup_{t \leq s \leq t'} \int_0^s \sigma^{h,i}_r \d B^i_r  \bigg\vert^p \leq C \E \bigg[ \bigg( \int_t^{t'} \vert \sigma^{h,i}_s \vert^2 \d s \bigg)^{p/2} \bigg] \leq C (t'-t)^{p/2} \E \big[  \sup_{t \leq s \leq t'} \vert \sigma^{h,i}_s \vert^p \big], \]
and we handle $\vert \sigma^{h,i}_s \vert^p$ as we did for $\vert b^{h,i}_s \vert^p$. 
We further use \ref{item:lemAppShifth} to bound the expectation of the two last terms in \eqref{eq:driftDec}. 
Noticing that
\begin{equation} \label{eq:subAdd2}
\forall s \in [t,t'], \quad \sup_{0 \leq r \leq s} \vert X^{h,i}_r - X^{0,i}_r \vert \leq \sup_{0 \leq r \leq t} \vert X^{h,i}_r - X^{0,i}_r \vert + \sup_{t \leq r \leq s} \vert X^{h,i}_r - X^{0,i}_r \vert, 
\end{equation} 
we set $f(t,t') := \E \big[ \sup_{t \leq s \leq t'} \vert X^{h,i}_s - X^{0,i}_s \vert^p \big]$ and we gather all the terms to obtain
\[ \forall 0 \leq t \leq t' \leq T, \quad f(t,t') \leq C f (0,t) + C \vert h \log h \vert^{p/2} + C (t'-t)^{p/2} f(t,t') + C \int_t^{t'}  f(t,s) \d s, \]
for a constant $C > 0$ that does not depend on $(t,t')$.
The desired bound on $f(0,T)$ now follows from \eqref{eq:subAdd2} and Lemma \ref{lem:Gronwall}.
\end{proof}

\begin{proposition}[Quantitative convergence for mean-field control] \label{pro:valueF}
Under \ref{ass:coef1}-\ref{ass:IniExp}, if $\sigma_t (x,P) = \sigma_t (P)$ does not depend on $x$, then 
\[ \sup_{N \geq 1} \vert \F^{h,N} - \F^{0,N} \vert \xrightarrow[h \rightarrow 0]{} 0. \]
\end{proposition}

\begin{proof} 
In the following, $C >0$ is a constant that may change from line to line, but always satisfying the requirements of Proposition \ref{pro:valueF}.

\medskip

\emph{\textbf{Step 1.} Coupling of controls.}
For $\varepsilon >0$, let $( \Sigma, \vec{X}^{h,N}_0, \vec{u}^{h,N} )$, satisfying \ref{item:lemresiid}-\ref{item:lemresL1} in Lemma \ref{lem:restrict}, be $\varepsilon$-optimal for $\F^{h,N}$, meaning that $J^h_\Sigma ( \vec{X}^{h,N}_0, \vec{u}^{h,N} ) \leq \F^{h,N} + \varepsilon$. 
For any $( \vec{X}^{0,N}_0, \vec{u}^{0,N} )$,
\begin{equation} \label{eq:boundinf}
\F^{0,N} - \F^{h,N} \leq \varepsilon + J^0_\Sigma ( \vec{X}^{0,N}_0, \vec{u}^{0,N} ) - J^h_\Sigma ( \vec{X}^{h,N}_0, \vec{u}^{h,N} ).  
\end{equation} 
We choose $\vec{X}^{0,N}_0 = \vec{X}^{h,N}_0$ and $\vec{u}^{0,N} = \vec{u}^{h,N}$.
For conciseness, we write $X^{h,i}$ instead of $X^{h,i,N,\vec{u}^{h,N}}$, and $u^i$ for $u^{h,i,N}$. 
For $M > 0$, let us control $\E \big[ \1_{\mathcal{E}(\vec{u}^{h,N}) \leq M} \sup_{0 \leq t \leq T} \vert X^{h,i}_t - X^{0,i}_t \vert^p \big]$ as in Lemma \ref{lem:Approxh}-\ref{item:lemApphhinf}. 
Setting $\vec{v}^{h,N} := \1_{\mathcal{E}(\vec{u}^{h,N}) \leq M} \vec{u}^{h,N}$, we notice that
\begin{align*}
\E \big[ \1_{\mathcal{E}(\vec{u}^{h,N}) \leq M} \sup_{0 \leq t \leq T} \vert X^{h,i,\vec{u}^{h,N}}_t - X^{0,i,\vec{u}^{h,N}}_t \vert^p \big] &= \E \big[ \1_{\mathcal{E}(\vec{u}^{h,N}) \leq M} \sup_{0 \leq t \leq T} \vert X^{h,i,\vec{v}^{h,N}}_t - X^{0,i,\vec{v}^{h,N}}_t \vert^p \big] \\
&\leq \E \big[ \sup_{0 \leq t \leq T} \vert X^{h,i,\vec{v}^{h,N}}_t - X^{0,i,\vec{v}^{h,N}}_t \vert^p \big].
\end{align*} 
Thus, it is sufficient to perform the estimate while assuming that $\mathcal{E}(\vec{u}^{h,N}) \leq M$ a.s.; this corresponds to replacing $\vec{u}^{h,N}$ by $\vec{v}^{h,N}$, removing $\1_{\mathcal{E}(\vec{u}^{h,N}) \leq M}$ from the expectation.
As in the proof of Lemma \ref{lem:Approxh}-\ref{item:lemApphhinf}, for $0 \leq t \leq t'$, we decompose
\[ X^{h,i}_{t'} - X^{0,i}_{t'} = X^{h,i}_{t} - X^{0,i}_{t} + \int_t^{t'} b^{h,i}_s \d s + \int_t^{t'} [ \sigma_{s_h} ( \pi ( \vec{X}^{h} ) ) - \sigma_s ( \pi ( \vec{X}^{0} ) )  ] u^{i}_s \d s + \int_t^{t'} \sigma^{h,i}_s \d B^i_s. \]
We then take $p$-power, supremum in time and expectations, and we get as previously that
\[ \E \big[ \vert b^{h,i}_s \vert^p ] \leq C h^p + C \E \big[ \sup_{0 \leq r \leq s} \vert X^{h,i}_r - X^{0,i}_{r} \vert^p + \sup_{0 \leq r \leq s} \vert X^{0,i}_r - X^{0,i}_{r_h} \vert^p \big], \]
\[ \E \bigg\vert \sup_{t \leq s \leq t'} \int_t^s \sigma^{h,i}_r \d B^i_r \bigg\vert^p \leq C h^p +C (t'-t)^{p/2} \E \big[ \sup_{0 \leq r \leq s} \vert X^{h,i}_r - X^{0,i}_{r} \vert^p + \sup_{0 \leq r \leq s} \vert X^{0,i}_r - X^{0,i}_{r_h} \vert^p \big]. \]
Using $p \in [1,2)$ with Hölder's and Jensen's inequalities, we further get
\begin{multline*}
\bigg\vert \int_t^{t'} [ \sigma_{s_h} ( \pi ( \vec{X}^{h} ) ) - \sigma_s ( \pi ( \vec{X}^{0} ) )  ] u^{i}_s \d s \bigg\vert^p \leq \bigg[ \int_t^{t'} \vert \sigma_{s_h} ( \pi ( \vec{X}^{h} ) ) - \sigma_s ( \pi ( \vec{X}^{0} ) ) \vert^{p'} \d s\bigg]^{p/p'} \int_t^{t'} \vert u^{i}_s \vert^{p} \d s, \\
\leq (t'-t)^{p/2} \sup_{t \leq s \leq t'} \vert \sigma_{s_h} ( \pi ( \vec{X}^{h} ) ) - \sigma_s ( \pi ( \vec{X}^{0} ) ) \vert^p \bigg[ \int_0t^{t'} \vert u^i_s \vert^2 \d s \bigg]^{p/2},
\end{multline*} 
denoting by $p'$ the conjugate exponent of $p$.
From Jensen's inequality,
\[ \frac{1}{N} \sum_{i=1}^N \bigg[ \int_t^{t'} \vert u^i \vert^2 \d s \bigg]^{p/2} \leq [ 2 \mathcal{E} ( \vec{u}^{h,N} ) ]^{p/2} \leq ( 2 M)^{p/2}. \]
Using the same splitting as \eqref{eq:driftDec}, we finally get 
\[ \E \big[ \sup_{t \leq s \leq t'} \vert \sigma_{s_h} ( \pi ( \vec{X}^{h} ) ) - \sigma_s ( \pi ( \vec{X}^{0} ) ) \vert^p \big] \leq C h^p + C \E \big[ \sup_{0 \leq r \leq s} \vert X^{h,i}_r - X^{0,i}_{r} \vert^p + \sup_{0 \leq r \leq s} \vert X^{0,i}_r - X^{0,i}_{r_h} \vert^p \big]. \]
Using \eqref{eq:subAdd2} and setting $f(t,t') := \E \big[ \sup_{t \leq s \leq t'} \vert X^{h,i}_s - X^{0,i}_s \vert^p \big]$, as in the proof of Lemma \ref{lem:Approxh}-\ref{item:lemApphhinf}, we get that for every $0 \leq t \leq t' \leq T$, 
\[ f(t,t') \leq C f (0,t) + C [ 1 + M ] \vert h \log h \vert^{p/2} + C [ 1 + M ] (t'-t)^{p/2} f(t,t') + C \int_t^{t'}  f(t,s) \d s, \]
for $C> 0$ independent of $(t,t')$, and we conclude as previously using Lemma \ref{lem:Gronwall}.

At the end of the day, we deduce that for every $M >0$,
\[ \forall h \in [0,1], \quad \E \big[ \1_{\mathcal{E}(\vec{u}^{h,N}) \leq M} \sup_{0 \leq t \leq T} \vert X^{h,i}_t - X^{0,i}_t \vert^p \big] \leq C_M \vert h \log h \vert^{p/2}, \]
for a constant $C_M > 0$ independent of $(h,N)$.

\medskip

\emph{\textbf{Step 2.} Difference of the costs.}
For $M >0$, $F$ being $W_p$-Lipschitz,
\[ \1_{\mathcal{E} (\vec{u}^{h,N}) \leq M} \vert F( \pi ( \vec{X}^{0} ) ) - F( \pi ( \vec{X}^{h} ) ) \vert^p \leq \frac{\lVert F \rVert^p_{\mathrm{Lip}}}{N} \sum_{i=1}^N \1_{\mathcal{E} (\vec{u}^{h,N}) \leq M} \sup_{0 \leq t \leq T} \vert X^{0,i}_t - X^{h,i}_t \vert^p. \]
Using \emph{\textbf{Step 1.}}, this implies
\[ J^0_\Sigma ( \vec{X}^{h,N}_0, \vec{u}^{h,N} ) - J^h_\Sigma ( \vec{X}^{h,N}_0, \vec{u}^{h,N} ) \leq C_M \vert h \log h \vert^{1/2} + 2 \lVert F \rVert_\infty \P ( \mathcal{E} (\vec{u}^{h,N}) > M ). \]
We bound the last term by Markov's inequality and Lemma \ref{lem:restrict}-\ref{item:lemresiid}.
Plugging this into \eqref{eq:boundinf},
\[ \F^{0,N} - \F^{h,N} \leq \varepsilon + C_M \vert h \log h \vert^{1/2} + C M^{-1}, \]
and we send $\varepsilon$ to $0$.
Since $C$ does not depend on $(h,M,N)$ and $C_M$ does not depend on $(h,N)$, this implies that $\limsup_{h \rightarrow 0} \sup_{N \geq 1} \F^{0,N} - \F^{h,N} \leq 0$. 

\medskip

\emph{\textbf{Step 3.} Conclusion.} Inverting the roles of $\F^{0,N}$ and $\F^{h,N}$, we readily adapt  \emph{\textbf{Steps 1-2.}} to obtain that $\limsup_{h \rightarrow 0} \sup_{N \geq 1} \F^{h,N} - \F^{0,N} \leq 0$, completing the proof.
\end{proof}

\begin{proof}[End of the proof of Theorem \ref{thm:LDPGood}]
For $h \in [0,1]$, as a particular case of Corollary \ref{cor:Repmf}, 
\[ \F^{h,N} = \inf_{ R_N \in \ps_p ( (\R^d \times \C^{d'} )^N ) } N^{-1} H( R_N \vert \L(X^1_0,B^1)^{\otimes N} ) + N^{-1} \int_{\R^d \times \C^d} [ F \circ \Psi_h ] \d R_N. \]
From \cite[Proposition 1.4.2]{dupuis2011weak}, this implies that
\[ \F^{h,N} = -N^{-1} \log \E \big[ \exp \big[ -N F(\Psi_h(\pi(\vec{X}^N_0,\vec{B}^N))) \big]. \]
Proposition \ref{pro:valueF} then yields the uniform-in-$N$ convergence stated in Theorem \ref{thm:LDPGood}, as $h \rightarrow 0$.

For $h \in (0,1]$, Proposition \ref{pro:LDPhB} gives that the sequence of the $\L(\Psi_h(\pi(\vec{X}^{N}_0,\vec{B}^N)))$ satisfies the LDP with good rate function $P \in \ps_p (\C^{d}) \mapsto H( P \vert \Gamma_h ( P) )$. 
Varadhan's lemma \cite[Theorem 1.2.1]{dupuis2011weak} then gives that
\begin{equation*}
-N^{-1} \log \E \big[ \exp \big[ -N F(\Psi_h(\pi(\vec{X}^N_0,\vec{B}^N))) \big] \big] \xrightarrow[N \rightarrow +\infty]{} \inf_{P \in \ps_p(\C^{d})} H(P \vert \Gamma_h(P)) + F(P). 
\end{equation*} 
From Proposition \ref{pro:CVexpF}, the r.h.s converges to $\inf_{P \in \ps_p(\C^{d})} H(P \vert \Gamma_0(P)) + F(P)$ as $h \rightarrow 0$.
The uniform-in-$N$ convergence now allows us to invert the $N \rightarrow +\infty$ and $h \rightarrow 0$ limits, so that 
\begin{equation*} 
-N^{-1} \log \E \big[ \exp \big[ -N F(\pi(\vec{X}^{0,N})) \big] \big] \xrightarrow[N \rightarrow +\infty]{} \inf_{P \in \ps_p(\C^{d})} H(P \vert \Gamma_0(P)) + F(P), 
\end{equation*} 
for every bounded Lipschitz $F : \ps_p ( \C^d) \rightarrow \R$.
From \cite[Corollary 1.2.5]{dupuis2011weak}, this proves that the sequence of the $\L(\pi(\vec{X}^{0,N}))$
satisfies the LDP with good rate function $P \mapsto H( P \vert \Gamma_0 ( P) )$, this rate function being indeed good from Corollary \ref{cor:goodRate}.
The $\Gamma$-convergence of $I_h$ towards $I_0$ precisely corresponds to the proof of  Proposition \ref{pro:CVexpF} with $F \equiv 0$. 
\end{proof}

\subsection{Markov setting} \label{subsec:contractproof}

In the setting of Theorem \ref{thm:Contraction}, still for $p \in [1,2)$, let now $F : \C( [0,T], \ps_p ( \R^d ) ) \rightarrow \R$ denote a bounded Lipschitz-continuous function. 
We keep all the notations from the previous section, replacing $F( \pi ( \vec{X}^{h,N,\vec{u}^N} ) )$ by $F( \pi_\cdot ( \vec{X}^{h,N,\vec{u}^N} ) )$ in the definition of $J^h$.
As previously, we fix a reference system $\Sigma$, and we start all the processes from the same initial condition $\vec{X}^{h,N}_0$.

\begin{proposition}[Quantitative convergence] \label{pro:valueFcontra}
Under \ref{ass:IniExp}-\ref{ass:coef3}, for every $M > 0$,
there exists $C_M > 0$ such that for every $N \geq 1$,
\[ \forall h \in [0,1], \quad \E \bigg[ \1_{\mathcal{E}(\vec{u}^N) \leq M} \sup_{0 \leq t \leq T} \frac{1}{N} \sum_{i =1}^N \vert X^{h,i,N,\vec{u}^N}_t - X^{0,i,N,\vec{u}^N}_t \vert^p \bigg] \leq C_M \vert h \log h \vert^{\tfrac{p(1-p/2)}{2}}, \]
for every admissible $(\vec{X}^{h,N}_0,\vec{u}^N)$ satisfying \ref{item:lemresiid}-\ref{item:lemresL1} in Lemma \ref{lem:restrict}.
\end{proposition}

\begin{proof}
As in \emph{\textbf{Step 1.}} in the proof of Proposition \ref{pro:valueF}, we can assume, up to modifying $\vec{u}^N$, that $\mathcal{E}(\vec{u}^N) \leq M$ almost surely, thus removing $\1_{\mathcal{E}(\vec{u}^N) \leq M}$ from the expectation.
Let us fix $\varepsilon \in (0,1]$.
We are going to control errors in a suitably weighted norm that allows us to get rid of the controls.

\medskip

\emph{\textbf{Step 1.} Notations.}
In the following, we write $\vec{X}^{h,N} = \vec{X}^{h,N,\vec{u}^N}$ and we introduce
\[ \Lambda^{h,i}_t := X^{h,i}_t - X^{h,i}_{t_h}, \quad \Delta^{h,i}_t := X^{h,i}_t - X^{0,i}_t, 
\quad \eta^{N}_t := \exp \bigg[ - \frac{2^p L^p_\sigma}{N} \sum_{i = 1}^N \int_0^t [1 + \lvert u^{i,N}_s \rvert^2 ] \d s \bigg],  \]
together with
\[ \overline\Lambda^{h}_t := \frac{1}{N} \sum_{i=1}^N  \lvert \Lambda^{h,i}_t \rvert^p, \qquad \overline\Delta^{h}_t := \frac{1}{N} \sum_{i=1}^N \lvert \Delta^{h,i}_t \rvert^p. \]
For $1 \leq i \leq N$, we have the decomposition
\[ \d \Delta^{h,i}_{t} = b^{h,i}_{t} \d t + \sigma^{h,i}_{t} u^{i,N}_{t} \d t + \sigma^{h,i}_{t} \d B^i_{t}, \]
where
\[ b^{h,i}_{t} := b_{t_h} ( X^{h,i}_{t_h}, \pi \big( \vec{X}^{h,N}_{t_h} ) ) - b_{t} ( X^{0,i}_t, \pi ( \vec{X}^{0,N}_t ) ), \]
and $\sigma^{h,i}_{t}$ is similarly defined.
Let $\vert \Delta^{h,i}_{t} \vert_\varepsilon := (\vert \Delta^{h,i}_{t} \vert^2 + \varepsilon^2)^{1/2}$.

\medskip

\emph{\textbf{Step 2.} Differentiation.}
Using Ito's formula,
\begin{align*}
\d [ \eta^N_{t} \lvert \Delta^{h,i}_{t} \rvert^p_\varepsilon ] = &-\eta^N_{t} 2^p L^p_\sigma \lvert \Delta^{h,i}_{t} \rvert_\varepsilon^{p} N^{-1} \textstyle{\sum_{i=1}^N [ 1 + \lvert u^{i,N}_t \rvert^2 ]} \d t + p \, \eta^N_{t} \lvert \Delta^{h,i}_{t} \rvert_\varepsilon^{p-2} \Delta^{h,i}_{t} \cdot b^{h,i}_{t} \d t \\
&+ p \, \eta^N_{t} \lvert \Delta^{h,i}_{t} \rvert_\varepsilon^{p-2} \Delta^{h,i}_{t} \cdot \sigma^{h,i}_{t} u^{i,N}_{t} \d t + p \, \eta^N_{t} \lvert \Delta^{h,i}_{t} \rvert_\varepsilon^{p-2} \Delta^{h,i}_{t} \cdot \sigma^{h,i}_{t} \d B^i_{t} \\
&+ \tfrac{p}{2} \, \eta^N_{t} \lvert \Delta^{h,i}_{t} \rvert_\varepsilon^{p-2} \, \mathrm{Tr} \big[ \sigma^{h,i}_{t} ( \sigma^{h,i}_{t} )^\top [ \mathrm{Id} + (p-2) \lvert \Delta^{h,i}_{t} \rvert_\varepsilon^{-2}( \Delta^{h,i}_{t} \otimes \Delta^{h,i}_t ) ] \big] \d t.
\end{align*} 
We now integrate in time. 
First, $b$ being Lipschitz-continuous in $(t,x,P_t)$ from \ref{ass:coef3},
\[ \vert b^{h,i}_t \vert^p \leq C [ h^p + \lvert \Delta^{h,i}_{t} \rvert^p + \overline{\Delta}^{h}_{t} + \lvert \Lambda^{h,i}_t \rvert^p + \overline{\Lambda}^{h}_t ], \]
using a decomposition similar to \eqref{eq:driftDec}.
Since $p \in [1,2)$ and $\lvert \Delta^{h,i}_t \rvert \leq \lvert \Delta^{h,i}_t \rvert_\varepsilon$, this yields
\begin{multline} \label{eq:Youngb}
p \lvert \Delta^{h,i}_t \rvert_\varepsilon^{p-2} \Delta^{h,i}_t \cdot b^{h,i}_t \leq p \lvert \Delta^{h,i}_t \rvert^{p-1} 
\lvert b^{h,i}_t \rvert \leq (p-1)\lvert \Delta^{h,i}_t \rvert^{p} + \vert b^{h,i}_t \vert^p \\ 
\leq C [ h^p + \lvert \Delta^{h,i}_{t} \rvert^p + \overline{\Delta}^{h}_{t} + \lvert \Lambda^{h,i}_t \rvert^p + \overline{\Lambda}^{h}_t ], \phantom{ancde}
\end{multline}
where the second inequality uses Young's inequality.
Similarly, using that $\sigma_t(x,P) = \sigma_t (x)$ and $\vert \sigma^{h,i}_t \vert \leq 2 M_\sigma$,
\begin{equation*} 
\vert \sigma^{h,i}_t ( \sigma^{h,i}_t )^\top \vert \leq (2 M_\sigma)^2 \wedge [ C h^2 + C \lvert \Delta^{h,i}_{t} \rvert^2 + C \lvert \Lambda^{h,i}_{t} \rvert^2 ] \leq C [ h^2 + \lvert \Delta^{h,i}_{t} \rvert^2 ] + ( 4 M^2_\sigma ) \wedge ( C \lvert \Lambda^{h,i}_{t} \rvert^2 ).
\end{equation*}  
We then use $(4 M^2_\sigma) \wedge ( C \lvert \Lambda^{h,i}_{t} \rvert^2 ) \leq (4 M^2_\sigma)^{1-p/2} C^{p/2} \lvert \Lambda^{h,i}_{t} \rvert^p$, $\lvert \Delta^{h,i}_t \rvert \leq \lvert \Delta^{h,i}_t \rvert_\varepsilon$ and $\lvert \Delta^{h,i}_{t} \rvert_\varepsilon \geq \varepsilon$, 
\begin{equation} \label{eq:ContractSigma}
p \lvert \Delta^{h,i}_t \rvert_\varepsilon^{p-2} \vert \sigma^{h,i}_t ( \sigma^{h,i}_t )^\top \vert \leq C [ \varepsilon^{p-2} h^2 + \varepsilon^{p-2} \lvert \Lambda^{h,i}_t \rvert^{p} + \lvert \Delta^{h,i}_t \rvert^{p} ].
\end{equation}

\emph{\textbf{Step 3.} Gathering and cancellations.} At this stage, using $\lvert \Delta^{h,i}_t \rvert \leq \lvert \Delta^{h,i}_t \rvert_\varepsilon$ and  $h^2 \leq h^p$, 
we get that for every $0 \leq t \leq t'\leq T$,
\begin{align*}
\eta^N_{t'} \lvert &\Delta^{h,i}_{t'} \rvert^p \leq \eta^N_t \lvert \Delta^{h,i}_t \rvert^p_\varepsilon + \int_t^{t'} C \eta^N_s [ (1+\varepsilon^{p-2}) h^p + \lvert \Delta^{h,i}_{s} \rvert^p + \overline{\Delta}^{h}_{s} + (1 + \varepsilon^{p-2} ) \lvert \Lambda^{h,i}_s \rvert^p + \overline{\Lambda}^{h}_s ] \d s \\ 
&+\int_t^{t'} \big\{ - \eta^N_{s} 2^p L^p_\sigma \vert \Delta^{h,i}_s \rvert^p N^{-1} \textstyle{\sum_{i=1}^N [ 1 + \lvert u^{i,N}_s \rvert^2 ]} + p \, \eta^N_{s} \lvert \Delta^{h,i}_{s} \rvert_\varepsilon^{p-2} \Delta^{h,i}_{s} \cdot \sigma^{h,i}_{t} u^{i,N}_{s} \big\} \d s \\ 
&+ \int_t^{t'} p \, \eta^N_{s} \lvert \Delta^{h,i}_{s} \rvert_\varepsilon^{p-2} \Delta^{h,i}_{s} \cdot \sigma^{h,i}_{s} \d B^i_{s}, 
\end{align*}
integrating all the pieces from \emph{\textbf{Step 2.}}. Reasoning as in \eqref{eq:Youngb}-\eqref{eq:ContractSigma}, 
\[
p \lvert \Delta^{h,i}_{s} \rvert_\varepsilon^{p-2}  \Delta^{h,i}_{s} \cdot \sigma^{h,i}_{s} u^{i,N}_{s}
\leq (p - 1) \vert \Delta^{h,i}_s \vert^p + 2^p L_\sigma^p \vert \Delta^{h,i}_s \vert^p \vert u^{i,N}_s \vert^p + \vert u^{i,N}_s \vert^p [ C h^p + ( 2 M_\sigma )^p \wedge ( C \vert \Lambda^{h,i}_s \vert^p) ].
\]
We now sum over $i$ and we divide by $N$.
Since $\lvert u^{i,N}_s \rvert^p \leq 1 + \vert u^{i,N}_s \vert^2$, this cancels the $2^p L^p_\sigma \vert \Delta^{h,i}_s \vert^p \vert u^i_s \vert^p$ terms.
We then use that $\eta^N_t \leq 1$, $\lvert \Delta^{h,i}_t \rvert_\varepsilon \leq \varepsilon + \lvert \Delta^{h,i}_t \rvert$,  and we take supremum in time to obtain
\begin{multline*}
\sup_{t \leq s \leq t'} \eta^N_s \overline\Delta^{h}_s \leq C [\varepsilon^p + \overline\Delta^{h}_t] + \int_t^{t'} C \big\{ (1+\varepsilon^{p-2})[ h^p + \sup_{0\leq r \leq s} \overline{\Lambda}^{h}_r ] + \sup_{0\leq r \leq s} \overline{\Delta}^{h}_r + h^p {\textstyle N^{-1} \sum_{i=1}^N \vert u^{i,N}_s \vert^p} \big\} \d s \\
+ {\int_t^{t'} \textstyle{ N^{-1} \sum_{i=1}^N \vert u^{i,N}_s \vert^p [ (2 M_\sigma )^p \wedge ( C \lvert \Lambda^{h,i}_{s} \rvert^p ) ] } \d s} \,
+ \, \sup_{t \leq s \leq t'} \int_t^s p \eta^N_r N^{-1} \textstyle{ \sum_{i=1}^N \lvert \Delta^{h,i}_{r} \rvert_\varepsilon^{p-2} \Delta^{h,i}_{r} \cdot \sigma^{h,i}_{r} \d B^i_{r}}.
\end{multline*}
Let $\mathbf{B}_t$ denote the last term on the r.h.s. Since $\mathcal{E}(\vec{u}^{N}) \leq M$, we have
\[ \int_t^{t'} {h^p \textstyle N^{-1} \sum_{i=1}^N \vert u^{i,N}_s \vert^p} \d s \leq 2 T [1 + M] h^p. \]
Using $p < 2$, Hölder's inequality further yields
\begin{multline*}
\int_t^{t'} {\textstyle N^{-1} \sum_{i=1}^N \vert u^{i,N}_s \vert^p [ (2 M_\sigma )^p \wedge ( C \lvert \Lambda^{h,i}_{s} \rvert^p ) ]} \d s \\ 
\leq [ 2 \mathcal{E}(\vec{u}^{N}) ]^{p/2} \bigg[ \int_t^{t'} {\textstyle N^{-1} \sum_{i=1}^N [ (2 M_\sigma )^p \wedge ( C \lvert \Lambda^{h,i}_{s} \rvert^p ) ]^{2/(2-p)} } \bigg]^{1-p/2} \leq C M^{p/2} \bigg[ \int_t^{t'} \overline{\Lambda}^h_s \d s \bigg]^{1-p/2}.
\end{multline*}
Jensen's inequality then bounds the expectation of the r.h.s. by $C M^{p/2} \E^{1-p/2} [  T \sup_{0 \leq s \leq t} \overline\Lambda^h_s ]$.

\medskip

\emph{\textbf{Step 4.} Gronwall argument.}
Using the BDG inequality, and $\eta^N_r \leq 1$,
\begin{align*}
\E [ \textbf{B}_t ] &\leq
C \E \bigg[ \bigg( \int_t^{t'} \big\vert \textstyle{N^{-1} \sum_{i=1}^N p \, \lvert \Delta^{h,i}_{r} \rvert_\varepsilon^{p-2} \Delta^{h,i}_{r} \cdot \sigma^{h,i}_{r} } \big\vert^2 \d s \bigg)^{1/2} \bigg] \\
&\leq C (t'-t)^{1/2} \E \big[ h^p + \sup_{0 \leq s \leq t} \overline{\Delta}^h_s + \sup_{0 \leq s \leq t} \overline\Lambda^h_s \big],
\end{align*} 
where we reasoned as in \eqref{eq:Youngb} for the second inequality. 
Noticing that
\begin{equation} \label{eq:subbad3}
  \forall 0 \leq t \leq s \leq t' \leq T, \quad \sup_{0 \leq r \leq s} \overline\Delta^h_r \leq \sup_{0 \leq r \leq t} \overline\Delta^h_r + \sup_{t \leq r \leq s} \overline\Delta^h_r,  
\end{equation}
we set $f^{h}(t,t') := \E [ \sup_{t \leq s \leq t'} \overline{\Delta}^h_s ]$, and we gather all the pieces from \emph{\textbf{Step 3.}} to get
\begin{multline*}
e^{- 2^p L_\sigma^p (2M + T)} f^h (t,t') \leq C \varepsilon^p + C f^h (0,t) + C [ 1 + \varepsilon^{p-2} ] \big[ (1+M) h^p + \E \big[ \sup_{0 \leq s \leq T} \overline\Lambda^{h}_s \big] \big] \\
+ C M^{p/2} \E^{1-p/2} \big[ \sup_{0 \leq s \leq T} \overline\Lambda^h_s \big]
+ (t'-t)^{1/2} f^h(t,t') + \int_t^{t'} f^{h} (t,s) \d s, 
\end{multline*}
where we used that $\eta^N_t \geq e^{- 2^p L_\sigma^p (2M + T)}$ a.s., and $C$ is independent of $(t,t')$. 
Using \eqref{eq:subbad3} and Lemma \ref{lem:Gronwall}, we get
\[ f^h (0,T) \leq C_M \varepsilon^p + C_M [ 1 + \varepsilon^{p-2} ] \big[ h^p + \E \big[ \sup_{0 \leq s \leq T} \overline\Lambda^{h}_s \big] \big] + C_M \E^{1-p/2} \big[ \sup_{0 \leq s \leq T} \overline\Lambda^h_s \big], \]
for $C_M > 0$ independent of $(h,N)$.
Since $ \E [ \sup_{0 \leq s \leq T} \overline\Lambda^{h}_s ]$ is bounded by $C \vert h \log h \vert^{p/2}$ from Lemma \ref{lem:Approxh}-\ref{item:lemAppShifthh}, taking $\varepsilon = \vert h \log h \vert^{1/2}$ concludes.  
\end{proof}

\begin{proof}[End of the proof of Theorem \ref{thm:Contraction}]
Reasoning as in \emph{\textbf{Step 2.}} in the proof of Proposition \ref{pro:valueF}, Proposition \ref{pro:valueFcontra} gives that for every $M > 0$, $C_M > 0$ exists such that
\[ \forall h \in(0,1], \; \forall N \geq 1, \qquad \vert \F^{h,N} - \F^{0,N} \vert \leq 2 \lVert F \Vert_\infty M^{-1} + C_M \vert h \log h \vert^{\tfrac{p(1-p/2)}{2}}. \]
As a consequence,
\[ \sup_{N \geq 1} \, \vert \F^{h,N} - \F^{0,N} \vert \xrightarrow{h \rightarrow 0} 0. \]
Inverting the limits, the result is then a direct adaptation of the proof of Theorem \ref{thm:LDPGood}.
\end{proof}

\subsection{Uniform fluctuation estimates} \label{subsec:CLT}

We here extend Proposition \ref{pro:ClThB} to the continuous setting $h =0$.
Throughout this section, we assume that \ref{ass:pair} holds.
The structure of the following proofs is very close to Section \ref{subsec:disCLTB}, studying the $h \rightarrow 0$ limit instead of $M \rightarrow +\infty$. 
We essentially emphasise the main differences and the needed adaptations.

Let $(B^{i})_{i \geq 1}$ be a countable sequence in $\C^{d'}$ of i.i.d. Brownian motions.
Let $\vec{X}^{h,N}$ denote the related particle system \eqref{eq:hNpart}, starting at $0$.
We will write $\vec{X}^{0,N}$ for the continuous particle system, where \eqref{eq:hNpart} now involves stochastic integrals.
For each $1 \leq i \leq N$,
let us introduce the McKean-Vlasov process $\overline{X}^{h,i}$ solution of \eqref{eq:GenMcK} with the driving noise $B^i$, starting at $0$. 
We write $\overline{X}^{0,i}$ for the related McKean-Vlasov process.
Let us introduce the difference process $\delta^{h,i} := X^{h,i,N} - \overline{X}^{h,i}$.
We first need a uniform in $h$ mean-field limit.

\begin{lemma}[Uniform limit] \label{lem:PoCh}
There exists $C > 0$ such that for every $N \geq 1$, $h \in (0,1]$, and $1 \leq i \leq N$,
\begin{enumerate}[label=(\roman*),ref=(\roman*)]
\item\label{item:lemPoChbound} $\E [ \sup_{0 \leq t \leq T} \vert X^{h,i,N}_t \vert^2 ] \leq C$.
\item\label{item:lemPoChPoCh} $\E [ \sup_{0 \leq t \leq T} \vert \delta^{h,i,N}_t \vert^2 ] \leq C N^{-1}$.
\item\label{item:lemPoChPoCinf} $\E [ \sup_{0 \leq t \leq T} \vert \delta^{0,i}_t \vert^2 ] \leq C N^{-1}$.
\item\label{item:lemPoChhinf} $\E [ \sup_{0 \leq t \leq T} \vert X^{h,i,N}_t - X^{0,i,N}_t \vert^2 ] \leq C \vert h \log h \vert$.
\item\label{item:lemPoChinfinf} $\E [ \sup_{0 \leq t \leq T} \vert \overline{X}^{h,i}_t - \overline{X}^{0,i}_t \vert^2 ] \leq C \vert h \log h \vert$.
\end{enumerate}
\end{lemma}

We then strengthen these results for controlling the fluctuations.

\begin{proposition} \label{pro:approxhh} 
There exists $C >0$ such that for every $M >0$,
\[ \sup_{N \geq 1} \frac{1}{\sqrt{N}} \sum_{i=1}^N \E \bigg[ \sup_{0 \leq t \leq T} \big\lvert \delta^{h,M,i,N}_t - \delta^{h,\infty,i}_t \big\rvert \bigg] \leq C \vert h \log h \vert^{1/2}. \]
\end{proposition}

With these results at hand, Theorem \ref{thm:ClT} will easily follow.
As previously, we rely on the classical coupling method from \cite[Theorem 1.4]{sznitman1991topics}.  
To alleviate notations, we drop the exponents $N$ in the remainder of this section (although the dependence on $N$ is crucial).
We will repeatedly use the technical estimates proved in Appendix \ref{app:tech}.
In the following, $C$ is a generic constant that may change from line to line, but staying independent of $(h,N)$.

\begin{proof}[Proof of Lemma \ref{lem:PoCh}]
\ref{item:lemPoChbound} is obtained by noticing that the proof of Lemma \ref{lem:Approxh}-\ref{item:lemAppbound} still works with $p =2$.

\ref{item:lemPoChPoCh} We adapt the classical coupling argument, as in e.g. \cite{bernou2022path}. By definition of $X^{h,i}$ and $\overline{X}^{h,i}$, for $0 \leq t \leq T$,
\begin{multline*}
\delta^{h,i}_t = \int_0^t [ b_{s_h} ( X^{h,i} , \pi(\vec{X}^{h}) ) - b_{s_h} ( \overline{X}^{h,i} , \L(\overline{X}^{h,i}) ) ] \d s \\
+ \int_0^t [ \sigma_{s_h} ( X^{h,i} , \pi(\vec{X}^{h}) ) - \sigma_{s_h} ( \overline{X}^{h,i} , \L(\overline{X}^{h,i}) ) ] \d B^i_s.
\end{multline*} 
We then take the square, supremum in time, and expectations.
Using \ref{ass:coef2}, reasoning as for proving Lemma \ref{lem:PoCM}-\ref{item:lemPoCMPoCM} yields
\begin{equation*}
\E \big\vert b_{s_h} ( X^{h,i} , \pi(\vec{X}^{h}) ) - b_{s_h} ( \overline{X}^{h,i} , \L (\overline{X}^{h,i} ) ) \big\vert^2 \leq C \E \big[ \sup_{0 \leq r \leq s_h} \vert X^{h,i}_r - \overline{X}^{h,i}_r \vert^2 \big] + C N^{-1}.
\end{equation*} 
Using the BDG inequality,
\begin{multline*}
\E \sup_{0 \leq s \leq t} \bigg\vert \int_0^s [ \sigma_{r_h} ( X^{h,i} , \pi(\vec{X}^{h}) ) - \sigma_{r_h} ( \overline{X}^{h,i} , \L(\overline{X}^{h,i}) ) ] \d B^i_r \bigg\vert^2 \\
\leq C \int_0^t \vert \sigma_{s_h} ( X^{h,i} , \pi(\vec{X}^{h}) ) - \sigma_{s_h} ( \overline{X}^{h,i} , \L(\overline{X}^{h,i}) ) \vert^2 \d s, 
\end{multline*}
and we similarly handle the r.h.s.
We then gather everything to get that
\[ \E \big[ \sup_{0 \leq s \leq t} \vert \delta^{h,i}_s \vert^2 \big] \leq C N^{-1} + C \int_0^t \E \big[ \sup_{0 \leq r \leq s} \vert \delta^{h,i}_r \vert^2 \big] \d s. \]
The conclusion follows from the Gronwall Lemma.

\ref{item:lemPoChPoCinf} If $h = 0$ in the above proof of \ref{item:lemPoChPoCh}, we notice that all the computations remain valid when replacing $s_h$ by $s$.
This yields the result.

\ref{item:lemPoChhinf} is obtained by noticing that the proof of Lemma \ref{lem:Approxh}-\ref{item:lemApphhinf} still works with $p=2$.

\ref{item:lemPoChinfinf} By symmetry, we recall that $\E [ \sup_{0 \leq t \leq T} \vert X^{h,i,N}_t - X^{0,i,N}_t \vert^2 ]$ does not depend on $i$. 
The result follows by taking the $N \rightarrow +\infty$ limit in the uniform estimate \ref{item:lemPoChhinf}, using \ref{item:lemPoChPoCh}-\ref{item:lemPoChPoCinf}.
\end{proof}

\begin{proof}[Proof of Proposition \ref{pro:approxhh}]
For every $0 \leq t \leq t' \leq T$, we have the decomposition
\begin{equation*} 
\delta^{h,i}_{t'} - \delta^{0,i}_{t'} = \delta^{h,i}_{t} - \delta^{0,i}_{t} + \int_t^{t'} b_s^{h,i} \d s + \int_t^{t'} \sigma_s^{h,i} \d B^i_s, 
\end{equation*} 
where
\[ b_s^{h,i} := b_{s_h} ( X^{h,i}, \pi( \vec{X}^{h} ) ) - b_{s_h} ( \overline{X}^{h,i}, \L(\overline{X}^{h,i}) ) - [ b_{s} ( X^{0,i}, \pi( \vec{X}^{\infty} ) ) - b_{s} ( \overline{X}^{\infty,i}, \L(\overline{X}^{0,i}) ) ],  \]
\[ \sigma_s^{h,i} := \sigma_{s_h} ( X^{h,i}, \pi( \vec{X}^{h} ) ) - \sigma_{s_h} ( \overline{X}^{h,i}, \L(\overline{X}^{h,i}) ) - [ \sigma_{s} ( X^{0,i}, \pi( \vec{X}^{\infty} ) ) - \sigma_{s} ( \overline{X}^{\infty,i}, \L(\overline{X}^{0,i}) ) ]. \]
We then take absolute values, supremum in times and expectations. 
For the drift, we split
\begin{align*}
b_s^{h,i} = \phantom{b}&b_{s_h} ( X^{h,i}, \pi(\vec{X}^h) ) - b_{s_h} ( \overline{X}^{h,i}, \pi(\vec{\overline{X}}^h) ) - \big[ b_s ( X^{h,i}, \pi(\vec{X}^h) ) - b_s ( \overline{X}^{h,i}, \pi(\vec{\overline{X}}^h) ) \big] \\
&+ b_{s_h} ( \overline{X}^{h,i}, \pi(\vec{\overline{X}}^h) ) - b_{s_h} ( \overline{X}^{h,i}, \L ( \overline{X}^{h,i} ) ) - \big[  b_s ( \overline{X}^{h,i}, \pi(\vec{\overline{X}}^h) ) - b_s ( \overline{X}^{h,i}, \L ( \overline{X}^{h,i} ) ) \big] \\
&+ b_s ( {X}^{h,i}, \pi(\vec{X}^h) ) - b_s ( \overline{X}^{h,i}, \pi(\vec{\overline{X}}^h) )
- \big[ b_s ( {X}^{i,0}, \pi(\vec{X}^0_{s}) ) - b_s ( \overline{X}^{i,0}, \pi(\vec{\overline{X}}^0) ) \big] \\
&+ b_s ( \overline{X}^{h,i}, \pi(\vec{\overline{X}}^h) ) - b_s ( \overline{X}^{h,i}, \L ( \overline{X}^{h,i} ) ) 
- \big[ b_s ( \overline{X}^{i,0}_{s}, \pi(\vec{\overline{X}}^0_{s}) ) - b_s ( \overline{X}^{i,0}, \L ( \overline{X}^{i,0} ) ) \big].
\end{align*} 
For each line of $b^{h,i}_s$, we now reason as we did for proving Proposition \ref{pro:approxhM}.
For the first line, we use the differentiability assumption \ref{ass:pair} to write
\[ b_{s_h} ( X^{h,i}, \pi(\vec{X}^h) ) - b_{s} ( X^{h,i}, \pi(\vec{X}^h) ) = \int_0^1 (s_h - s) \partial_s b_{(1-r)s_h + r s} ( X^{h,i}, \pi(\vec{X}^h) ) \d r, \]
and we similarly rewrite $b_{s_h} ( \overline{X}^{h,i}, \pi(\vec{\overline{X}}^h) ) - b_{s} ( \overline{X}^{h,i}, \pi(\vec{\overline{X}}^h) )$.
Using the $\C^{1,1}$ assumption \ref{ass:pair} on $b$ and Lemma \ref{lem:PoCh}, this yields
\[ \sup_{0 \leq s \leq T} \E \big[ \vert b_{s_h} ( X^{h,i}, \pi(\vec{X}^h) ) - b_{s_h} ( \overline{X}^{h,i}, \pi(\vec{\overline{X}}^h) ) - [ b_s ( X^{h,i}, \pi(\vec{X}^h) ) - b_s ( \overline{X}^{h,i}, \pi(\vec{\overline{X}}^h) ) ] \vert \big] \leq C h N^{-1/2}. \]
We similarly handle the second line of $b^{h,i}_s$, using Lemma \ref{lem:techiid} to get that
\[ \max_{0 \leq r \leq 1} \E \big[ \big\vert \partial_s b_{(1-r)s_h + r s} ( \overline{X}^{h,i}, \pi(\vec{\overline{X}}^h) ) - \partial_s b_{(1-r)s_h + r s} ( \overline{X}^{h,i}, \L ( \overline{X}^{h,i} ) ) \big\vert \big] \leq C N^{-1/2}. \]
To handle the third line in $b^{h,i}_s$, we use Lemma \ref{lem:techC11} with $(X^i,Y^i) = (X^{h,i},X^{0,i})$ and $(\overline{X}^i,\overline{Y}^i) = (\overline{X}^{h,i}_s,\overline{X}^{0,i}_s)$, together with Lemma \ref{lem:PoCh}.
To handle the last line, we use Lemma \ref{lem:techiid} with the i.i.d. processes $(\overline{X}^i,\overline{Y}^i) = (\overline{X}^{h,i},\overline{X}^{0,i})$, $1 \leq i \leq N$. Gathering everything yields
\[ \sup_{t \leq s \leq t'} \E [ \vert b^{h,i}_s \vert ] \leq C N^{-1/2} \vert h \log h \vert^{1/2} + C \E [ \sup_{0 \leq s \leq t'} \vert \delta^{h,i}_s \vert ]. \]
We use a similar splitting for $\sigma^{h,i}_{s}$ and we reason as in the proof of Proposition \ref{pro:approxhM} to obtain
\[ \E \sup_{t \leq s \leq t'} \bigg\vert \int_t^s \sigma^{h,i}_r \d B^i_r \bigg\vert \leq C N^{-1/2} \vert h \log h \vert^{1/2} + C (t'-t)^{1/2} \E [ \sup_{0 \leq s \leq t'} \vert \delta^{h,i}_s \vert ]. \]
Setting $f^{h}(t,t') := \sqrt{N} \, \E [ \sup_{t \leq s \leq t'} \vert \delta^{h,i}_s \vert ]$, we gather terms as previously to get that
\[ \forall 0 \leq t \leq t' \leq T, \quad f^{h} (t,t') \leq C f(0,t) + C \vert h \log h \vert^{1/2} + C (t'-t)^{1/2} f^h (t,t') + C \int_t^{t'} f^{h} (t,s) \d s, \]
and we conclude using Lemma \ref{lem:Gronwall}, as in the proof of Proposition \ref{pro:approxhM}.
\end{proof}

\begin{proof}[End of the proof of Theorem \ref{thm:ClT}]
Let us fix $\varphi$ in $\C^{1,1}(\C^d,\R)$, before defining
\[ \delta^{h}_\varphi := N^{-1} {\textstyle\sum}_{i=1}^N \big[ \varphi ( X^{h,i} ) - \E [ \varphi ( \overline{X}^{h,i} ) ] \big]. \]
We then decompose
\begin{align*}
\delta^{h}_\varphi - \delta^{0}_\varphi =& \, N^{-1} {\textstyle\sum}_{i=1}^N \big[ \varphi ( X^{h,i} ) - \varphi ( \overline{X}^{h,i} ) \big] - N^{-1} {\textstyle\sum}_{i=1}^N \big[ \varphi ( X^{0,i} ) - \varphi ( \overline{X}^{0,i} ) \big]  \\
&\,+ N^{-1} {\textstyle\sum}_{i=1}^N \big[ \varphi ( \overline{X}^{h,i} ) - \E [ \varphi ( \overline{X}^{h,i} ) ] \big] - N^{-1} {\textstyle\sum}_{i=1}^N \big[ \varphi ( \overline{X}^{0,i} ) - \E [ \varphi ( \overline{X}^{0,i} ) ] \big].
\end{align*} 
As previously, we control the first line using Lemma \ref{lem:techC11}, Lemma \ref{pro:approxhh} and Proposition \ref{pro:approxhh}.
Similarly, we handle the second line using Lemma \ref{lem:techiid} and Lemma \ref{pro:approxhh}.
At the end of the day, we obtain that
\[ \sup_{N \geq 1} {\sqrt{N}} \, \E \big[ \lvert \delta^{h}_\varphi - \delta^{0}_\varphi \rvert \big] \xrightarrow[h \rightarrow 0]{} 0. \]
As a consequence, the convergence given by Proposition \ref{pro:ClThB} holds uniformly in $h$. 
Let $(B_t)_{0 \leq t \leq T}$, $(\tilde{B}_t)_{0 \leq t \leq T}$ be independent Brownian motions in $\R^d$.
Let $X^h$ denote the strong solution of the discretised McKean-Vlasov \eqref{eq:GenMcK} with driving noise $B$. 
Similarly, let $\tilde{X}^h$ denote the strong solution of \eqref{eq:GenMcK} with driving noise $\tilde{B}$. 
Let us then introduce
the solution $\delta X^h$ of the McKean-Vlasov SDE with common noise given by  \eqref{eq:McKNoiseh},
\begin{align*} 
\d \delta X^h_{t_h} &= \big[ D_x \sigma_{t_h} ( X^h, \L ( X^h ) ) \cdot \delta X^h + \delta_P \sigma_{t_h} ( X^h, \L ( X^h ), \tilde{X}^h ) \big] \d B_t \\
&+ \big[ D_x b_{t_h} ( X^h, \L ( X^h ) ) \cdot \delta X^h + \delta_P b_{t_h} ( X^h, \L ( X^h ), \tilde{X}^h ) \big] \d t \\
&+ \E \big[ D_y \delta_P b_{t_h} ( X^h , \L ( X^h ), X^h ) \cdot  \delta X^h \big\vert ( \tilde{B}_s )_{0 \leq s \leq t_h} \big] \d t \\
&+ \E \big[ D_y \delta_P \sigma_{t_h} ( X^h, \L ( X^h ), X^h ) \cdot  \delta X^h \big\vert ( \tilde{B}_s )_{0 \leq s \leq t_h} \big] \d B_t, \qquad \delta X^h_0 = 0.\phantom{abcd}
\end{align*}
As previously, synchronous coupling and a Gronwall argument give that
\[ \E[ \sup_{0 \leq t \leq T} \vert \delta X^{h}_t - \delta X_t \vert ] \xrightarrow[h \rightarrow 0]{} 0, \]
where $\delta X$ is the solution of \eqref{eq:McKNoise} with the same $(B,\tilde{B})$. 
From this, we conclude that $\sigma^2_{h,\varphi}$ given by Proposition \ref{pro:ClThB} converges to $\sigma^2_{\varphi}$ as required by Theorem \ref{thm:ClT}.
\end{proof}

\appendix

\section{Some useful estimates} \label{app:tech}

First, let us state a variation on the Gronwall lemma, which is used throughout the article.  

\begin{lemma} \label{lem:Gronwall}
Let $f : [0,T] \times [0,T] \rightarrow \R$, $g : [0,T] \rightarrow \R$ be continuous functions, and $ \alpha, C >0$ be such that 
\[ \forall 0 \leq t \leq t' \leq T, \quad f(t,t') \leq g(t) + C f (0,t) + C (t'-t)^{\alpha} f (t,t') + C \int_t^{t'} f (t,s) \d s, \]
together with
\[ \forall 0 \leq t \leq s \leq t', \quad f(0,s) \leq f ( 0, t ) + f (t,s). \]
Then, there exists $C' >0$ that only depends on $(\alpha,
C,T)$, such that $\sup_{0 \leq t \leq T} f(0,t) \leq C' [ f ( 0, 0 ) + \sup_{0 \leq t \leq T} g (t) ]$.
\end{lemma}

\begin{proof}
Let us fix $t_0 \in (0,T]$ such that $1 - C t_0^{\alpha} \geq 1/2$. 
For any $k \geq 0$ with $k t_0 \leq T$, we get that for every $t \in [kt_0,(k+1)t_0 \wedge T]$,
\[ f ( k t_0, t) \leq 2 g(t) + 2 C f ( 0, k t_0 ) + 2 C \int_{kt_0}^t f (k t_0,s) \d s.  \]
The usual Gronwall lemma then provides the bound
\[ \sup_{kt_0 \leq t \leq (k+1)t_0 \wedge T} f ( k t_0, t ) \leq 2 e^{2 C t_0} \big[ C f ( 0, k t_0 ) + \sup_{0 \leq t \leq T} g (t) \big].  \]
Since 
\[ f (0,(k+1)t_0) \leq f (0,k t_0) + f (k t_0,(k+1)t_0), \]
the desired bound on $f (0,T)$ follows by (finite) induction.
\end{proof}

Let now $(E, \vert . \vert)$ be a Banach space.
For $k \geq 1$, let $F : E \times E \rightarrow \R^k$ be a $C^{1,1}$ function. 
For $P \in \ps_1 (E)$, we define
\[ \forall x \in E, \quad F(x,P) := \int_E F(x,y) \d P ( y), \]
with a slight abuse of notation.

\begin{lemma} \label{lem:techC11}
Let $(X^{i})_{1 \leq i \leq N}$, $(\overline{X}^{i})_{1 \leq i \leq N}$, $(Y^{i})_{1 \leq i \leq N}$, $(\overline{Y}^{i})_{1 \leq i \leq N}$ be triangular arrays of $E$-valued random variables.
Then, for every $1 \leq i \leq N$,
\begin{multline*}
\vert F(X^{i},\pi(\vec{X}^{N})) - F(\overline{X}^{i},\pi(\vec{\overline{X}}^{N})) - [ F(Y^{i},\pi(\vec{Y}^{N})) - F(\overline{Y}^{i},\pi(\vec{\overline{Y}}^{N})) ] \vert  \\
\leq \frac{1}{N} \sum_{j =1}^N \lVert F \rVert_{\mathrm{Lip}} [ \vert X^{i} - \overline{X}^{i} - [ Y^{i} - \overline{Y}^{i} ] \vert + \vert X^{j} - \overline{X}^{j} - [ Y^{j} - \overline{Y}^{j} ] \vert ] \\
+ \lVert D F \rVert_{\mathrm{Lip}} [ \vert Y^{i} - \overline{Y}^{i} \vert + \vert Y^{j} - \overline{Y}^{j} \vert ] [ \vert {X}^{i} - {Y}^{i} \vert + \vert {X}^{j} - {Y}^{j} \vert + \vert \overline{X}^{i} - \overline{Y}^{i} \vert + \vert \overline{X}^{j} - \overline{Y}^{j} \vert ] 
\end{multline*} 
\end{lemma}

\begin{proof}
Since $F$ is differentiable, 
\begin{multline*}
F(X^{i},\pi(\vec{X}^{N})) - F(\overline{X}^{i},\pi(\vec{\overline{X}}^{N}))
=\, \frac{1}{N} \sum_{j=1}^N \int_0^1 (X^{i} - \overline{X}^{i}) \cdot \partial_1 F ((1-r) (\overline{X}^{i},\overline{X}^{j}) + r ({X}^{i},{X}^{j}) ) \d r  \\
+\frac{1}{N} \sum_{j=1}^N \int_0^1 (X^{j} - \overline{X}^{j}) \cdot \partial_2 F ((1-r) (\overline{X}^{i},\overline{X}^{j}) + r ({X}^{i},{X}^{j}) ) \d r, 
\end{multline*} 
and the same decomposition holds writing $Y$ instead of $X$.
We then subtract the decomposition for $Y$ to the above one.
The difference of the first terms reads
\begin{multline*}
\frac{1}{N} \sum_{j=1}^N \int_0^1 [ X^{i} - \overline{X}^{i} - ( Y^{i} - \overline{Y}^{i}) ] \cdot \partial_1 F ((1-r) (\overline{X}^{i},\overline{X}^{j}) + r ({X}^{i},{X}^{j}) ) \\
+ (Y^{i} - \overline{Y}^{i}) \cdot [ \partial_1 F ((1-r) (\overline{X}^{i},\overline{X}^{j}) + r ({X}^{i},{X}^{j}) ) - \partial_1 F ((1-r) (\overline{Y}^{i},\overline{Y}^{j}) + r ({Y}^{i},{Y}^{j}) ) ] \d r.
\end{multline*} 
We similarly handle the difference of the second terms.
The result follows using the assumed bounds on $D F$.
\end{proof}

\begin{lemma} \label{lem:techiid}
Let $(\overline{X}^{i},\overline{Y}^{i})_{1 \leq i \leq N}$ be an i.i.d. sequence of $E \times E$-valued variables with marginal laws $(\overline{\pi}_X,\overline{\pi}_Y)$. 
There exists $C>0$ that only depends on $\Vert F \Vert_{\mathrm{Lip}}$ such that
\[ \max_{1 \leq i \leq N} \E \, \vert F(\overline{X}^i,\pi(\vec{\overline{X}}^N)) - F(\overline{X}^i,\overline{\pi}_X) - [ F(\overline{Y}^i,\pi(\vec{\overline{Y}}^N)) - F(\overline{Y}^i,\overline{\pi}_Y) ] \vert^2 \leq C N^{-1} \E [ \vert \overline{X}^1 - \overline{Y}^1 \vert^2 ]. \]
\end{lemma}

\begin{proof}
By definition,
\begin{multline*}
F(\overline{X}^i,\pi(\vec{\overline{X}}^N)) - F(\overline{X}^i,\overline{\pi}_X) - [ F(\overline{Y}^i,\pi(\vec{\overline{Y}}^N)) - F(\overline{Y}^i,\overline{\pi}_Y) ] \\
= \frac{1}{N} \sum_{j=1}^N F(\overline{X}^i,\overline{X}^j)-F(\overline{Y}^i,\overline{Y}^j) - [ F(\overline{X}^i,\overline{\pi}_X) - F(\overline{Y}^i,\overline{\pi}_Y) ] =: \frac{1}{N} \sum_{j =1}^N F^N_{i,j}.  
\end{multline*}
Following the classical argument in \cite[Theorem 1.4]{sznitman1991topics}, we take the expectation of the square of this sum. 
By independence of the $(X^i,Y^i)$, we have for $k \neq l$,
\[ \E [ F^N_{i,k} F^N_{i,l} ] = \E  [ \E [ F^N_{i,k} F^N_{i,l} \vert (X^i,Y^i) ]  ] = \E  [ \E [ F^N_{i,k} \vert (X^i,Y^i) ] \E [ F^N_{i,l} \vert (X^i,Y^i) ] ]. \]
By definition of $(\overline{\pi}_{X},\overline{\pi}_{Y})$, $\E [ F^N_{i,k} \vert  (X^i,Y^i) ] = 0$ as soon as $k \neq i$.
Since $(k,l) \neq (i,i)$ because $k \neq l$, this implies $\E [ F^N_{i,k} F^N_{i,l} ] = 0$.
Thus, we only keep the diagonal terms:
\[ \E \bigg\vert \frac{1}{N} \sum_{j =1}^N F^N_{i,j} \bigg\vert^2 = \frac{1}{N^2} \sum_{j =1}^{N} \E [ \vert F^N_{i,j} \vert^2 ] \leq \frac{C'}{N} \big\{ \E \big[ \vert \overline{X}^1 - \overline{Y}^1 \vert^2 \big] + \big[ \E [ \vert \overline{X}^1 - \overline{Y}^1 \vert ] \big]^2 \big\}, \]
for some $C'> 0$. Using Jensen's inequality, the result follows with $C := 2 C'$.
\end{proof}

\section{Some results on rate functions} \label{app:LDP}

\subsection{Representation formulae} 

The following result is a slight variation of \cite[Theorem 5.2]{fischer2014form}.
We recall that the notion of reference system $\Sigma = (\Omega,(\F_t)_{t \leq 0 \leq T}, \P, (B_t)_{0 \leq t \leq T})$ is defined in Section \ref{subsec:not}.
For such a $\Sigma$, $h \in [0,1]$ and any square-integrable progressively measurable process $u = (u_t)_{0\leq t\leq T}$ on $\Sigma$, we say that $X^{h,u} := ( X^{h,u}_t )_{0 \leq t \leq T}$ is a solution of the McKean-Vlasov SDE
\begin{equation} \label{eq:controlled_process}
\d X^{h,u}_s = b_{s_h} ( X^{h,u}, \mathcal{L}( X^{h,u} ) ) \d s + \sigma_{s_h} (X^{h,u}, \mathcal{L}( X^{h,u}) )u_s \d s + \sigma_{s_h} (X^{h,u}, \mathcal{L}( X^{h,u}) )\d B_s,
\end{equation} 
if $X^{h,u}$ is $(\F_t)_{t \leq 0 \leq T}$-adapted and the integrated version of \eqref{eq:controlled_process} holds $\P$-a.s.
Let $W$ denote the Wiener measure on the canonical space $\Omega = \C^d$, with initial law $\L( X^1_0 )$, and let $P \in \ps_p (\C^d)$.
From \eqref{eq:defGamma}, $\Gamma_h ( P )$ is the law of the strong solution under $W$ to 
\[ \d Y^{h,P}_t = b_{t_h} (Y^{h,P},P) \d t + \sigma_{t_h} ( Y^{h,P},P ) \d \omega_t, \quad Y^{h,P}_0 = \omega_0, \]
where $( \omega_t )_{0 \leq t \leq T}$ is the canonical process. 
Indeed, strong existence and pathwise-uniqueness for this SDE are standard under \ref{ass:coef1}-\ref{ass:IniExp}. 

\begin{lemma} \label{lem:RepEntr}
Under \ref{ass:coef1}-\ref{ass:IniExp}, for $h \in [0,1]$ and $P \in \ps_p (\C^d)$, 
\[ H ( P \vert \Gamma_h ( P ) ) = \inf_\Sigma \inf_{\mathcal{L}(X^{h,u}) = P} H(\L(X^{h,u}_0) \vert \L( X^1_0 ) ) + \E \int_0^T \frac{1}{2} \vert u_t \rvert^2 \d t. \]    
where we minimise over $(X^{h,u},u)$ satisfying \eqref{eq:controlled_process} in the reference system $\Sigma$, with the convention that an infimum over an empty set equals $+\infty$.
\end{lemma}

\begin{proof}
We notice that $\Gamma_h ( P ) = Y^{h,P}_\# W$ is the push-forward of $W$ by $Y^{h,P}$.
Following the proof of \cite[Theorem 5.2]{fischer2014form}, we use the contraction property of entropy \cite[Lemma A.1]{fischer2014form} to write that
\[ H ( P \vert \Gamma_h ( P ) ) = \inf_{\substack{R \in \ps ( \C^{d'} ) \\ P = Y^{h,P}_\# R}} H( R \vert W ). \]
For $R \in \ps ( \C^{d'} )$ with finite $H( R \vert W )$, we now claim that
\begin{equation} \label{eq:suffEntr}
H( R \vert W ) = H ( R_0 \vert \L( X^1_0 )) + \inf_\Sigma \inf_{\mathcal{L}(Z^u) = R} \E \int_0^T \frac{1}{2} \vert u_t \rvert^2 \d t, 
\end{equation} 
where we minimise over processes $Z^u = ( Z^u_t )_{0 \leq t \leq T}$ satisfying $\L( Z^u ) = R$, and $\P$-a.s.
\[ \forall t \in [0,T], \quad Z^u_t = Z^u_0 + \int_0^t u_s \d s + B_t. \]
Equation \eqref{eq:suffEntr} can be obtained as a consequence of the Girsanov transform \cite[Theorem 2.3]{leonard2012girsanov}, or by adding an initial condition in \cite[Lemma B.1]{fischer2014form} ($Z^u$ is no more a non-linear process, contrary to $X^{h,u}$).
As a consequence,
\[ H ( P \vert \Gamma_h ( P ) ) = \inf_\Sigma \inf_{P = Y^{h,P}_\# R} \inf_{\mathcal{L}(Z^u) = R}  H ( R_0 \vert \L( X^1_0 )) + \inf_{\mathcal{L}(Z^u) = R} \E \int_0^T \frac{1}{2} \vert u_t \rvert^2 \d t. \]
This yields the desired result, by noticing that $Y^{h,P} ( Z^u)$ is a strong solution of \eqref{eq:controlled_process} if $P = Y^{h,P}_\# R$, and that uniqueness in law holds for the solution of 
\[ \d X^{h,u}_s = b_{s_h} ( X^{h,u}, P ) \d s + \sigma_{s_h} (X^{h,u}, P )u_s \d s + \sigma_{s_h} (X^{h,u}, P )\d B_s, \]
using the Girsanov transform as for \eqref{eq:NControlled}.
This unique law is precisely $Y^{h,P}_\# \mathcal{L}(Z^u)$.
\end{proof}

We now prove a useful tightness result.
The next result shows pre-compactness for sequences of processes of type \eqref{eq:controlled_process}, for which strong existence may not always hold.

\begin{lemma} \label{lem:Tightness} 
For $p \in [1,2)$ and every $k \geq 1$, on a reference system $\Sigma_k$, let us assume that  
\[ \d X^k_t = b_{t_{h_k}} (X^k , \L(X^k)) \d t + \sigma_{t_{h_k}} ( X^k , \L ( X^k ) ) u^k_t \d t + \sigma_{t_{h_k}} ( X^k_t, \L ( X^k ) ) \d B^k_t, \]
has a strong solution with path-law $P_k$, $( u^k_t)_{0 \leq t \leq T}$ being a control. 
We assume \ref{ass:coef1}-\ref{ass:IniExp} and
\[ \sup_{k \geq 1} H ( \L ( X^k_0 ) \vert \L ( X^1_0 ) ) + \E \int_0^T \frac{1}{2} \lvert u^k_t \vert^2 \d t < +\infty. \]
Then, for any sequence $(h_k)_{k \geq 1}$ in $[0,1]$, $( P_k )_{k \geq 1}$ is pre-compact in $\ps_p ( \C^d )$. 
\end{lemma}

\begin{proof}
As in Lemma \ref{lem:Approxh}-\ref{item:lemAppbound}, a Gronwall argument gives that
\begin{equation} \label{eq:Gronwallk}
\sup_{0 \leq s \leq T} \vert X^k_s \rvert^p \leq C \bigg[ 1 + \vert X^k_0 \vert^p + \int_0^T \vert u^k_s \vert^p \d s + \sup_{0 \leq t \leq T} \bigg\vert \int_0^t \sigma_{s_{h_k}} ( X^k, \L( X^{k} ) ) \d B^k_s \bigg\vert^p \bigg], 
\end{equation} 
for $C > 0$ independent of $k$. From the BDG inequality,
\begin{equation} \label{eq:BGK}
\E \, \sup_{0 \leq t \leq T} \bigg\vert \int_0^t \sigma_{s_{h_k}} ( X^k, \L( X^{k} )) \d B^k_s \bigg\vert^p \leq \E \bigg[ \bigg( \int_0^T \vert \sigma_{s_{h_k}} ( X^k, \L( X^{k} ) ) \vert^2 \d s \bigg)^{p/2} \bigg] \leq T^{p/2} M_\sigma. 
\end{equation} 
Taking expectations in \eqref{eq:Gronwallk}, we get the moment bound
\[ \sup_{k \geq 1} \E [ \sup_{0 \leq t \leq T} \vert X^k_t \vert^p \big] < +\infty. \]
For $0 \leq s \leq t \leq T$, 
\[ X^k_t - X^k_s = \int_s^t b_{r_{h_k}} ( X^k, \L( X^{k} ) ) \d r + \int_s^t \sigma_{r_{h_k}} ( X^k, \L( X^{k} ) ) u^k_r \d r + \int_s^t \sigma_{r_{h_k}} ( X^k, \L( X^{k} ) ) \d B^k_r. \]
We use the sub-linear growth of $b$ and the moment bound to control the first term. 
The last term can be controlled using \eqref{eq:BGK} again. 
For the middle term, we use the Cauchy-Schwarz inequality to write that
\begin{equation*} 
\E \bigg \vert \int_s^t \sigma_{r_{h_k}} ( X^k, \L( X^{k} ) ) u^k_r \d r \bigg\vert \leq (t-s)^{1/2} M_\sigma \E \bigg[ \int_0^T \vert u^k_r \rvert^2 \d r \bigg].
\end{equation*} 
At the end of the day, we deduce that
\[ \sup_{k \geq 1} \, \E \big[ \sup_{\vert s -t \vert \leq \delta} \vert X^k_t - X^k_s \vert \big] \xrightarrow[\delta \rightarrow 0]{} 0. \]
Together with the moment bound, this shows that $( P_k )_{k \geq 1}$ is tight \cite[Chapter 2,Theorem 7.3]{billingsley2013convergence}, and thus relatively compact for the weak convergence of measures.
To obtain pre-compactness in $\ps_p (\C^d)$, it is now sufficient to show uniform integrability for $( \sup_{0 \leq t \leq T} \vert X^k_t \vert^p )_{k \geq 1}$ \cite[Definition 6.8-(iii)]{villani2009optimal}.
To do so, we show that each term on the r.h.s. of \eqref{eq:Gronwallk} is uniformly integrable. 
The dual representation \eqref{eq:VarEntrop}  for relative entropy reads
\begin{equation*}
H( \L (X^{k}_0) \vert \L( X^1_0 )) = \sup_{\substack{\phi \text{ measurable} \\ \E [ \phi ( X^{1}_0 ) ] < +\infty}} \E [ \phi ( X^{k}_0 ) ] - \log \E [ e^{\phi ( X^{1}_0 )} ],  
\end{equation*}
Applying this to $\phi (x) = \alpha \vert x \vert^p \1_{\lvert \phi \rvert \geq M} $ for every $\alpha, M > 0$, uniform integrability for $( \vert X^k_0 \vert^p )_{k \geq 1}$ follows from \ref{ass:IniExp} and the bound on $( H( \L (X^{k}_0) \vert \L( X^1_0 )) )_{k \geq 1}$.
For the uniform integrability of 
\[ \int_0^T \vert u^k_s \vert^p \d s \quad \text{and} \quad \sup_{0 \leq t \leq T} \bigg\vert \int_0^t \sigma_{s_{h_k}} ( X^k, \L( X^{k} ) ) \d B^k_s \bigg\vert^p, \]
we use that $p <2$ and that
\[ \int_0^T \vert u^k_s \vert^2 \d s \quad \text{and} \quad \sup_{0 \leq t \leq T} \bigg\vert \int_0^t \sigma_{s_{h_k}} ( X^k, \L( X^{k} ) ) \d B^k_s \bigg\vert^2 \]
are bounded uniformly in $k$, using Ito's isometry and the bound on $\sigma$ for the second term. 
This concludes the proof.
\end{proof}

\begin{rem}
We notice that the same proof actually works if the coefficients $b = b^k$ and $\sigma = \sigma^k$ depend on $k$, provided the linear growth of $b^k$ and the bound on $\sigma^k$ are uniform in $k$.
\end{rem}

Combined with Lemma \ref{lem:RepEntr}, Lemma \ref{lem:Tightness} then allows us to extend the compactness result from \cite[Remark 5.2]{fischer2014form} to the topology on $\ps_p ( \C^d)$.

\begin{corollary} \label{cor:goodRate}
Under \ref{ass:coef1}-\ref{ass:IniExp}, for every $p \in [1,2)$ and $h \in [0,1]$, $P \mapsto H ( P \vert \Gamma_h ( P ) )$ has compact level sets in $\ps_p ( \C^d )$; it is a good rate function.
\end{corollary}

\begin{proof}
Let $( P_k )_{k \geq 1}$ be a sequence in $\ps_p ( \C^d )$ such that $( H ( P \vert \Gamma_h ( P ) ) )_{k \geq 1}$ is bounded.
From Lemma \ref{lem:RepEntr}, for every $k \geq 1$, there exist a reference system $\Sigma_k$, a progressively measurable process $u^k$ and a strong solution $X^{h,k}$ on $\Sigma_k$ with path-law $P_k$ for
\[ \d X^{h,k}_t = b_{t_{h}} (X^{h,k} , \L(X^{h,k})) \d t + \sigma_{t_{h}} ( X^{h,k} , \L ( X^{h,k} ) ) u^k_t \d t + \sigma_{t_{h}} ( X^{h,k}_t, \L ( X^{h,k} ) ) \d B^k_t, \]
satisfying
\[ \sup_{k \geq 1} H ( \L ( X^k_0 ) \vert \L ( X^1_0 ) ) + \E \int_0^T \frac{1}{2} \lvert u^k_t \vert^2 \d t < +\infty. \]
Lemma \ref{lem:Tightness} then shows that $(P_k)_{k \geq 1}$ is pre-compact in $\ps_p ( \C^d )$.
\end{proof}

\subsection{Gamma-convergence}

As a consequence of Lemma \ref{lem:RepEntr}, we get the following result, which can be seen as an extension of \cite[Theorem 3.6]{budhiraja2000variational} to our setting.

\begin{corollary} \label{cor:Repmf}
For every $h \in [0,1]$ and every bounded measurable $F : \ps_p ( \C^d ) \rightarrow \R$,
\[ \inf_{P \in \ps_p ( \C^d )} H( P \vert \Gamma_h ( P) ) + F(P)
= \inf_{\Sigma, u, X^u_0} F ( \L ( X^{h,u} ) ) + H(\L(X^u_0) \vert \L( X^1_0 ) ) + \E \int_0^T \frac{1}{2} \vert u_t \rvert^2 \d t. \]
\end{corollary}

\begin{proposition} \label{pro:CVexpF}
For $p \in [1,2)$ and every bounded continuous $F : \ps_p ( \C^d ) \rightarrow \R$,
\[ \inf_{P \in \ps_p ( \C^d )} H( P \vert \Gamma_h ( P) ) + F(P)
\xrightarrow[h \rightarrow 0]{} \inf_{P \in \ps_p ( \C^d )} H( P \vert \Gamma_0 ( P) ) + F(P). \]    
\end{proposition}

\begin{proof}
The lower limit is quite direct, while the upper limit requires extra care to build approximate minimisers.

\medskip

\emph{\textbf{Step 1.} Lower limit.}
For every $h \in (0,1]$, Corollary \ref{cor:goodRate} provides a minimiser $P_h$ for
\[  \overline{\F}_h := \inf_{P \in \ps_p ( \C^d )} H( P \vert \Gamma_h ( P) ) + F(P). \]
Since $\overline{\F}_h$ is bounded by $\lVert F \rVert_\infty$ independently of $H$, Lemma \ref{lem:Tightness} provides pre-compactness in $\ps_p ( \C^d )$ for $(P_h)_{h \geq 0}$. 
From any sub-sequence of $(P_h)_{h \geq 0}$, we can thus extract a sub-sequence which converges towards some $\overline{P}$ in $\ps_p ( \C^d )$.
Along this sub-sequence, it is a standard stability result in the spirit of Lemma \ref{lem:Approxh}-\ref{item:lemApphhinf} that $\Gamma_h(P_h)$ weakly converges towards $\Gamma_0 (\overline{P})$.
Since $F$ is continuous and $(P,Q) \mapsto H(P \vert Q)$ is lower semi-continuous, this shows that $\liminf_{h \rightarrow 0} \overline{\F}^h \geq \overline{\F}^0$.

\medskip

\emph{\textbf{Step 2.} Discretised minimiser.} Reciprocally, let us fix $P$ with finite $H(P \vert \Gamma_0(P))$. 
From the Girsanov transform \cite[Theorem 2.1]{leonard2012girsanov}, there exists a square-integrable process $( u_t )_{0 \leq t \leq T}$ on the canonical space $\C^d$ such that the canonical process $(X^0_t)_{0 \leq t \leq T}$ satisfies
\[ H(P \vert \Gamma_0 ( P)) = \E_P \int_0^T \frac{1}{2} \vert u_t \vert^2 \d t, \quad \d X^0_t = b_t ( X^0, P ) \d t + \sigma_t ( X^0, P ) u_t \d t + \sigma_t ( X^0, P ) \d B^P_t, \]
where $( B^P_t )_{0 \leq t \leq T}$ is a Brownian motion under $P$.
For $M >0$, let $u^M_t := u_t \1_{\mathcal{E}_t (u) \leq M}$, where $\mathcal{E}_t(u) := \int_0^t \vert u_s \vert^2 \d s$. From Lemma \ref{lem:SolMap}, for every $h \in (0,1]$, the McKean-Vlasov equation
\[ \d X^{h,M}_{t} = b_{t_h} ( X^{h,M}, \L ( X^{h,M} ) ) \d t + \sigma_{t_h} ( X^{h,M}, \L ( X^{h,M} ) ) u^M_t \d t + \sigma_{t_h} ( X^{h,M}, \L ( X^{h,M} ) ) \d B^P_t, \]
with $X^{h,M}_0 = X^0$,
has a pathwise-unique strong solution under $P$.
Lemma \ref{lem:Tightness} shows that $( \L_P ( X^{h,M} ) )_{h \geq 0}$ is pre-compact in $\ps_p ( \C^d )$. 
From any sub-sequence of $(\L_P ( X^{h,M}))_{h \geq 0}$, we can thus extract a sub-sequence that converges towards some $P^M$ in $\ps_p ( \C^d )$. Up to reindexing, let us assume that $\L_P ( X^{h,M} ) \rightarrow P^M$ as $h \rightarrow 0$. 
Using the Girsanov transform, the SDE
\[ \d X^M_t = b_t ( X^M, P^M ) \d t + \sigma_t ( X^M, P^M ) u^M_t \d t + \sigma_t ( X^M, P^M ) \d B^P_t, \quad X^M_0 = X^0 \]
has a patwhise-unique strong solution under $P$.
We now make the change of measure
\begin{equation} \label{eq:MeasChange}
\frac{\d Q}{\d P} ( X^0 ) = \exp \bigg[ -\int_0^T u^M_t \d B^P_t - \frac{1}{2} \int_0^T \vert u^M_t \vert^2 \d t \bigg].  
\end{equation}
The Girsanov theorem provides a Brownian motion $( B^Q_t )_{0 \leq t \leq T}$ under $Q$ such that
\[ \d X^{h,M}_{t} = b_{t_h} ( X^{h,M}, \L_{P} ( X^{h,M} ) ) \d t + \sigma_{t_h} ( X^{h,M}, \L_P ( X^{h,M} ) ) \d B^Q_t, \quad X^{h,M}_0 = X^0,  \]
\[ \d X^M_t = b_t ( X^M, P^M ) \d t + \sigma_t ( X^M, P^M ) \d B^Q_t, \quad X^M_0 = X^0, \]
these equations holding $Q$-.a.s.
As previously, a coupling argument yields the stability result
$\E_Q [ \sup_{0 \leq t \leq T} \vert X^{h,M}_{t} - X^M_t \vert^p ] \rightarrow 0$ as $h \rightarrow 0$.
By \eqref{eq:MeasChange} and dominated convergence, we deduce that $\L_P ( X^{h,M} )$ weakly converges towards $\L_P ( X^{M} )$, proving that $P^M = \L_P ( X^{M} )$.
Thus, $X^M$ is the strong solution of a McKean-Vlasov equation under $P$.

\medskip

\emph{\textbf{Step 3.} Coupling in $L^2$.}
Subtracting and integrating, \ref{ass:coef1} and the moment bound from Lemma \ref{lem:Approxh}-\ref{item:lemAppbound} yield
\[ \sup_{0 \leq s \leq t} \vert X^{0}_s - X^M_s \vert \leq C + C \int_0^t \sup_{0 \leq r \leq s} \vert X^{0}_r - X^M_r \vert + \vert u^M_s \vert \d s + C \sup_{0 \leq s \leq t} \bigg\vert \int_0^s \sigma^{M}_r \d B^P_r \bigg\vert, \]
for some $\sigma^{M}$ with $\vert \sigma^{M}_r \vert \leq 2 M_\sigma$ and some $C >0$ independent of $h$ that may change from line to line.
Since this holds for every $t \in [0,T]$, Gronwall's lemma yields
\[ \sup_{0 \leq t \leq T} \vert X^{0}_t - X^M_t \vert \leq C + C \int_0^T \vert u^M_t \vert \d t + C \sup_{0 \leq t \leq T} \bigg\vert \int_0^t \sigma^{M}_s \d B^P_s \bigg\vert, \]
Taking the square and using the BDG inequality, this shows that $\E [ \sup_{0 \leq t \leq T} \vert X^{h,M}_t - X^M_t \vert^2 ]$ is finite.
Since $p \in [1,2)$, the Jensen inequality yields
\[ ( W_p ( \L(X^{0}), \L(X^{M} ) ) )^2 \leq C \E \big[ \sup_{0 \leq t \leq T} \vert X^{h,M}_t - X^M_t \vert^2 \big], \]
and we can now use coupling arguments in $L^2$.
At this stage, we use the definition of $u^M$ and the proof of \cite[Proposition C.1]{fischer2014form} to obtain that $\E_P [ \sup_{0 \leq t \leq T} \vert X^{0}_{t \wedge \tau_M} - X^M_{t \wedge \tau_M} \vert^2 ] = 0$, 
where $\tau_M := \inf \{ t \in [0,T] , \, \mathcal{E}_t (u) \geq M \}$.
From Lemma \ref{lem:Tightness}, $(\L_P(X^M))_{ M >0}$ is pre-compact in $\ps_p (\C^d)$.
Sending $M \rightarrow +\infty$, we eventually deduce that $P^M = \L_P (X^M)$ converges towards $P = \L_P ( X^0 )$.

\medskip

\emph{\textbf{Step 4.} Upper limit.} Using Corollary \ref{cor:Repmf},
\[ \overline{\F}_h \leq F ( \L ( X^{h,M} ) ) + \E_P \int_0^T \vert u^M_t \vert^2 \d t. \]
Taking the $h \rightarrow 0$ limit, 
\[ \limsup_{h \rightarrow 0} \overline{\F}_h \leq F ( P^{M} ) + \E_P \int_0^T \vert u^M_t \vert^2 \d t. \]
Sending $M \rightarrow +\infty$
\[ \limsup_{h \rightarrow 0} \overline{\F}_h \leq F ( P ) + \E_P \int_0^T \vert u_t \vert^2 \d t = F(P) + H(P \vert \Gamma_0 ( P)). \]
Since this holds for every $P$, $\limsup_{h \rightarrow 0} \overline{\F}^h \leq \overline{\F}^0$, completing the proof. 
\end{proof}

\printbibliography
\addcontentsline{toc}{section}{References}

\end{document}